\documentclass[11pt]{amsart}

\usepackage{amsfonts,epsfig}
\usepackage{latexsym}
\usepackage{amssymb}
\usepackage{amsmath}
\usepackage{amsthm}
\usepackage{graphics}
\usepackage[all]{xy}
\usepackage[T2A]{fontenc}
\usepackage{multirow}
\usepackage{hhline}
\usepackage{array, booktabs, ctable}
\usepackage{hyperref}
\usepackage{enumerate}
\newcommand{\eclabel}[1]{\href{https://www.lmfdb.org/EllipticCurve/Q/#1}{\texttt{#1}}}
\newcommand{\ecnflabel}[2]{\href{https://www.lmfdb.org/EllipticCurve/#1-#2}{\texttt{#1-#2}}}

\usepackage{longtable}

\newtheorem{defn}{Definition}[section]

\newtheorem{corollary}[defn]{Corollary}
\newtheorem{lemma}[defn]{Lemma}

\newtheorem{thm}[defn]{Theorem}
\newtheorem{theorem}[defn]{Theorem}

\newtheorem{prop}[defn]{Proposition}
\newtheorem{proposition}[defn]{Proposition}

\theoremstyle{definition}

\newtheorem*{ack}{Acknowledgements}
\newtheorem{remark}[defn]{Remark}

\newtheorem{example}[defn]{Example}

\newenvironment{romanenum}{\hfill \begin{enumerate} }{\end{enumerate}}

\newcommand{\Q}{\mathbb Q}
\newcommand{\QQ}{\mathbb Q}
\newcommand{\Z}{\mathbb Z}

\newcommand{\F}{\mathbb F}
\newcommand{\FF}{\mathbb F}

\newcommand{\Qbar}{\overline{\mathbb{Q}}}
\newcommand\legendre[2]{{#1\overwithdelims () #2}}

\newcommand{\OK}{\mathcal{O}_K}
\newcommand{\OF}{\mathcal{O}_F}
\newcommand{\Of}{\mathcal{O}_{K,f}}
\newcommand{\cC}{\mathcal{C}}
\newcommand{\OO}{\mathcal{O}}

\newcommand{\Gal}{\operatorname{Gal}}
\newcommand{\Aut}{\operatorname{Aut}}

\newcommand{\GL}{\operatorname{GL}}

\newcommand{\id}{\operatorname{Id}}

\begin{document}
%\tableofcontents
% Title, authors and addresses

% use the thanksref command within \title, \author or \address for footnotes;
% use the corauthref command within \author for corresponding author footnotes;
% use the ead command for the email address,
% and the form \ead[url] for the home page:
% \title{Title\thanksref{label1}}
% \thanks[label1]{}
% \author{Name\corauthref{cor1}\thanksref{label2}}
% \ead{email address}
% \ead[url]{home page}
% \thanks[label2]{}
% \corauth[cor1]{}
% \address{Address\thanksref{label3}}
% \thanks[label3]{}

\bibliographystyle{plain}
\title[Models of CM elliptic curves]{Models of CM elliptic curves with a prescribed $\ell$-adic Galois image}

\begin{abstract}
    For each prime number $\ell$ and for each imaginary quadratic order of class number one or two, we determine all the possible $\ell$-adic Galois representations that occur for any elliptic curve with complex multiplication by such an order over its minimal field of definition, and then we determine all the isomorphism classes of elliptic curves that have a prescribed $\ell$-adic Galois representation.
\end{abstract}

\author{Enrique Gonz\'alez--Jim\'enez}
\address{Universidad Aut{\'o}noma de Madrid, Departamento de Matem{\'a}ticas, Madrid, Spain}
\email{enrique.gonzalez.jimenez@uam.es}
\urladdr{http://www.uam.es/enrique.gonzalez.jimenez}
\author{\'Alvaro Lozano-Robledo}
\address{University of Connecticut, Department of Mathematics, Storrs, CT 06269, USA}
\email{alvaro.lozano-robledo@uconn.edu} 
\urladdr{http://alozano.clas.uconn.edu}

\author{Benjamin York}
\address{University of Connecticut, Department of Mathematics, Storrs, CT 06269, USA}
\email{benjamin.york@uconn.edu} 
\urladdr{https://benjamin-york.github.io/}

\thanks{The first author is supported by Grant PID2022-138916NB-I00 funded by MCIN/AEI/ 10.13039/501100011033 and by ERDF A way of making Europe.}

\subjclass{Primary: 11G05, Secondary: 11G15, 14H52, 14K22.}

\maketitle

\section{Introduction}

Let $F$ be a number field, let $E/F$ be an elliptic curve, and let $\ell$ be a prime number. Let $T_\ell(E)=\varprojlim E[\ell^n]$ be the $\ell$-adic Tate module of $E/F$. Then, $T_\ell(E)$ is a $\Gal(\overline{F}/F)$-module, and the associated Galois representation
$$\rho_{E,\ell^{\infty}}\colon \Gal(\overline{F}/F)\to \Aut(T_\ell(E))$$
is an object of great interest in modern algebraic number theory (see, for example, \cite{elladic} and \cite{zywina}). In particular,
there is a vast amount of literature on the problem of classifying the possible images (up to conjugation as a subgroup of $\GL(2,\Z_\ell)$) of such Galois representations, which is usually referred to as Mazur's ``Program B'' after \cite{mazurprogramb}. 

When $E$ is an elliptic curve with complex multiplication by an order $\mathcal{O}_{K,f}$ of an imaginary quadratic field $K$, and $F=\Q(j(E))$ is the minimal field of definition of an elliptic curve with $j$-invariant $j(E)$, the second author completed a classification of $\ell$-adic Galois representations attached to $E/F$ in \cite{lozano-galoiscm}. However, an inverse problem remains: given a possible CM image $G_{E,\ell^\infty}\subseteq \GL(2,\Z_\ell)$, what are the isomorphism classes of elliptic curves with precisely $G_{E,\ell^\infty}$ for their $\ell$-adic Galois image (up to conjugation)?

In this paper we accomplish the following, for each order $\mathcal{O}=\mathcal{O}_{K,f}$ of class number $1$ or $2$:
\begin{enumerate}
    \item We give a Weierstrass model for an elliptic curve $E=E_{\mathcal{O}}$ defined over $\Q(j_{K,f})$, such that $E$ has CM by $\mathcal{O}$, which is minimal in a certain sense (see Section \ref{sec-choiceofmodels}). The models are provided in Table \ref{tab-modelsclassnumber2}.
    \item For each prime $\ell$, we give a description of all the possible $\ell$-adic images of Galois attached to an elliptic curve $E$ with CM by an order $\mathcal{O}$ as above and, for each possible image $G\subseteq \GL(2,\Z_\ell)$, we give a precise description of the twists of $E_{\mathcal{O}}$ that have $G$ as their $\ell$-adic Galois image, up to conjugation. The possibilities are recorded in Tables \ref{cm2big}, \ref{tab-ellnotdividing}, 
 and \ref{tab-cmodddividing}.
    \item As a by-product of the proofs, we provide explicit descriptions of the $2$nd, $4$th, and $8$th division fields of elliptic curves with $j=1728$, $8000$, and $287496$ (see Cor. \ref{cor-8divfieldj1728}, Prop. \ref{prop-8divfieldm8} and Prop. \ref{prop-8divfieldm16}), and the $3$rd and $9$th division fields of curves with $j=0$ (see Lemma \ref{lem-3and9divfielddegreej0}).
\end{enumerate}

In order to clarify our results and their usage, we provide three examples.

\begin{example}
For instance, let $K=\Q(\sqrt{-3}$) and $\mathcal{O}=\Z[(1+\sqrt{-3})/2]$, with discriminant $\Delta(\mathcal{O})=-3$, so that $\Delta_K=-3$ and $\mathcal{O}$ is an order of $K$ of conductor $f=1$, and class number $1$. By Table \ref{tab-modelsclassnumber2}, the chosen model is $E_\mathcal{O}: y^2=x^3+16$, and $j_{K,f}=j(E)=0$. Let $\ell$ be a prime. 
\begin{itemize}
    \item If $\ell=2$, then Table \ref{cm2big} shows that there are precisely two distinct possibilities for the $2$-adic image of an elliptic curve $E/\Q$ with CM by $\mathcal{O}$. Namely, if $E$ is isomorphic to a model of the form $y^2=x^3+16t$, with $t\in \Q^\ast$ but $t\not\in 4(\Q^\ast)^3$, then the image is as large as possible given the constraints imposed by the order, namely the image is conjugate to the normalizer $\mathcal{N}_{-1,1}(2^\infty)$, which is defined in Section \ref{sec-preliminaries}. Otherwise, if $E$ is isomorphic to a model of the form $y^2=x^3+64t^3$ for any $t\in \Q^\ast$, then the image is of index $3$ in the normalizer and it is of the form $\langle \mathcal{C}_{-1,1}(2^\infty)^3,c_1'\rangle$, which is again a group explicitly described in Section \ref{sec-preliminaries}.

 \item If $\ell=3$, then Table \ref{cmodddividing} shows that there are $12$ possibilities for the $3$-adic image of an elliptic curve $E$ with $j(E)=0$, and the table lists the twists of $E_\mathcal{O}$ that afford each of the possible Galois images.

\item Finally, if $\ell>3$, then Table \ref{cmoddnotdividing} shows that there are an additional $6$ possible $\ell$-adic Galois images, according to the class of $\ell \bmod 9$, and the table lists the conditions on the twists of $E_\mathcal{O}$ that afford each possibility.

   \end{itemize} 
\end{example}

\begin{example}
    Now let $K=\Q(i)$ and let $\mathcal{O}=\Z[4i]$, with discriminant $\Delta(\mathcal{O})=-4\cdot 4^2=-64$, so that $\Delta_K=-4$ and $\mathcal{O}$ is an order of $K$ of conductor $f=4$, and class number $2$. By Table \ref{tab-modelsclassnumber2}, the model $E_\mathcal{O}$ is given by
    $$y^2 + axy = x^3 + x^2 + (15a - 22)x + 46a - 69$$
    where $a$ is a root of $x^2-2=0$, i.e., $a=\sqrt{2}$ (note that if we replace $a$ by a conjugate, then we get another elliptic curve with the same CM properties, see Remark \ref{rem-conjugates}). Moreover, $j_{K,f}=j(E_\mathcal{O})= 41113158120-29071392966 \sqrt{2}$. Let $\ell$ be a prime.
    \begin{itemize}
        \item If $\ell=2$, then Theorem \ref{thm-m8and16alvaro} shows that there are at most $9$ possibilities for the $2$-adic image, and  Table \ref{cm2big} shows that all such $9$ possibilities do occur, for the $2$-adic image of an elliptic curve $E/\Q(\sqrt{2})$ with CM by $\mathcal{O}$, and the table provides the twists of $E_\mathcal{O}$ which realize each kind of image.
        \item If $\ell>2$, then Theorem \ref{thm-goodredn} shows that there is only one possibility for the image, namely the full image $\mathcal{N}_{-16,0}(\ell^\infty)$.
    \end{itemize}
\end{example}

\begin{example}
    Now let $K=\Q(\sqrt{-22})$ and let $\mathcal{O}=\Z[\sqrt{-22}]$, with discriminant $\Delta(\mathcal{O})=-88$, so that $\Delta_K=-88$ and $\mathcal{O}$ is an order of $K$ of conductor $f=1$, and class number $2$. By Table \ref{tab-modelsclassnumber2}, the model $E_\mathcal{O}$ is given by
    $$y^2 + axy + y = x^3 + x^2 + (-193a - 453)x + 2233a + 4008$$
    where $a$ is a root of $x^2-2=0$, i.e., $a=\sqrt{2}$. Moreover, $j_{K,f}=j(E_\mathcal{O})=  3147421320000+2225561184000\sqrt{2}$. Let $\ell$ be a prime.
    \begin{itemize}
        \item If $\ell=2$, and since $\Delta(\mathcal{O})\not\equiv 0 \bmod 16$, Theorem \ref{thm-m8and16alvaro} shows that there are at most $5$ possibilities for the $2$-adic image of an elliptic curve $E/\Q(\sqrt{2})$ with CM by $\mathcal{O}$. However, Table \ref{cm2big} shows that only one possibility does occur, namely the maximal normalizer image $\mathcal{N}_{-22,0}(2^\infty)$.
        \item If $\ell=11$, then Theorem \ref{thm-oddprimedividingdisc} shows that there are at most three possibilities for the $\ell$-adic image, and Table \ref{cmodddividing} shows that all three possibilities do occur, and provides twists of $E_\mathcal{O}$ that afford each possibility.  
        \item If $\ell\ne 2,11$, then Theorem \ref{thm-goodredn} shows that the image must be the maximal normalizer $\mathcal{N}_{-22,0}(\ell^\infty).$
    \end{itemize}
\end{example}

\subsection{Structure of the article}\label{sec-structure} In Section \ref{sec-preliminaries} we provide a number of previous results that have appeared in the literature (e.g., \cite{lozano-galoiscm}, \cite{yelton}, \cite{zywina}) that we heavily use in our proofs, and we also describe the labels we shall use to denote the different types of $\ell$-adic images.  The classification of Weierstrass models of CM elliptic curves with a prescribed $\ell$-adic image is completed in Sections \ref{sec-classnumber1} and \ref{sec-classnumber2}. The tables containing the details about the models and the images appear at the end of this article, in Section \ref{sec-tables} and Tables \ref{tab-modelsclassnumber2}, \ref{tab-ell2}, \ref{tab-ellnotdividing}, and \ref{tab-cmodddividing}.

In this work, we will only consider elliptic curves with CM by an order $\Of$ of class number $1$ or $2$ (new challenges appear when the class number is greater or equal to $3$, see Remark \ref{rem-classnumber3}). Thus, the absolute value of the discriminant $\Delta(\Of)=\Delta_K f^2$ belongs to one of the sets $\Delta_k$ for $k=1$ or $2$ below, where $h(\Of)=k$.
\begin{align*} \Delta_1 &= \{ 3, 4, 7, 8, 11, 12, 16, 19, 27, 28, 43, 67, 163\},\\
\Delta_2 &= \{15, 20, 24, 32, 35, 36, 40, 48, 51, 52, 60, 64, 72, 75, 88, 91, 99, 100, 112, 115, 123, 147,\phantom{\}}\\
& \phantom{= \{} 148, 187, 232, 235, 267, 403, 427\}.
\end{align*}
We will distinguish several cases according first to the class number of $\Of$, and then according to the value of $\Delta_K f^2$ and the prime $\ell$. The following is a summary of the cases treated in each section and subsection of the article, according to $\Delta_K f^2$ and $\ell$:

\begin{enumerate}
    \item[(\ref{sec-classnumber1})] $|\Delta_K f^2| \in \Delta_1$,
\begin{enumerate}
    \item[(\ref{sec-notminus3ellmorethan2})] $\Delta_K f^2 \neq -3$ and $\ell\geq 3$,
    \begin{itemize}
        \item[(\ref{sec-ellnotdividedisc})] $\ell$ does not divide $\Delta_K f$,
        \item[(\ref{sec-elldividesdisc})] $\ell$ divides $\Delta_K f$.
    \end{itemize}
    \item[(\ref{sec-notminus3ellequals2})] $\Delta_K f^2 \neq -3$ and $\ell=2$,
    \begin{itemize}
        \item[(\ref{sec-not34816ell2})] $\Delta_K f^2 \neq -4,-8,-16$ and $\ell =2$,
        \item[(\ref{sec-4ell2})] $\Delta_K f^2 = -4$ and $\ell=2$,
        \item[(\ref{sec-8ell2})] $\Delta_K f^2 = -8$ and $\ell=2$,
        \item[(\ref{sec-16ell2})] $\Delta_K f^2 = -16$ and $\ell=2$.
    \end{itemize}
    \item[(\ref{sec-delta3})] $\Delta_K f^2 = -3$,
    \begin{itemize}
        \item[(\ref{sec-delta3ell2})] $\Delta_K f^2 = -3$ and $\ell=2$,
        \item[(\ref{sec-delta3ell3})] $\Delta_K f^2 = -3$ and $\ell=3$,
        \item[(\ref{sec-delta3ellmorethan3})] $\Delta_K f^2 = -3$ and $\ell>3$.
    \end{itemize}
\end{enumerate}
 \item[(\ref{sec-classnumber2})] $|\Delta_K f^2| \in \Delta_2$.
\end{enumerate}

\subsection{Code}

This work makes extensive use of the computer algebra system \verb|Magma| \cite{magma}. Code verifying the computational claims made in the paper can be found at GitHub repository \cite{githubrepo}
\begin{center}
\url{https://github.com/benjamin-york/CM-galois-images}
\end{center}

\begin{ack}
    We would like to thank Drew Sutherland for his help in devising the labeling convention of CM images. We would also like to thank the referees for their many detailed comments that have helped improve the paper.
\end{ack}

\section{Preliminaries}\label{sec-preliminaries}

In this section, we will set the notation that we will follow in the rest of the paper, define the matrix groups that are the maximal possible images in the CM case, and recall results from previous articles that we will use in later section. 

Let $K$ be an imaginary quadratic field, and let $\OK$ be the ring of integers of $K$ with discriminant $\Delta_K$. Let $f\geq 1$ be an integer and let $\OO_{K,f}$ be the order of $K$ of conductor $f$. Let $j_{K,f}$ be a CM $j$-invariant corresponding to the order $\OO_{K,f}$ (thus, every CM $j$-invariant with CM by $\OO_{K,f}$ is a Galois conjugate of $j_{K,f}$) and let $N\geq 2$. We define associated constants $\delta$ and $\phi$ as follows:

\begin{itemize}
    \item If $\Delta_Kf^2\equiv 0\bmod 4$ or $N$ is odd, let $\delta=\Delta_K f^2/4$, and $\phi=0$.
	\item If $\Delta_Kf^2\equiv 1 \bmod 4$ and $N$ is even, let $\delta=\frac{(\Delta_K-1)}{4}f^2$, let $\phi=f$.
\end{itemize}

We define matrices in $\GL(2,\Z/N\Z)$ by
$$c_\varepsilon=\left(\begin{array}{cc} \varepsilon & 0\\ \phi & -\varepsilon\\\end{array}\right),\quad c'_\varepsilon = \left(\begin{array}{cc} 0 & \varepsilon \\ \varepsilon & 0 \\ \end{array}\right), \quad \text{ and } \quad c_{\delta,\phi}(a,b)=\left(\begin{array}{cc}a+b\phi & b\\ \delta b & a\\ \end{array}\right)$$ where $\varepsilon \in \{\pm 1\}$ and $a,b \in \Z/N\Z$ such that $\det(c_{\delta,\phi}(a,b))\in (\Z/N\Z)^\times$. 
We define the Cartan subgroup $\cC_{\delta,\phi}(N)$ of $\GL(2,\Z/N\Z)$ by
$$\cC_{\delta,\phi}(N)=\left\{c_{\delta,\phi}(a,b): a,b\in\Z/N\Z,\  \det(c_{\delta,\phi}(a,b)) \in (\Z/N\Z)^\times \right\},$$
and $\mathcal{N}_{\delta,\phi}(N) = \left\langle \cC_{\delta,\phi}(N),c_1\right\rangle$. We will sometimes call $\mathcal{N}_{\delta,\phi}(N)$ the ``normalizer'' of $\cC_{\delta,\phi}(N)$ in $\GL(2,\Z/N\Z)$. In the case of a prime number $N$, then $\mathcal{N}_{\delta,\phi}(N)$ is the true group-theoretic normalizer, but this is often not the case for composite values of $N$ (see \cite[Section 5]{lozano-galoiscm}). Finally, we write $\mathcal{N}_{\delta,\phi} = \varprojlim \mathcal{N}_{\delta,\phi}(N)$ and regard it as a subgroup of $\GL(2,\widehat{\Z})$. If $\ell$ is a prime, we write  $\cC_{\delta,\phi}(\ell^{\infty}) = \varprojlim \cC_{\delta,\phi}(\ell^n)$ and $\mathcal{N}_{\delta,\phi}(\ell^{\infty}) = \varprojlim \mathcal{N}_{\delta,\phi}(\ell^n)$.

Let $E/\Q(j_{K,f})$ be an elliptic curve with CM by $\OO_{K,f}$, let $N\geq 3$, and let $\rho_{E,N}$ be the Galois representation $\Gal(\overline{\Q(j_{K,f})}/\Q(j_{K,f})) \to \Aut(E[N])\cong \GL(2,\Z/N\Z)$. Similarly, for a prime $\ell$, we will write $\rho_{E,\ell^\infty}$ for the Galois representation attached to the action of Galois on the $\ell$-adic Tate module $T_\ell(E)$. We will denote the $\ell$-adic image by $G_{E,\ell^\infty} = \rho_{E,\ell^{\infty}}(G_{\Q(j_{K,f})})$, with an appropriate choice of basis such that $G_{E,\ell^\infty} \subseteq \mathcal{N}_{\delta,\phi}(\ell^{\infty})$, where $G_{\Q(j_{K,f})}=\Gal(\overline{\Q(j_{K,f})}/\Q(j_{K,f}))$ is the absolute Galois group of the field of definition $\Q(j_{K,f})$ (see Theorem \ref{thm-firstofsec1}, cited from \cite{lozano-galoiscm}, which shows that such a choice of basis exists).

\subsection{Labels} \label{sec-labels} In Theorems \ref{thm-firstofsec1} through \ref{thm-lastofsec1} below, we describe certain explicit subgroups of $\GL(2,\Z_\ell)$ that correspond to the possible images (up to conjugation) of $\rho_{E,\ell^{\infty}}$. We use the following conventions:
$$G_{\delta,\phi}^{i,t}=G_{\delta,\phi}^{i,t}(\ell^\infty)$$
is a subgroup of the Cartan $\mathcal{C}_{\delta,\phi}(\ell^{\infty})$ of index $i$, that is, $[\mathcal{C}_{\delta,\phi}(\ell^{\infty}) : G_{\delta,\phi}^{i,t}]=i$. When there are several possible subgroups of $\mathcal{C}_{\delta,\phi}(\ell^{\infty})$ with index $i$, we use a parameter $t \in \{1,2,3,4\}$ to differentiate between subgroups (the choice of numbering, though, was arbitrary).

Alternatively, in the tables, we will refer to the subgroups of $\mathcal{N}_{\delta, \phi}(\ell^{\infty})$ generated by the subgroups $G_{\delta,\phi}^{i,t} \subseteq \mathcal{C}_{\delta, \phi}(\ell^{\infty})$ and by the matrices $c_{\varepsilon}$ or $c'_{\varepsilon}$ using a labeling system similar to the labels that are used in the LMFDB (\cite{lmfdb}) for subgroups of $\GL(2,\Z_\ell)$. The subgroups will be denoted by a label of the form
$$\ell\texttt{.}\nu\texttt{.c-n.i.t}$$
where
\begin{itemize}
    \item $\ell$ is the prime we are considering,
    \item $\nu$ is the $\ell$-adic valuation of $\Delta(\mathcal{O})=\Delta_K\cdot f^2$,
    \item \texttt{c} is the square class of $u=\Delta_K\cdot f^2\cdot \ell^{-\nu}$ in $\Z_\ell^\times/(\Z_\ell^\times)^2$. If $\ell>2$, then \texttt{c} is \texttt{s} or \texttt{ns} according to whether $u$ is a square or a non-square, respectively. If $\ell=2$, then $\texttt{c}$ is \texttt{s}, \texttt{ns3}, \texttt{ns5}, or \texttt{ns7} according to whether $u\equiv 1,3,5$ or $7\bmod 8$, respectively,
    \item \texttt{n} is the level of definition of the subgroup. The level of definition of a subgroup $G\subseteq \mathcal{N}_{\delta,\phi}(\ell^n)$ is $n\geq 1$ is smallest with the property that $\pi_n^{-1}(\pi_n(G))=G$, where $\pi_n\colon \mathcal{N}_{\delta,\phi}(\ell^{\infty}) \to \mathcal{N}_{\delta,\phi}(\ell^n)$,
    \item \texttt{i} is the index of the subgroup in $\mathcal{N}_{\delta,\phi}(\ell^{\infty})$, and
    \item \texttt{t} is a tie-breaker among subgroups of $\mathcal{N}_{\delta,\phi}(\ell^{\infty})$ that share $\ell$, $\nu$, \texttt{c}, \texttt{n}, and \texttt{i}. The tie-breaker is defined in \cite[Section 2.4]{elladic}.
\end{itemize} 
We would like to thank Drew Sutherland for his help in devising the above labeling method.

\subsection{Weierstrass models of elliptic curves with CM by orders of class number 1 and 2} \label{sec-choiceofmodels}

In Table \ref{tab-modelsclassnumber2}, for each order $\Of$ of class number $1$ or $2$, we have given an example of an elliptic curve $E_{\Of}$ defined over $\Q(j_{K,f})$ with CM by $\Of$. We chose these models as follows. 
\begin{itemize} 
\item For orders of class number $1$, our choice of $E_{\Of}$ coincides with Zywina's choice in \cite{zywina}. In other words, $E_{\Of}$ is equal to Zywina's $E_{D,f}$, where $D=-\Delta_K$. 

\item For orders $\Of$ of class number $2$, we proceeded as follows. 
    \begin{itemize}
        \item  If the LMFDB contained data on elliptic curves with CM by $\Of$ defined over $\Q(j_{K,f})$, our choice of curve $E_{\Of}$ coincided with the first curve listed in the LMFDB with CM by $\Of$, according to the LMFDB's labeling system. 
        \item  If data was not available in the LMFDB, we first computed $j_{K,f}$ and an arbitrary model $E'/\Q(j_{K,f})$ with $j(E')=j_{K,f}$, using the usual formula (see \cite[Ch. III, Prop. 1.4]{silverman1}). We then found a list of twists of $E'$ with smallest conductor-norm, and whose minimal model had smallest discriminant-norm. If this list did not contain a unique elliptic curve, we choose $E_{\Of}$ to be a curve in the list whose minimal model could be written with the fewest characters.
        \item When the class number of $\Of$ is $2$, there are two $j$-invariants with CM by $\Of$, namely $j_{K,f}$ and its non-trivial Galois conjugate. However, we only list one of them in the tables: see Remark \ref{rem-conjugates}.
    \end{itemize}
\end{itemize}

%\subsection{Zywina's classification of mod-$\ell$ Galois representation of elliptic curve with CM over $\QQ$}
\label{sec-zywinasresults}
\subsection{Classification of mod-$\ell$ Galois representation of elliptic curve with CM over $\QQ$}
\label{sec-zywinasresults}
In this section we summarize Zywina's results from \cite{zywina} as it pertains to elliptic curves with CM defined over $\Q$. We have edited the statements using our notation, otherwise the results are just as they appear in \cite{zywina}. Up to isomorphism over $\Qbar$, there are thirteen $\Qbar$-isomorphic classes of elliptic curves with complex multiplication that are defined over $\QQ$, one for each order $\Of$ of class number one.  In Table \ref{tab-modelsclassnumber2} below, the reader can find an elliptic curve $E_{\Of}/\QQ$ for each of these classes. 

Let  $\ell$ be an odd prime. Let  $C_s(\ell)$ (resp. $C_{ns}(\ell)$) be the usual split (resp. non-split) Cartan subgroup in $\GL_2(\FF_\ell)$. Let $N_s(\ell)$ and $N_{ns}(\ell)$ be the normalizers of $C_{s}(\ell)$ and $C_{ns}(\ell)$, respectively, in $\GL_2(\FF_\ell)$.   Let $B(\ell)$ be the Borel (upper triangular) subgroup in $\GL_2(\FF_\ell)$.

We first describe the group $\rho_{E,\ell}(G_\Q)$ up to conjugacy when $E$ is a CM elliptic curve with non-zero $j$-invariant and $\ell$ odd, where $G_\Q=\Gal(\overline{\Q}/\Q)$ and we note that in this case $j_{K,f}\in \Q$ so that $\Q(j_{K,f})=\Q$.

\begin{prop}[\cite{zywina}, Prop. 1.14]  \label{prop-zywina1}
Let $E$ be a CM elliptic curve defined over $\QQ$ with $j_E\neq 0$.  The ring of endomorphisms of $E/{\Qbar}$ is an order $\Of$ of conductor $f$ in the ring of integers of an imaginary quadratic field of discriminant $\Delta_K$. Take any odd prime $\ell$.
\begin{romanenum}
\item \label{P:CM main a}
If $\legendre{\Delta_K}{\ell}=1$, then $\rho_{E,\ell}(G_\Q)$ is conjugate in $\GL_2(\FF_\ell)$ to $N_s(\ell)$.
\item \label{P:CM main b}
If $\legendre{\Delta_K}{\ell}=-1$, then $\rho_{E,\ell}(G_\Q)$ is conjugate in $\GL_2(\FF_\ell)$ to $N_{ns}(\ell)$.

\item \label{P:CM main c}
Suppose that $\ell$ divides $\Delta_K$ and hence $\Delta_K = -\ell$.  Define the groups 
\[
G=\{ \left(\begin{smallmatrix} a & b \\0 & \pm a \end{smallmatrix}\right) : a\in \FF_\ell^\times, b\in \FF_\ell\},
\]
\[
H_1 = \{ \left(\begin{smallmatrix} a & b \\0 & \pm a \end{smallmatrix}\right) : a\in (\FF_\ell^\times)^2, b\in \FF_\ell\}, \quad \text{ and }\quad H_2 = \{ \left(\begin{smallmatrix} \pm a & b \\0 &  a \end{smallmatrix}\right) : a\in (\FF_\ell^\times)^2, b\in \FF_\ell\} 
\]

\noindent 
If $E$ is isomorphic to $E_{\Of}$, then $\rho_{E,\ell}(G_\Q)$ is conjugate in $\GL_2(\FF_\ell)$ to $H_1$.

\noindent
If $E$ is isomorphic to the quadratic twist of $E_{\Of}$ by $-\ell$, then $\rho_{E,\ell}(G_\Q)$ is conjugate in $\GL_2(\FF_\ell)$ to $H_2$.

\noindent
If $E$ is not isomorphic to $E_{\Of}$ or its quadratic twist by $-\ell$, then $\rho_{E,\ell}(G_\Q)$ is conjugate in $\GL_2(\FF_\ell)$ to $G$.

\end{romanenum}
\end{prop}

The following deals with the prime $\ell=2$, which was excluded in the previous proposition.

\begin{prop}[\cite{zywina}, Prop. 1.15] \label{prop-zywina2}
Let $E/\QQ$ be a CM elliptic curve and let $j_E=j_{K,f}$.  Define the subgroup $G_2=\{ I, \left(\begin{smallmatrix}1 & 1 \\0 & 1 \end{smallmatrix}\right)\}$ of $\GL_2(\FF_2)$.
\begin{romanenum}
\item \label{P: prime 2 i}
If  $j_E \in \{2^4 3^3 5^3,\, 2^3 3^3 11^3,\, -3^3 5^3,\, 3^3 5^3 17^3,\, 2^6 5^3\},$
then $\rho_{E,2}(G_\Q)$ is conjugate to $G_2$.
\item
If $j_E \in \{ -2^{15} 3 \cdot 5^3,\, -2^{15},\, -2^{15} 3^3,\, -2^{18} 3^3 5^3,\, -2^{15} 3^3 5^3 11^3,\, -2^{18} 3^3 5^3 23^3 29^3\},$
then $\rho_{E,2}(G_\Q)=\GL_2(\FF_2)$.

\item
Suppose that $j_E=1728$.  The curve can be given by a Weierstrass equation $y^2=x^3-dx$ for some $d\in \QQ^\times$.

\noindent
If $d$ is a square, then $\rho_{E,2}(G_\Q)=\{I\}$.  

\noindent 
If $d$ is not a square, then the group $\rho_{E,2}(G_\Q)$ is conjugate to $G_2$.  

\item 
Suppose that $j_E=0$.  The curve $E$ can be given by a Weierstrass equation $y^2=x^3+d$ for some $d\in \QQ^\times$. 

\noindent 
If $d$ is a cube, then $\rho_{E,2}(G_\Q)$ is conjugate in $\GL_2(\FF_2)$ to the group $G_2$.

\noindent 
If $d$ is not a cube, then $\rho_{E,2}(G_\Q)=\GL_2(\FF_2)$. 
\end{romanenum}
\end{prop}

It remains to consider the situation where $\ell$ is an odd prime and $E/\QQ$ is an elliptic curve with $j_E=0$.   That such curves have cubic twists make the classification more involved. %Define the following subgroups of $\GL_2(\FF_3)$:
%\begin{itemize}
%\item 
%Let $G_1$ be the group $C_s(3)$.  
%\item
%Let $G_2$ be the group $N_s(3)$. 
%\item 
%Let $G_3$ be the group $B(3)$.    
%\item
%Let $G_4$ be the group $N_{ns}(3)$.  
%\item
%Let $H_{1,1}$ be the subgroup consisting of the matrices of the form $\left(\begin{smallmatrix}1 & 0 \\0 & * \end{smallmatrix}\right)$.   
%\item
%Let $H_{3,1}$ be the subgroup consisting of the matrices of the form $\left(\begin{smallmatrix}1 & * \\0 & * \end{smallmatrix}\right)$.   
%\item
%Let $H_{3,2}$ be the subgroup consisting of the matrices of the form $\left(\begin{smallmatrix} * & * \\0 & 1 \end{smallmatrix}\right)$.   
%\end{itemize}

\begin{prop}[\cite{zywina}, Prop. 1.16]  \label{prop-zywina3}
Let $E$ be an elliptic curve over $\QQ$ with $j_E=0$.  Take any odd prime $\ell$.
\begin{romanenum}
\item  \label{P:j=0 situation i}
If $\ell \equiv 1 \pmod{9}$, then $\rho_{E,\ell}(G_\Q)$ is conjugate to $N_{s}(\ell)$ in $\GL_2(\FF_\ell)$.
\item \label{P:j=0 situation ii}
If $\ell \equiv 8 \pmod{9}$, then $\rho_{E,\ell}(G_\Q)$ is conjugate to $N_{ns}(\ell)$ in $\GL_2(\FF_\ell)$.
\item \label{P:j=0 situation iii}
Suppose that $\ell$ is congruent to $4$ or $7$ modulo $9$.  Let $E'/\QQ$ be the elliptic curve over $\QQ$ defined by $y^2=x^3+16 \ell^e$, where $e\in \{1,2\}$ satisfies $ \frac{\ell-1}{3} \equiv e \pmod{3}$.   

\noindent
If $E$ is not isomorphic to a quadratic twist of $E'$, then $\rho_{E,\ell}(G_\Q)$ is conjugate to $N_{s}(\ell)$ in $\GL_2(\FF_\ell)$.

\noindent 
If $E$ is isomorphic to a quadratic twist of $E'$, then $\rho_{E,\ell}(G_\Q)$ is conjugate in $\GL_2(\FF_\ell)$ to the subgroup $G$ of $N_s(\ell)$ consisting of the matrices of the form  $\left(\begin{smallmatrix} a & 0 \\0 & b\end{smallmatrix}\right)$ or $\left(\begin{smallmatrix} 0 & a \\b & 0 \end{smallmatrix}\right)$ with $a/b \in (\FF_\ell^\times)^3$.

\item 
\label{P:j=0 situation iv}
Suppose that $\ell$ is congruent to $2$ or $5$ modulo $9$.  Let $E'/\QQ$ be the elliptic curve over $\QQ$ defined by $y^2=x^3+16 \ell^e$, where $e\in \{1,2\}$ satisfies $ \frac{\ell+1}{3} \equiv -e \pmod{3}$.   

\noindent
If $E$ is not isomorphic to a quadratic twist of $E'$, then $\rho_{E,\ell}(G_\Q)$ is conjugate to $N_{ns}(\ell)$ in $\GL_2(\FF_\ell)$.

\noindent 
If $E$ is isomorphic to a quadratic twist of $E'$, then $\rho_{E,\ell}(G_\Q)$ is conjugate in $\GL_2(\FF_\ell)$ to the subgroup $G$ of $N_{ns}(\ell)$ generated by the unique index $3$ subgroup of $C_{ns}(\ell)$ and by $\left(\begin{smallmatrix} 1 & 0 \\0 & -1\end{smallmatrix}\right)$.

\item  \label{P:j=0 situation v}
Suppose that $\ell=3$.  The curve $E$ can be given by a Weierstrass equation $y^2=x^3+d$ for some $d\in \QQ^\times$.   Fix notation as above.

\noindent
If $d$ or $-3d$ is a square and $-4d$ is a cube, then $\rho_{E,3}(G_\Q)$ is conjugate to
\[
H_{1,1}=\{ \left(\begin{smallmatrix} 1 & 0 \\0 & a \end{smallmatrix}\right) : a\in \FF_\ell^\times\}.
\]
\noindent
If $d$ and $-3d$ are not squares and $-4d$ is a cube, then $\rho_{E,3}(G_\Q)$ is conjugate to $C_s(3)$.

\noindent
If $d$ is a square and $-4d$ is not a cube, then $\rho_{E,3}(G_\Q)$ is conjugate to
\[
H_{3,1}=\{ \left(\begin{smallmatrix} 1 & b \\0 & a \end{smallmatrix}\right) : a\in \FF_\ell^\times\}.
\]
\noindent
If $-3d$ is a square and $-4d$ is not a cube, then $\rho_{E,3}(G_\Q)$ is conjugate to 
\[
H_{3,2}=\{ \left(\begin{smallmatrix} a & b \\0 & 1 \end{smallmatrix}\right) : a\in \FF_\ell^\times\}.
\]
\noindent
If $d$ and $-3d$ are not squares and $-4d$ is not a cube, then $\rho_{E,3}(G_\Q)$ is conjugate to $B(3)$.

\end{romanenum}
\end{prop}

\subsection{Classification of $\ell$-adic images} \label{sec-alvaroCMresults}

For a prime $\ell$, the second author has given a complete classification of the possible $\ell$-adic images $G_{E,\ell^\infty} = \rho_{E,\ell^{\infty}}(G_{\Q(j_{K,f})})$ attached to elliptic curves $E/\Q(j_{K,f})$ with CM by $\OO_{K,f}$.  Here we summarize the  results from \cite{lozano-galoiscm}. 

\begin{theorem}[\cite{lozano-galoiscm}, Theorems 1.1 and 1.2]\label{thm-cmrep-intro-alvaro}\label{thm-firstofsec1} Let $E/\Q(j_{K,f})$ be an elliptic curve with CM by $\OO_{K,f}$, and let $N\geq 2$.
	 Then:
	\begin{enumerate}
            \item There is a $\Z/N\Z$-basis of $E[N]$ such that the image of $\rho_{E,N}$
	is contained in $\mathcal{N}_{\delta,\phi}(N)$, and	 the index of the image of $\rho_{E,N}$ in $\mathcal{N}_{\delta,\phi}(N)$ is a divisor of the order of $\Of^\times/\mathcal{O}_{K,f,N}^\times$, where $\mathcal{O}_{K,f,N}^\times=\{u\in\Of^\times: u\equiv 1 \bmod N\Of\}$.
		\item There is a compatible system of bases of $E[N]$ such that the image of $\rho_E$ is contained in $\mathcal{N}_{\delta,\phi}$, and the index of the image of $\rho_{E}$ in $\mathcal{N}_{\delta,\phi}$ is a divisor of the order $\Of^\times$. In particular, the index is a divisor of $4$ or $6$. 
	\end{enumerate} 
\end{theorem}

The following theorem shows that, for a prime $\ell>3$, the image of $\rho_{E,\ell^\infty}$ is, in fact, defined modulo $\ell$. Moreover, if the prime $\ell$ does not divide $2\Delta_K f$, then the $\ell$-adic image is as large as possible.

\begin{thm}[\cite{lozano-galoiscm}, Theorem 1.2]\label{thm-goodredn}
Let $E/\Q(j_{K,f})$ be an elliptic curve with CM by $\OO_{K,f}$, and let $\ell$ be prime. If $\ell > 3$, or if $\ell > 2$ and $j_{K,f} \neq 0$, then $G_{E,\ell^\infty}$ is the full inverse image via the natural reduction mod $\ell$ map $\mathcal{N}_{\delta,\phi}(\ell^\infty)\to \mathcal{N}_{\delta,\phi}(\ell)$ of the image $G_{E,\ell}$ of $\rho_{E,\ell}\equiv \rho_{E,\ell^\infty}\bmod \ell$. Further, if $\ell \nmid 2 \Delta_K f$, then $G_{E,\ell} = \mathcal{N}_{\delta, \phi}(\ell)$, and so $G_{E,\ell^{\infty}} = \mathcal{N}_{\delta, \phi}(\ell^{\infty})$.
\end{thm}

The following result describes the $\ell$-adic image in the case of $j=0$ and $\ell \neq 3$. Before we state this result, we introduce the following terminology: we say that the $\ell$-adic image $G_{E,\ell^\infty}$ of $\rho_{E,\ell}$ is defined modulo $\ell^n$ if $\pi_{\ell^n}^{-1}(\pi_{\ell^n}(G_{E,\ell^\infty}))=G_{E,\ell^\infty}$, where $\pi_{\ell^n}\colon \mathcal{N}_{\delta,\phi}(\ell^{\infty})\to \mathcal{N}_{\delta,\phi}(\Z/\ell^n\Z)$ is the natural reduction mod-$\ell^n$ map.  In other words, we say that the $\ell$-adic image is defined modulo $\ell^n$ if the $\ell$-adic image can be retrieved from the mod-$\ell^n$ image as the full inverse image under the reduction map $\pi_{\ell^n}$.

\begin{thm}[\cite{lozano-galoiscm}, Theorems 1.4 and 1.8]\label{thm-jzero-goodredn} Let $E/\Q$ be an elliptic curve with $j(E)=0$, and let $\ell > 3$, so that the $\ell$-adic image of $E$ is defined mod $\ell$. Then one of the following holds:
	\begin{enumerate}
	\item If $\ell\equiv \pm 1\bmod 9$, then $G_{E,\ell^{\infty}}=\mathcal{N}_{\delta,0}(\ell^{\infty})$,
	\item If $\ell\equiv 2$ or $5\bmod 9$, then $G_{E,\ell^{\infty}}=\mathcal{N}_{\delta,0}(\ell^{\infty})$ or $[\mathcal{N}_{\delta,0}(\ell^{\infty}) : G_{E,\ell^{\infty}}] = 3$ and  $G_{E,\ell^{\infty}}=\langle \cC_{\delta,0}(\ell^{\infty})^3, c_\varepsilon \rangle $,
	\item If $\ell\equiv 4$ or $7\bmod 9$, then $G_{E,\ell^{\infty}} = \mathcal{N}_{\delta,0}(\ell^{\infty})$ or $[\mathcal{N}_{\delta,0}(\ell^{\infty}) : G_{E,\ell^{\infty}}] = 3$ and $G_{E,\ell^\infty}$ is generated by $c_{-\varepsilon}$ and 
 \[ G^{\, 3,1}_{-3/4,0} = \left\{  \left(\begin{array}{cc} \nu\cdot (a + b) & a - b \\ \delta\cdot (a - b) & \nu\cdot (a + b)\\ \end{array}\right) : a/b \in (\Z_\ell^\times)^3\right\},  \]
 where $\delta \equiv \nu^2 \bmod \ell$.
\end{enumerate}
If instead $\ell = 2$, the $2$-adic image of $E$ is defined mod $2$, and $G_{E,2^\infty} = \mathcal{N}_{-1,1}(2^\infty)$, or $[\mathcal{N}_{-1,1}(2^\infty):G_{E,2^\infty}]=3$ and $G_{E,2^\infty} =\left\langle \cC_{-1,1}(2^\infty)^3, c_{\varepsilon}' \right\rangle$.
\end{thm}

\begin{proof} The results follow directly from \cite[Theorems 1.4, 1.8]{lozano-galoiscm}, except for (3) so we prove this case. \cite[Thm. 1.4]{lozano-galoiscm} says that the possible index 3 subgroup in case (3) is
conjugate to the subgroup \[ \left\langle \left\{\left(\begin{array}{cc} a & 0\\ 0 & b\\ \end{array}\right) : a/b \in (\Z_\ell^\times)^3\right\}, \left(\begin{array}{cc} 0 & \varepsilon\\ \varepsilon & 0\\ \end{array}\right)\right\rangle. \]
With a choice of $\nu$ such that $\delta \equiv \nu^2 \bmod \ell$, the change of basis matrix $\displaystyle \left(\begin{array}{cc} -\nu & -1 \\ \nu & -1 \\ \end{array}\right)$
sends 
% \[ \left(\begin{array}{cc} a & 0\\ 0 & b\\ \end{array}\right) 
% \mapsto  \frac{1}{2\nu} 
% \left(\begin{array}{cc} -1 & 1 \\ -\nu & -\nu \\ \end{array}\right)\left(\begin{array}{cc} a & 0 \\ 0 & b \\ \end{array}\right)\left(\begin{array}{cc} -1 & \nu \\ -1 & -\nu \\ \end{array}\right) = \frac{1}{2 \nu}
% \left(\begin{array}{cc} \nu\cdot (a + b) & a - b \\ \delta\cdot (a - b) & \nu\cdot (a + b)\\ \end{array}\right) \]
% %\[
% %\text{and} 
% \qquad
% \left( \begin{array}{cc} 0 & \varepsilon\\ \varepsilon & 0\\ \end{array}\right) 
% \mapsto 
% \frac{1}{2\nu}
% \left(\begin{array}{cc} -1 & 1 \\ -\nu & -\nu \\ \end{array}\right)
% \left( \begin{array}{cc} 0 & \varepsilon\\ \varepsilon & 0\\ \end{array}\right) 
% \left(\begin{array}{cc} -1 & \nu \\ -1 & -\nu \\ \end{array}\right)
% =
% \left( \begin{array}{cc} -\varepsilon & 0 \\ 0 &\varepsilon \\ \end{array}\right), \]
\[ \left(\begin{array}{cc} a & 0\\ 0 & b\\ \end{array}\right) 
\mapsto \frac{1}{2 \nu}
\left(\begin{array}{cc} \nu\cdot (a + b) & a - b \\ \delta\cdot (a - b) & \nu\cdot (a + b)\\ \end{array}\right) \qquad
\text{and} 
\qquad
\left( \begin{array}{cc} 0 & \varepsilon\\ \varepsilon & 0\\ \end{array}\right) 
\mapsto 
\left( \begin{array}{cc} -\varepsilon & 0 \\ 0 &\varepsilon \\ \end{array}\right), \]
from which the result follows.
\end{proof}

The following result describes the $\ell$-adic image in the case of a prime $\ell>2$ that divides $\Delta_K f$.

\begin{thm}[\cite{lozano-galoiscm}, Theorem 1.5]\label{thm-j0ell3alvaro}\label{thm-oddprimedividingdisc}
	Let $E/\Q(j_{K,f})$ be an elliptic curve with CM by $\Of$, and let $\ell$ be an odd prime dividing $f\Delta_K$ (so, in particular, $j_{K,f} \neq 1728$). Then $G_{E,\ell^\infty}$ is precisely one of the following groups:
	\begin{enumerate}
		\item[(a)] If $j_{K,f}\neq 0,1728$, then either $G_{E,\ell^\infty}=\mathcal{N}_{\delta,0}(\ell^\infty)$, or $G_{E,\ell^\infty}$ is generated by $c_{\varepsilon}$ and the group
		 \[ G^{\, 2, 1}_{\delta, 0} = \left\{ \left(\begin{array}{cc} a^2 & b\\ \delta b & a^2 \\ \end{array}\right): a \in \Z_\ell^\times, b \in \Z_\ell \right\}.\] 
		\item[(b)] If $j_{K,f}=0$, then $\ell=3$, and there are twelve possibilities for $G_{E,3^\infty}$. More concretely, either $G_{E,3^\infty}=\mathcal{N}_{-3/4,0}(3^\infty)$, or $[\mathcal{N}_{-3/4,0}(3^\infty):G_{E,3^\infty}]=2,3,$ or $6$, and one of the following holds:
		\begin{enumerate} \item[(i)] If $[\mathcal{N}_{-3/4,0}(3^\infty):G_{E,3^\infty}]=2$, then $G_{E,3^\infty}$ is generated by $c_{\varepsilon}$ and 
		\[ G^{\, 2,1}_{-3/4, 0}=\left\{ \left(\begin{array}{cc}  a & b\\ -3b/4 & a\\ \end{array}\right): a,b\in  \Z_3,\ a\equiv 1 \bmod 3\right\}. \]
		
		\item[(ii)]  If $[\mathcal{N}_{-3/4,0}(3^\infty):G_{E,3^\infty}]=3$, then $G_{E,3^\infty}$ is generated by $c_{\varepsilon}$ and 
		\[ G^{\, 3,1}_{-3/4, 0} = \left\{ \left(\begin{array}{cc} a & b\\ -3b/4 & a\\ \end{array}\right): a\in \Z_3^\times,\ b\equiv 0 \bmod 3 \right\}, \]
		\[ \text{or }\  G^{\, 3,2}_{-3/4, 0}=\left\langle \left(\begin{array}{cc}  2 & 0\\ 0 & 2\\ \end{array}\right) ,\left(\begin{array}{cc}  1 & 1\\ -3/4 & 1\\ \end{array}\right)\right\rangle,\ \text{ or } \ G^{\, 3,3}_{-3/4, 0} =\left\langle \left(\begin{array}{cc}  2 & 0\\ 0 & 2\\ \end{array}\right) ,\left(\begin{array}{cc}  -5/4 & 1/2\\ -3/8 & -5/4\\ \end{array}\right)\right\rangle. \] 
		\item[(iii)]  If $[\mathcal{N}_{-3/4,0}(3^\infty):G_{E,3^\infty}] = 6$, then $G_{E,3^\infty}$ is generated by $c_{\varepsilon}$ and one of
		\[ \ G^{\, 6,1}_{-3/4, 0}=\left\{ \left(\begin{array}{cc}  a & b\\ -3b/4 & a\\ \end{array}\right) : a\equiv 1,\ b\equiv 0 \bmod 3\Z_3 \right\}, \]
		\[ \text{or }\ G^{\, 6,2}_{-3/4, 0}=\left\langle \left(\begin{array}{cc}  4 & 0\\ 0 & 4\\ \end{array}\right) ,\left(\begin{array}{cc}  1 & 1\\ -3/4 & 1\\ \end{array}\right)\right\rangle,\ \text{ or } \  G^{\, 6,3}_{-3/4, 0}=\left\langle \left(\begin{array}{cc}  4 & 0\\ 0 & 4\\ \end{array}\right) ,\left(\begin{array}{cc}  -5/4 & 1/2\\ -3/8 & -5/4\\ \end{array}\right)\right\rangle. \]
		\end{enumerate}
	\end{enumerate}
\end{thm}

The next result describes the $2$-adic image in the case of $j_{K,f}\neq 0,1728$.

\begin{thm}[\cite{lozano-galoiscm}, Theorem 1.6]\label{thm-m8and16alvaro}
	Let $E/\Q(j_{K,f})$ be an elliptic curve with CM by $\OO_{K,f}$ with $j_{K,f}\neq 0, 1728$. Then, either $G_{E,2^{\infty}} = \mathcal{N}_{\delta, \phi}(2^{\infty})$, or $[\mathcal{N}_{\delta, \phi}(2^{\infty}) :  G_{E,2^{\infty}}] = 2$ and one of the following two possibilities hold:
	\begin{enumerate} 
		\item $\Delta_K f^2 \equiv 0 \bmod 16$, and in particular
        \begin{itemize}
         \item $\Delta_K \equiv 1 \bmod 4$ and $f \equiv 0 \bmod 4$ or 
         \item $\Delta_K \equiv 0 \bmod 4$ and $f \equiv 0 \bmod 2$. \end{itemize}
      In this case, $G_{E,2^{\infty}}$ is generated by $c_{\varepsilon}$ and one of the groups
			$$G^{\, 2,1}_{\delta, \phi}=\left\langle \left(\begin{array}{cc} 5 & 0\\ 0 & 5\\\end{array}\right),\left(\begin{array}{cc} 1 & 1\\ \delta & 1\\\end{array}\right)\right\rangle \text{ or } G^{\, 2,2}_{\delta, \phi}=\left\langle \left(\begin{array}{cc} 5 & 0\\ 0 & 5\\\end{array}\right),\left(\begin{array}{cc} -1 & -1\\ -\delta & -1\\\end{array}\right)\right\rangle.$$
		\item $\Delta_K\equiv 0 \bmod 8$, or $\Delta_K\equiv 4 \bmod 8$ and $f\equiv 0 \bmod 4$, or $\Delta_K\equiv 1 \bmod 4$ and $f\equiv 0 \bmod 8$, and in this case $G_{E,2^{\infty}}$ is generated by $c_{\varepsilon}$ and one of the groups
			$$G^{\, 2,3}_{\delta, \phi} =\left\langle \left(\begin{array}{cc} 3 & 0\\ 0 & 3\\\end{array}\right),\left(\begin{array}{cc} 1 & 1\\ \delta & 1\\\end{array}\right)\right\rangle \text{ or } G^{\, 2,4}_{\delta, \phi}=\left\langle \left(\begin{array}{cc} 3 & 0\\ 0 & 3\\\end{array}\right),\left(\begin{array}{cc} -1 & -1\\ -\delta & -1\\\end{array}\right)\right\rangle.$$
	\end{enumerate}
	%$$\left\langle \left(\begin{array}{cc} \varepsilon & 0\\ 0 & -\varepsilon\\\end{array}\right),\left(\begin{array}{cc} \alpha & 0\\ 0  & \alpha\\\end{array}\right),\left(\begin{array}{cc} 1 & 1\\ \delta  & 1\\\end{array}\right)\right\rangle \text{ or } \left\langle \left(\begin{array}{cc} \varepsilon & 0\\ 0 & -\varepsilon\\\end{array}\right),\left(\begin{array}{cc} \alpha & 0\\ 0  & \alpha\\\end{array}\right),\left(\begin{array}{cc} -1 & -1\\ -\delta  & -1\\\end{array}\right)\right\rangle \subseteq \GL(2,\Z_2).$$
\end{thm}

The following result describes the $2$-adic image when $j_{K,f}=1728$.

\begin{thm}[\cite{lozano-galoiscm}, Theorem 1.7]\label{thm-j1728alvaro}\label{thm-lastofsec1}
	Let $E/\Q$ be an elliptic curve with $j(E)=1728$, and let $\Gamma = \{ c_{1}, c_{-1}, c'_{1}, c'_{-1} \}$. Then $G_{E,2^\infty} = \mathcal{N}_{-1,0}(2^\infty)$, or $[\mathcal{N}_{-1,0}(2^\infty) : G_{E,2^\infty}] = 2$ or $4$ and $G_{E,2^\infty}$ is one of the following groups:
	\begin{itemize} 
		\item If $[\mathcal{N}_{-1,0}(2^\infty):G_{E,2^\infty}]=2$, then $G_{E,2^\infty}$ is generated by some $\gamma\in \Gamma$ and one of the groups
		\[ G^{\, 2,1}_{-1, 0}=\left\langle -\operatorname{Id}, 3\cdot \operatorname{Id},\left(\begin{array}{cc} 1 & 2\\ -2 & 1\\\end{array}\right) \right\rangle, \text{ or } G^{\, 2,2}_{-1, 0}=\left\langle -\operatorname{Id}, 3\cdot \operatorname{Id},\left(\begin{array}{cc} 2 & 1\\ -1 & 2\\\end{array}\right) \right\rangle. \]
		\item If $[\mathcal{N}_{-1,0}(2^\infty) : G_{E,2^\infty}]=4$, then $G_{E,2^\infty}$ is generated by some $\gamma \in \Gamma$ and one of the groups
		\[ G^{\, 4,1}_{-1, 0}=\left\langle 5\cdot \operatorname{Id},\left(\begin{array}{cc} 1 & 2\\ -2 & 1\\\end{array}\right) \right\rangle, \text{ or } G^{\, 4,2}_{-1, 0}=\left\langle 5\cdot \operatorname{Id},\left(\begin{array}{cc} -1 & -2\\ 2 & -1\\\end{array}\right) \right\rangle, \text{ or } \]
		\[ G^{\, 4,3}_{-1, 0}= \left\langle -3\cdot \operatorname{Id},\left(\begin{array}{cc} 2 & -1\\ 1 & 2\\\end{array}\right) \right\rangle, \text{ or } G^{\, 4,4}_{-1, 0}=\left\langle -3\cdot \operatorname{Id},\left(\begin{array}{cc} -2 & 1\\ -1 & -2\\\end{array}\right) \right\rangle. \]
\end{itemize}
\end{thm}

%%%%%%%%%%%%%%%%%%%%%%
%Jeff Yelton's Theorem, useful for later sections
%%%%%%%%%%%%%%%%%%%%%%

\subsection{The $4$- and $8$-division fields}

We conclude this section citing a theorem from \cite{yelton} which describes the $4$- and $8$-division fields of an elliptic curve.

\begin{theorem}[\cite{yelton}]\label{thm-8divisionfield}
\label{8-division-thm} Let $K$ be a field of characteristic not $2$, and let $E/K$ be an elliptic curve. Let $E/K$ have a model of the form $y^2 = (x - \alpha_1)(x - \alpha_2)(x - \alpha_3)$,
where $\alpha_1,\alpha_2,\alpha_3 \in \overline{K}$. Also, fix $\zeta_8 \in \overline{K}$ a primitive eighth root of unity, and $\zeta_4 = \zeta_8^2$ a primitive fourth root of unity. For $j \in \Z/3\Z$, choose an element $A_j \in \overline{K}$ whose square is $\alpha_{j + 1} - \alpha_{j + 2}$, and choose an element $B_j \in \overline{K}$ whose square is $A_j(A_{j + 1} + \zeta_4 A_{j + 2})$. Then $K(E[4]) = K(E[2], \zeta_4, A_1, A_2, A_3)$ and $K(E[8]) = K(E[4], \zeta_8, B_1, B_2, B_3)$.
\end{theorem}

\section{Orders of class number $1$}\label{sec-classnumber1}

We are ready to begin our classification of Weierstrass models of CM elliptic curves according to their $\ell$-adic images. See Section \ref{sec-structure} for a summary of the structure of the proofs. In this section we discuss the case of an order of class number $1$, i.e., elliptic curves with CM defined over $\Q$.

\subsection{The case of $\Delta_K f^2\neq -3$ and $\ell\geq 3$}\label{sec-notminus3ellmorethan2}

 Suppose that $E/\Q(j_{K,f})$ is an elliptic curve with CM by $\Of$ of class number one ($|\Delta_K f^2|\in \Delta_1$) and $\Delta_K f^2\neq -3$ (i.e., $j_{K,f}\in \Q^\ast$), and let $\ell\geq 3$ be prime. 
 
 \subsubsection{$\ell\geq 3$ does not divide $\Delta_K f$}\label{sec-ellnotdividedisc} We first consider the case where $\ell$  is not a  divisor of $\Delta(\mathcal{O})=\Delta_K f^2$ (since $\ell>2$, the prime $\ell$ is not a divisor of $2\Delta_K f$). In this case, by Theorem \ref{thm-goodredn}, two things are true: (1) the image of $\rho_{E,\ell^\infty}$ is  the full normalizer $\mathcal{N}_{\delta,\phi}(\ell^\infty)$, and (2) the conjugacy class of the image is determined by the image mod $\ell$. Since in the case we are discussing there are only two possibilities for the  mod $\ell$ image (by Prop. \ref{prop-zywina1}), namely $N_s(\ell)$ or $N_{\text{ns}}(\ell)$, the $\ell$-adic image corresponds to one of two possible labels, namely $\ell\texttt{.}0\texttt{.s-1.1.1}$ or $\ell\texttt{.}0\texttt{.ns-1.1.1}$ (where the labels were defined in Section \ref{sec-labels}). In particular, Prop. \ref{prop-zywina1} implies the following:
\begin{romanenum}
\item \label{P:CM main a}
If $\legendre{\Delta_K}{\ell}=1$, then $\rho_{E,\ell}(G_\QQ)$ is conjugate in $\GL_2(\FF_\ell)$ to $N_s(\ell)$. Thus, the label of the $\ell$-adic image  is $\ell\texttt{.}0\texttt{.s-1.1.1}$.
\item \label{P:CM main b}
If $\legendre{\Delta_K}{\ell}=-1$, then $\rho_{E,\ell}(G_\QQ)$ is conjugate in $\GL_2(\FF_\ell)$ to $N_{ns}(\ell)$. Thus, the label of the image (with labels as defined in Section \ref{sec-labels}) is $\ell\texttt{.}0\texttt{.ns-1.1.1}$.
\end{romanenum}

 \subsubsection{$\ell\geq 3$ divides $\Delta_K f$}\label{sec-elldividesdisc} Now suppose that $\ell\geq 3$ divides $\Delta_K f^2$. Since we are in the case of $|\Delta_K f^2|\in \Delta_1$ and $\Delta_K f^2\neq -3$, then $\Delta_K=-\ell$ as Zywina points out (note that this is no longer true in the class number $2$ case where we can have $\ell=3$ and $\Delta_K f^2=(-4)\cdot 3^2$, for example). By Prop. \ref{prop-zywina1} (iii), there are three possibilities (using the notation of subgroups $G$, $H_1$, and $H_2$ of  $\GL(2,\F_\ell)$ as in  the statement): 

\begin{enumerate}
\item If $E$ is isomorphic to $E_{\Of}$ over $\Q$, then $\rho_{E,\ell}(G_\Q)$ is conjugate in $\GL_2(\FF_\ell)$ to $H_1$, which corresponds to the group $\langle G^{2,1}_{\delta,0}(\Z/\ell\Z), c_1\rangle$ of Theorem \ref{thm-oddprimedividingdisc} (a). Therefore the $\ell$-adic image is $\langle G^{2,1}_{\delta,0}, c_1\rangle$, which in turn corresponds to the label $\ell\texttt{.}1\texttt{.ns-}\ell\texttt{.2.1}$ if $\Delta_K f^2\neq -3\cdot 3^2=-27$, and to $3\texttt{.}3\texttt{.ns-}\texttt{3.2.1}$ if $\Delta_K f^2 = -27$ (note that $\Delta_K f^2/\ell = -f^2$ and $-1$ is a quadratic non-residue mod $\ell$ because every $\ell\geq 3$ dividing $\delta\in \Delta_1$ is congruent to $3 \bmod 4$).

\item If $E$ is isomorphic to the quadratic twist of $E_{\Of}$ by $-\ell$ over $\Q$, then $\rho_{E,\ell}(G_\Q)$ is conjugate in $\GL_2(\FF_\ell)$ to $H_2$. As in the previous case, the $\ell$-adic image is $\langle G^{2,1}_{\delta,0}, c_{-1}\rangle$, which in turn corresponds to the label $\ell\texttt{.}1\texttt{.ns-}\ell\texttt{.2.2}$ if $\Delta_K f^2\neq -3\cdot 3^2=-27$, and to $3\texttt{.}3\texttt{.ns-}\texttt{3.2.2}$ if $\Delta_K f^2 = -27.$

\item If $E$ is not isomorphic to $E_{\Of}$ or its quadratic twist by $-\ell$ over $\Q$, then $\rho_{E,\ell}(G_\Q)$ is conjugate in $\GL_2(\FF_\ell)$ to $G$. Hence, the $\ell$-adic image is all of $\mathcal{N}_{\delta,\phi}(\ell^\infty)$ and it corresponds to the label $\ell\texttt{.}1\texttt{.ns-1.1.1}$ if $\Delta_K f^2\neq -3\cdot 3^2=-27$, and to $3\texttt{.}3\texttt{.ns-1.1.1}$ if $\Delta_K f^2 = -27.$

\end{enumerate}

\subsection{The case of $\Delta_K f^2\neq -3$ and $\ell=2$}\label{sec-notminus3ellequals2} We continue the case when $E/\Q(j_{K,f})$ is an elliptic curve with CM by $\Of$ of class number one ($|\Delta_K f^2|\in \Delta_1$) and $\Delta_K f^2\neq -3$ (i.e., $j_{K,f}\in \Q^\ast$), and now let $\ell=2$. 

\subsubsection{$\Delta_K f^2 \neq -3,-4,-8,-16$ and $\ell =2$}\label{sec-not34816ell2} Suppose first that $\ell=2$ and $-\Delta_K f^2\in \Delta_1\smallsetminus\{-3,-4\}$. Then, Theorem \ref{thm-m8and16alvaro} says that unless $\Delta_K f^2 = -8$ or $-16$, then the $2$-adic image must be all of $\mathcal{N}_{\delta,\phi}(2^\infty)$. It remains to examine the square class of $\Delta_K f^2/2^\nu$, where $\nu=\nu_2(\Delta_K f^2)$ and $\nu_2$ denotes the $2$-adic valuation. One verifies that $-7,-28$ are split while $-12,-27,-11,-19,-43,-67$ and $-163$ are in the non-square class of $5$. Thus, for $\Delta_K f^2=-7$ we have a label \texttt{2.0.s-1.1.1} and for $-28$ we have \texttt{2.2.s-1.1.1}, for $-12$ we get \texttt{2.2.ns5-1.1.1}, and for $-27,-11,-19,-43,-67,-163$ we get \texttt{2.0.ns5-1.1.1}, as it appears in Table \ref{tab-ell2}.

\subsubsection{$\Delta_Kf^2=-4$ ($j = 1728$), $\ell = 2$}\label{sec-4ell2} By our work in Section \ref{sec-not34816ell2}, now it remains to analyze the cases of $\Delta_K f^2 = -4,-8,-16$ and $\ell=2$. Here we discuss the case of $\Delta_K f^2=-4$ (i.e., $j_{K,f}=1728$).

%\section{$j=1728$, $\ell=2$}
\begin{figure}[ht]
 \scalebox{.68}{
$$
\xymatrix@R=1mm@C=1mm{ 
& & & & & & \fbox{1.1.1}\ar@/_1pc/@{-}[dlllll]_2 \ar@{-}[dll]_2\ar@{-}[drr]^2 \ar@/^1pc/@{-}[drrrrr]^2 \\
& \fbox{$\begin{array}{c}4.2.2\\ -t^2\end{array}$}
\ar@{-}[dl]_2\ar@{-}[dr]^2
& & &
 \fbox{$\begin{array}{c}4.2.1\\ t^2\end{array}$} 
 \ar@{-}[dl]_2\ar@{-}[dr]^2
 & &
 & &
 \fbox{$\begin{array}{c}8.2.1\\ 2t^2\end{array}$} 
 \ar@{-}[dl]_2
 \ar@{-}[dr]^2
 & & & 
 \fbox{$\begin{array}{c}8.2.2\\ -2t^2\end{array}$} 
 \ar@{-}[dl]_2\ar@{-}[dr]^2
 &  
\\
\fbox{$\begin{array}{c} 4.4.3\\ -4\end{array}$}  & 
  & 
\fbox{$\begin{array}{c}4.4.4\\ -1\end{array}$}  & 
\fbox{$\begin{array}{c}4.4.2\\ 1\end{array}$}  & 
 & 
\fbox{$\begin{array}{c}4.4.1\\ 4\end{array}$}  & 
&
\fbox{$\begin{array}{c}16.4.1\\ 2\end{array}$}  & 
& 
\fbox{$\begin{array}{c}16.4.2\\ 8\end{array}$}  & 
\fbox{$\begin{array}{c}16.4.3\\ -8\end{array}$}  & 
& 
\fbox{$\begin{array}{c}16.4.4\\ -2\end{array}$}  & 
}
$$
}
\caption{The labels for the $2$-adic images of curves $y^2=x^3+dx$. For each value of $d$ indicated in the second row of each box, the $2$-adic label starts with \texttt{2.2.ns7-} and ends with the digits indicated in the first row of the box.}\label{tab-1728}
\end{figure}
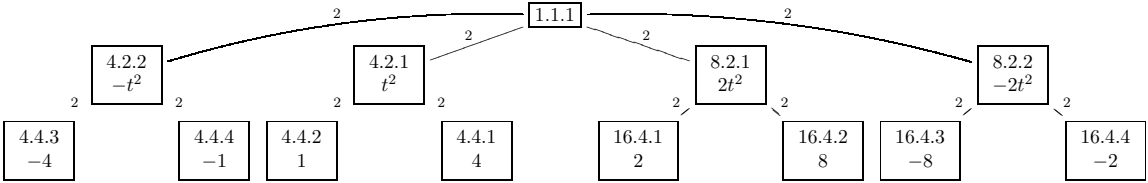

Our goal is to prove the classification $2$-adic images of twists of $E_{\Z[i]}$ which is summarized in Figure \ref{tab-1728} and also detailed in Table \ref{tab-ell2}. Our first preliminary result describes the $4$- and $8$-division fields of an elliptic curve with $j=1728$, as a corollary of Theorem \ref{thm-8divisionfield}.

\begin{corollary}\label{cor-8divfieldj1728}
Let $E/\Q$ be an elliptic curve with $j(E) = 1728$, and suppose that $E : y^2 = x^3 -N x$, where  $N>0$. If $\zeta_{16}$ is a primitive $16$th root of unity, and $\zeta_8 = \zeta_{16}^2$ is a primitive $8$th root of unity, then
\begin{itemize}
	\item $\Q(E[4]) = \Q(\zeta_8, \sqrt[4]{N})$, and
	\item $\Q(E[8]) = \Q(\zeta_{16}, \sqrt{-1+\sqrt{2}}, \sqrt[4]{N})$
\end{itemize}
\end{corollary}

\begin{proof}
Let $\zeta_{16} := e^{\pi i/8}$, so that $\zeta_8 = \zeta_{16}^2 = (1 + i)/\sqrt{2}$ and $\zeta_4 = \zeta_8^2 = i$. The Weierstrass model for $E$ can be expressed as $y^2 = x(x - \sqrt{ N})(x + \sqrt{N})$, so we let $\alpha_1 = 0$, $\alpha_2 = \sqrt{ N}$, and $\alpha_3 = - \sqrt{N}$. Then, let $A_1 = \sqrt{2} \sqrt[4]{N}$ and $A_2 = A_3 = i \sqrt[4]{N}$ as in Theorem \ref{8-division-thm}. Thus,
\begin{align*}
\Q(E[4]) &= \Q(E[2], i, A_1, A_2, A_3) = \Q( \sqrt{N}, i, \sqrt{2}, \sqrt[4]{N}) = \Q(i, \sqrt{2}, \sqrt[4]{N}),
\end{align*}
and since $\Q(\zeta_8) = \Q(i, \sqrt{2})$, we have 
$\Q(E[4]) = \Q(\zeta_8, \sqrt[4]{N}),$ as stated. Next, let
\begin{align*}
B_1 &= \sqrt{ \sqrt{2} \sqrt[4]{N} \left(i \sqrt[4]{N} - \sqrt[4]{N} ) \right)} = \sqrt{ \sqrt{2} (i - 1) \sqrt{N} } = \sqrt{2} \, \zeta_{16}^3  \sqrt[4]{N}, \\
B_2 &= \sqrt{ i \sqrt[4]{N} \left( i \sqrt[4]{N} + \sqrt{2} i \sqrt[4]{N} \right) } = \sqrt{ (-\sqrt{2} - 1) \sqrt{N} } = i \sqrt{ \sqrt{2} + 1} \sqrt[4]{N},  \\
B_3 &= \sqrt{ i \sqrt[4]{N} \left( \sqrt{2} \sqrt[4]{N} -  \sqrt[4]{N} \right) } = \sqrt{ i (\sqrt{2} - 1) \sqrt{ N}} = \zeta_8 \sqrt{ \sqrt{2} - 1} \sqrt[4]{N}.
\end{align*}

Since $(\sqrt{2} + 1)(\sqrt{-1+\sqrt{2}}) = \sqrt{1+\sqrt{2}}$, we have that 
\begin{align*}
\Q(E[8]) 
&= \Q(E[4], \zeta_8, B_1, B_2, B_3) \\
&= \Q\left(\sqrt[4]{N}, \zeta_8, \sqrt{2} \, \zeta_{16}^3  \sqrt[4]{N}, i \sqrt{ 1+\sqrt{2}} \sqrt[4]{N},  \zeta_8 \sqrt{ -1+\sqrt{2}} \sqrt[4]{N}\right) \\
&= \Q\left(\zeta_{16}, \sqrt{-1+\sqrt{2}}, \sqrt[4]{N}\right),
\end{align*}
as desired.
\end{proof}

\begin{lemma}\label{lem-j1728index}
    Let $E/\Q$ be an elliptic curve with $j(E)=1728$. Let $G_{E,8}$ be the image of $\rho_{E,8}\colon G_\Q\to\mathcal{N}_{-1,0}(8)$, and let $d_{E,8}=[\mathcal{N}_{-1,0}(8):G_{E,8}]$. Then:
    \begin{enumerate}
        \item $d_{E,8}=1$, $2$, or $4$, and
        \item $d_{E,8}$ coincides with the index of the image $G_{E,2^\infty}$ of $\rho_{E,2^\infty}$ in $\mathcal{N}_{-1,0}(2^{\infty})$.
    \end{enumerate}
\end{lemma}
\begin{proof}
    By Theorem \ref{thm-j1728alvaro}, the index of $G_{E, 2^{\infty}}$ in $\mathcal{N}_{-1, 0}(2^{\infty})$ is $1$, $2$, or $4$. This proves (1). For (2), Theorem \ref{thm-j1728alvaro} describes all the $2$-adic images that are possible when $\Delta_K f^2=-4$, namely $\mathcal{N}_{-1,0}(2^{\infty})$, or $\langle G_{-1,0}^{2,j},\gamma \rangle$ for $1\leq j \leq 2$ and a certain choice of $\gamma$, or $\langle G_{-1,0}^{4,k},\gamma\rangle$ for $1\leq k\leq 4$. In all cases, a \verb|Magma| (\cite{magma}) computation shows that 
    $$d_{E,8}=[\mathcal{N}_{-1,0}(8):G_{E,8}]=[\mathcal{N}_{-1,0}(\Z_2):G_{E,2^\infty}],$$ i.e., the $2$-adic index equals the mod-$8$ index, as desired.
\end{proof}

For an integer $d\in\Z$, let $E^d/\Q$ be the elliptic curve given by the model $y^2=x^3+dx$. We will let $E=E^1$ for simplicity.

\begin{lemma}\label{lem-8divfielddegreej1728}
    Let $d$ be a non-zero, fourth-power-free integer, and let $E^d/\Q$ be the curve given by $y^2=x^3+dx$. Let $F^d_8 = \Q(E^d[8])$. Then,
    \begin{enumerate}
        \item $[F^d_8:\Q]=16$ if and only if $d\in \{\pm 1, \pm 2, \pm 4, \pm 8 \}.$
        \item $[F^d_8:\Q]=32$ if and only if $d = \pm t^2$ or $\pm 2t^2$ for some other square-free integer $t\neq \pm 1, \pm 2$.
        \item $[F^d_8:\Q]=64$ otherwise.
    \end{enumerate}
\end{lemma}
\begin{proof}
    Let $F=\Q\left(\zeta_{16}, \sqrt{-1+\sqrt{2}}\right)$, which can be easily verified to be a number field of degree $16$ over $\Q$. By Cor. \ref{cor-8divfieldj1728}, the number field $F$ coincides with $\Q(E^1[8])$, the $8$-division field of $y^2=x^3+x$. By Theorem 9.7 and Example 9.8 of \cite{lozano-galoiscm}, we have
    $$\Gal(F/\Q)=\Gal(\Q(E^1[8])/\Q)\cong \left\langle \left(\begin{array}{cc} 5 & 0 \\ 0 & 5\\\end{array}\right), \left(\begin{array}{cc} 1 & 2 \\ -2 & 1\\\end{array}\right), \left(\begin{array}{cc} 0 & -1 \\ -1 & 0\\\end{array}\right) \right\rangle \cong (\Z/2\Z \times \Z/4\Z)\rtimes \Z/2\Z.$$
    Thus, one can verify that $F$ contains precisely three distinct quadratic subfields, namely $\Q(i)$, $\Q(\sqrt{2})$, and $\Q(\sqrt{-2})$. Moreover, the discriminant of $\mathcal{O}_F$ is $2^{54}$ (hence, the only ramified prime is $2$) so if a fourth root $\sqrt[4]{t}$ lies in $F$ for some square-free integer $t$, then $t$ must be a power of $2$, and therefore $t=\pm 1$ or $\pm 2$. Since $\Q(\zeta_{16})\subseteq F$, it follows that $\sqrt[4]{\pm 1}\in F$. In addition, a calculation shows that $\sqrt[4]{\pm 2}\in F$ as well. In particular, if
    $$\Q\left(\zeta_{16}, \sqrt{-1+\sqrt{2}},\sqrt[4]{d}\right) = F(\sqrt[4]{d})=F$$
    for some $4$th-power-free integer $d$, then $d=\pm 1$, $\pm 2$, $\pm 4$, or $\pm 8$. This proves (1).

    For (2), suppose $[F_8^d:\Q]=32$. Then, $F(\sqrt[4]{d})/F$ is quadratic, and therefore $\sqrt{d}\in F$. If we write $d=nt^2$ with $n$ square-free, then $\sqrt{n}\in F$, and our previous paragraph implies that $n=\pm 1$ or $\pm 2$. Hence $d=\pm t^2$ or $\pm 2t^2$, as desired.

    Finally, if $d\not\in \{\pm 1,\pm 2,\pm 4,\pm 8, \pm t^2,\pm 2t^2 : t \in \Z\}$, then our previous work shows that $F(\sqrt[4]{d})/F$ must be of degree $4$, and therefore $[F_8^d:\Q]=4\cdot [F:\Q]=64$. This shows (3) and concludes the proof.
\end{proof}

Before we prove the classification of models with a prescribed $2$-adic image in the $j_{K,f}=1728$ case, we need a lemma about how division fields behave under quadratic twists.

\begin{lemma}\label{lem-twistimage}
    Let $E/F : y^2=x^3+Ax+B$ be an elliptic curve defined over a number field $F$, let $N>2$, and let $G_{E,N}$ be the image of $\rho_{E,N}\colon \Gal(\overline{F}/F)\to \GL(2,\Z/N\Z)$. Let $\alpha \in F$ and let $E^\alpha$ be the quadratic twist of $E$ by $\alpha$, i.e., $E^\alpha : \alpha y^2 = x^3+Ax+B$. Then, 
    \begin{enumerate}
        \item $F(E^\alpha[N])\subseteq F(\sqrt{\alpha},E[N])$. In particular, if $\sqrt{\alpha}\in F(E[N])$, then $F(E^\alpha[N])\subseteq F(E[N])$.
        \item If $\sqrt{\alpha}$ does not belong to $F(E[N])$, then $F(E^\alpha[N])=F(\sqrt{\alpha},E[N])$, and $G_{E^\alpha,N}$ is conjugate to $\langle -\operatorname{Id},G_{E,N}\rangle \subseteq \GL(2,\Z/N\Z)$.
    \end{enumerate}
\end{lemma}

\begin{proof}
    Let $N$, $E$, $\alpha$, and $E^\alpha$ as in the statement. Then, $E$ and $E^\alpha$ are isomorphic over $F(\sqrt{\alpha})$. In particular, $F(\sqrt{\alpha},E[N]) = F(\sqrt{\alpha}, E^\alpha[N])$, and so $F(E^\alpha[N])\subseteq F(\sqrt{\alpha},E[N])$. Moreover, if $\sqrt{\alpha}\in F(E[N])$, then $F(E^\alpha[N])\subseteq F(E[N]).$

    Now suppose that $\sqrt{\alpha}$ does not belong to $F(E[N])$, and choose any point $(x_N,y_N)\in E[N]$. It follows that $(x_N,y_N/\sqrt{\alpha})\in E^\alpha[N]$. If $\sigma \in \Gal(\overline{F}/F)$, then 
    $$\sigma (x_N,y_N/\sqrt{\alpha}) = (\sigma(x_N), \sigma(y_N)/\sigma(\sqrt{\alpha})).$$
    If we define $\chi_\alpha\colon G_F \to \{\pm 1\}$ by $\chi_\alpha(\sigma)=\sigma(\sqrt{\alpha})/\sqrt{\alpha}$, then:
    $$\sigma (x_N,y_N/\sqrt{\alpha}) = (\sigma(x_N), \chi_{\alpha}(\sigma)\cdot \sigma(y_N)/\sqrt{\alpha})=\chi_{\alpha}(\sigma)\cdot (\sigma(x_N),\sigma (y_N)/\sqrt{\alpha}).$$
Moreover, if we fix a $\Z/N\Z$-basis ${P,Q}$ of $E[N]$ and a basis $\{(x(P),y(P)/\sqrt{\alpha}),(x(Q),y(Q)/\sqrt{\alpha}) \}$ of $E^\alpha[N]$, and if we let $\rho_{E,N}(\sigma) =: M(\sigma)\in \GL(2,\Z/N\Z)$, then
$$\rho_{E^\alpha,N}(\sigma) = (\chi_{\alpha}(\sigma)\cdot \operatorname{Id})\cdot M(\sigma).$$
Since $\sqrt{\alpha}$ is not in $F(E[N])$, we may choose a $\sigma$ which fixes $F(E[N])$ but acts on $F(\sqrt{\alpha})$ as $\chi_\alpha$. In this case, $M(\sigma)=\operatorname{Id}$ and therefore:
    $$\sigma (x_N,y_N/\sqrt{\alpha}) = (\sigma(x_N), \sigma(y_N)/\sigma(\sqrt{\alpha})) = (x_N,\chi_{\alpha}(\sigma)\cdot y_N/\sqrt{\alpha})=\chi_{\alpha}(\sigma)\cdot  (x_N,y_N/\sqrt{\alpha}).$$
    In other words, if $\rho_{E,N}(\sigma) = M(\sigma)=\operatorname{Id}\in \GL(2,\Z/N\Z)$, then 
$\rho_{E^\alpha,N}(\sigma) = \chi_{\alpha}(\sigma)\cdot \operatorname{Id}.$ Since $\sigma$ exists that acts trivially  on $F(E[N])/F$ and non-trivially on $F(\sqrt{\alpha})/F$, it follows that $-\operatorname{Id}\in \rho_{E^\alpha,N}(G_F)=G_{E^\alpha,N}$, and $\sqrt{\alpha}\in F(E^\alpha[N])$ in this case. It follows that:
$$F(E^\alpha[N]) = F(\sqrt{\alpha},E^\alpha[N])=F(\sqrt{\alpha},E[N]).$$
Further, since $\sqrt{\alpha}$ is not in $F(E[N])$, it follows that $\Gal(F(\sqrt{\alpha},E[N])/F)\cong \Gal(F(\sqrt{\alpha})/F)\times \Gal(F(E[N])/F)$ and so $\rho_{E^\alpha,N}(G_F)$ is generated by $-\operatorname{Id}$ and $\rho_{E,N}(G_F)=G_{E,N}$, i.e., 
$$G_{E^\alpha,N} = \langle -\operatorname{Id}, G_{E,N} \rangle$$
as claimed.
\end{proof}

\begin{thm}\label{j1728ver }\label{thm-j1728}
    Let $d$ be a non-zero fourth-power-free integer, and let $E/\Q: y^2=x^3+dx$ be an elliptic curve with CM by the order $\Of=\Z[i]$ of discriminant $\Delta_K f^2=-4$ (i.e., $j(E)=1728$). Then, there are $13$ possible $2$-adic images, according to whether $d = \pm 1, \pm 2, \pm 4, \pm 8, \pm t^2, \pm 2t^2, t$, with $t\neq \pm (\Q^\times)^2, \pm 2(\Q^\times)^2$. Moreover, the image labels corresponding to each possibility are precisely as specified in Table \ref{cm2big}.
\end{thm}
\begin{proof}
 Let $E/\Q$ be an elliptic curve with $j(E)=1728$. Then, there is some fourth-power-free integer $d$ such that $E$ is isomorphic to  $E^d : y^2=x^3+dx$. Let $G_{E^d}=G_{E^d,2^\infty}$ be the $2$-adic image of $\rho_{E,2^\infty}$, as a subgroup of  $\mathcal{N}_{-1,0}(2^{\infty})\subseteq \GL(2,\Z_2)$. By Theorem \ref{thm-j1728alvaro}, we have $[\mathcal{N}_{-1,0}(2^{\infty}): G_{E^d}]=1,2$, or $4$.
 
 Suppose first that $[\mathcal{N}_{-1,0}(2^{\infty}): G_{E^d}]=4$. Then, Lemma \ref{lem-j1728index} shows that the index of the mod-$8$ image is also $4$ inside $\mathcal{N}_{-1,0}(8)$, and Lemma \ref{lem-8divfielddegreej1728} shows that $d\in \{\pm 1, \pm 2, \pm 4, \pm 8\}.$ Using code by Drew Sutherland, we have verified (see \cite[Example 9.8]{lozano-galoiscm}) that the $2$-adic image of $E^d$ is generated by $H\subseteq \cC_{-1,0}(2^{\infty})$ and $\gamma$, where the triples $(d,H,\gamma)$ are given by this list (using notation from Theorem \ref{thm-j1728alvaro}): 
 $$
 \begin{array}{cccc}
 (-4,G^{\,4,1}_{-1,0},c_{-1}), &%verified
	 (1,G^{\, 4,1}_{-1,0},c_{-1}'),& %verified
	   (-1,G^{\,4,2}_{-1,0},c_{-1}), &%verified
	    (4,G^{\,4,2}_{-1,0},c_{-1}'), %verified
	    \\
	     (2,G^{\,4,3}_{-1,0},c_{1}'), &%verified
	(-2,G^{\,4,4}_{-1,0},c_{-1}),& %verified
	 (8,G^{\,4,4}_{-1,0},c_{1}'),& %verified
	 (-8,G^{\,4,3}_{-1,0},c_{-1}). %verified
	\end{array}
	$$ 
 An additional \verb|Magma| calculation shows that the groups $\langle H,\gamma\rangle$ coincide with those given by the labels in Table \ref{cm2big}. 

 Now suppose that $[\mathcal{N}_{-1,0}(2^{\infty}): G_{E^d}]=2$. Then, Lemmas \ref{lem-j1728index} and \ref{lem-8divfielddegreej1728} show that $[F_8^d:\Q]=32$ and $d=\pm t^2$ or $d=\pm 2t^2$ for some square-free integer $t\in\Z$, with $t\neq \pm 1$ or $\pm 2$. Hence, $E^d$ is a quadratic twist of $E^1$, $E^{-1}$, $E^{2}$, or $E^{-2}$.
 \begin{itemize}
     \item Let $E=E^{t^2}$ which is a quadratic twist of $E^1$. As in the proof of Lemma \ref{lem-8divfielddegreej1728}, the only quadratic subfields of $\Q(E^1[8])$ are $\Q(i)$, $\Q(\sqrt{2})$, and $\Q(\sqrt{-2})$. Thus, if $t\neq \pm 1, \pm 2$, and $t$ is square-free, then $\Q(\sqrt{t})\cap \Q(E^1[8])=\Q$. It follows from Lemma \ref{lem-twistimage} that 
     $$G_{E^d,8} = \langle G_{E^1,8}, -\operatorname{Id}\rangle = \langle G^{\, 4,1}_{-1,0},c'_{-1},-\operatorname{Id}\rangle = \langle G^{\, 2,1}_{-1,0},c'_{-1}\rangle$$
     which corresponds to the label \texttt{2.2.ns7-4.2.1}.
     \item Let $E=E^{-t^2}$ which is a quadratic twist of $E^{-1}$. As in the previous bullet point, it follows from Lemma \ref{lem-twistimage} that 
     $$G_{E^d,8} = \langle G_{E^{-1},8}, -\operatorname{Id}\rangle = \langle G^{\, 4,2}_{-1,0},c_{-1},-\operatorname{Id}\rangle = \langle G^{\, 2,1}_{-1,0},c_{-1}\rangle$$
     which corresponds to the label \texttt{2.2.ns7-4.2.2}.
     \item Let $E=E^{2t^2}$ which is a quadratic twist of $E^{2}$. As in the previous bullet point, it follows from Lemma \ref{lem-twistimage} that 
     $$G_{E^d,8} = \langle G_{E^{2},8}, -\operatorname{Id}\rangle = \langle G^{\, 4,3}_{-1,0},c_{1}',-\operatorname{Id}\rangle = \langle G^{\, 2,2}_{-1,0},c_{1}'\rangle= \langle G^{\, 2,2}_{-1,0},c_{-1}'\rangle$$
     which corresponds to the label \texttt{2.2.ns7-8.2.1}.
     \item Let $E=E^{-2t^2}$ which is a quadratic twist of $E^{-2}$. As in the previous bullet point, it follows from Lemma \ref{lem-twistimage} that 
     $$G_{E^d,8} = \langle G_{E^{-2},8}, -\operatorname{Id}\rangle = \langle G^{\, 4,4}_{-1,0},c_{-1},-\operatorname{Id}\rangle = \langle G^{\, 2,2}_{-1,0},c_{-1}\rangle= \langle G^{\, 2,2}_{-1,0},c_{1}\rangle$$
     which corresponds to the label \texttt{2.2.ns7-8.2.2}.
 \end{itemize}
 Finally, if $[\mathcal{N}_{-1,0}(2^{\infty}): G_E]=1$, it follows from Lemmas \ref{lem-j1728index} and \ref{lem-8divfielddegreej1728} that $E$  is isomorphic to $E^d$ with $d\not\in \{\pm 1,\pm 2,\pm 4,\pm 8, \pm t^2,\pm 2t^2 : t \in \Z\}$, and the image is as large as possible, i.e., $G_E=\mathcal{N}_{-1,0}(2^{\infty})$ which corresponds to the label \texttt{2.2.ns7-1.1.1}.
\end{proof}

\subsubsection{$\Delta_Kf^2=-8$, $\ell = 2$}\label{sec-8ell2} Our goal is to prove the classification $2$-adic images of elliptic curves with CM by the order $\Z[\sqrt{-2}]$ of discriminant $-8$ (i.e., $j_{K,f}=8000$), which is summarized in Figure \ref{fig-8ell2} and also detailed in Table \ref{tab-ell2}. Our first preliminary result describes the $2$-, $4$-, and $8$- division fields.

\begin{figure}[ht]
 \scalebox{.88}{
$$
\xymatrix@R=1mm@C=1mm{ 
& &  \fbox{1.1.1}
\ar@/_1pc/@{-}[ddll]_2
\ar@{-}[ldd]_2 
\ar@/^1pc/@{-}[ddrr]^2
\ar@{-}[rdd]^2 
\\
\\
\fbox{$\begin{array}{c} 16.2.1\\ 2\end{array}$}  &  
\fbox{$\begin{array}{c} 16.2.2\\ 1\end{array}$}  & 
  & 
  \fbox{$\begin{array}{c} 16.2.3\\ -1\end{array}$}  &  
  \fbox{$\begin{array}{c} 16.2.4\\ -2\end{array}$}   }
$$
}
\caption{The labels for the $2$-adic images of curves $y^2=x^3 - 4320d^2 x + 96768d^3$ with $j=8000$. For each value of $d$ indicated in the second row of each box, the $2$-adic label starts with \texttt{2.3.ns7-} and ends with the digits indicated in the first row of the box.}\label{fig-8ell2}
\end{figure}

%{\color{red}
%Let $E: y^2 = x^3 + Ax + B$ be an elliptic curve defined over $\mathbb{Q}$, and let $E^d: y^2 = x^3 + d^2Ax + d^3B$ be a quadratic twist, for some $d\in \Q^*$ square-free. It can be directly verified that the $m$-division polynomials $\Psi_m$ associated with $E$ and $\Psi_m^d$ corresponding to $E^d$ satisfy the relation:
%\[
%    \Psi_m^d(dx) = d^{\deg \Psi_m} \Psi_m(x).
%\]
%Thus, analyzing the division field $\Q(E^d[m])$ reduces to studying the division field $\Q(E[m])$.
%}

\begin{remark}\label{rem-howtocomputedivisionfield}
    Let $E: y^2=f(x)$ be an elliptic curve over $\Q$, let $d\in\Q^\ast$ be a square-free rational number, and let $E^d : dy^2=f(x)$ be the quadratic twist of $E$ by $d$. There is an isomorphism $\phi\colon E\to E^d$ given by $(x,y)\mapsto (x,y/\sqrt{d})$ defined over $\Q(\sqrt{d})$, and therefore, for any $N\geq 1$, we have an equality of division fields $\Q(x(E[N]))=\Q(x(E^d[N]))$. Moreover, suppose that $E[N]=\langle (x_1,y_1), (x_2,y_2) \rangle$. Then, $E^d[N]=\langle (x_1,y_1/\sqrt{d}),(x_2,y_2/\sqrt{d})\rangle $ and $$\Q(E[N])=\Q(x(E[N]),\sqrt{df(x_1)},\sqrt{df(x_2)}).$$
    We will make use of this equation in several of our results below, namely Prop. \ref{prop-8divfieldm8} and \ref{prop-8divfieldm16}.
\end{remark}

\begin{proposition}\label{prop-8divfieldm8}
Let $d\in \Q^*$ square-free, and let $E^d\,:\,y^2=x^3 - 4320d^2 x + 96768d^3$. Then
\begin{itemize}
\item  $\Q(E^d[2]) = \Q(\sqrt{2})$.
%\item  $\Q(E^d[4]) = \Q(\zeta_4,\sqrt{d}\sqrt[4]{2},\sqrt{d}\sqrt{2+\sqrt{2}})$.
\item  $\Q(E^d[4])=\Q(\alpha,\sqrt{d}\sqrt[4]{2})$ such that $f_4(\alpha)=0$, where $f_4(x)=x^8+6x^4+1$,
 and
	\item $\Q(E^d[8]) = \Q(\beta,\sqrt{d})$ such that $f_8(\beta)=0$, where $f_8(x)=x^{32} + 16x^{31} + 128x^{30} + 672x^{29} +  2544x^{28} + 7200x^{27} + 15352x^{26} + 24272x^{25} + 26904x^{24} + 17312x^{23} - 304x^{22} - 11984x^{21} - 9672x^{20} - 2720x^{19} - 3592x^{18} - 7552x^{17} -     2798x^{16} + 6224x^{15} + 6368x^{14} - 672x^{13} - 2224x^{12} + 3360x^{11} +     4952x^{10} - 1072x^9 - 4600x^8 - 1120x^7 + 1776x^6 + 752x^5 - 264x^4 -     96x^3 + 24x^2 + 1
$.
\end{itemize}
In particular, we have $[\Q(E^d[2]):\Q]=2$ and $[\Q(E^d[4]):\Q]=16$, and we have $[\Q(E^d[8]):\Q]=32$ if and only if $d \in\{\pm 1,\pm2\}$; otherwise $[\Q(E^d[8]) : \Q] =64$.

\end{proposition}

\begin{proof}
    Let $f(x) = x^3 - 4320x + 96768$. This polynomial factors as 
    \[ f(x) = (x-48)(x + 24 - 36 \sqrt{2})(x + 24 + 36 \sqrt{2}) \]
    over $\Q(\sqrt{2})$. Thus, $\Q(E[2]) = \Q(E^d[2]) = \Q(\sqrt{2})$, which has degree $2$. Next, let $\psi_N(x)$ denote the $N$th division polynomial for $E$. Then 
    \begin{align*}
        \psi_4(x)/2 \psi_2(x) 
        &= (x^2 - 96x - 288)(x^4 + 96x^3 - 12096x^2 + 801792x - 19823616) \\
        &=: g_4(x) \, h_4(x).
    \end{align*}
The quadratic polynomial $g_4(x)$ factors as $g_4(x) = (x - \alpha_1)(x - \alpha_2)$ where $\alpha_1 = 48 + 36 \sqrt{2}$ and $\alpha_2 = 48 - 36 \sqrt{2}$. Let $K_4$ be the splitting field for $h_4(x)$, and suppose $h_4(x)$ factors as $\prod_{j=1}^{4} (x - \beta_j)$, with $\beta_j \in K_4$. By computation, we find that there exist non-square units $u_{\alpha_i} \in \OO_{\Q(\sqrt{2})}$ and $u_{\beta_j} \in \OO_{K_4}$ such that 
\begin{align*}
    f(\alpha_i) = u_{\alpha_i} 3^{6} 2^{8} \sqrt{2} \\
   f(\beta_j) = u_{\beta_j} 3^6 2^8 \gamma^4,
\end{align*}
where $2 \OO_{K_4} = \langle \gamma \rangle^8$.
Again, by computation, we find that
\begin{itemize}
    \item $\sqrt{2} \in K_4$,
    \item $u_{\alpha_i}$ is a square in $\OO_{K_4}$ for $i \in \{1, 2\}$,
    \item $u_{\beta_j}/\sqrt{2}$ is a square in $\OO_{K_4}$ for $1 \leq j \leq 4$.
\end{itemize}
Thus, Remark \ref{rem-howtocomputedivisionfield} shows that  $\Q(x(E^d[4])) = K_4$ and
\[ \Q(E^d[4]) = K_4 \left(\sqrt{d \sqrt{2}}, \sqrt{d u_{\beta_j}} \right) = K_4\left(\sqrt{d} \sqrt[4]{2}\right). \]
Next, we compute the $8$-division field. We find that $\psi_8(x)/2 \psi_4(x) = g_8(x) h_8(x)$, where $g_8(x)$ and $h_8(x)$ are irreducible polynomials of degree $8$ and $16$, respectively. Let $K_{8,g}$ and $K_{8,h}$ be the splitting fields of $g_8(x)$ and $h_8(x)$, which are of degrees $16$ and $32$, respectively.

We find that $K_{8,g}$ is a subfield of $K_{8,h}$, and that for each $\gamma \in K_{8,h}$ where $g_{8}(\gamma) = 0$ or $h_8(\gamma) = 0$, $f(\gamma)$ is a square in $K_{8,h}$. We also note that $\sqrt[4]{2} \in K_{8,h}$, and that the only quadratic fields contained in $K_{8, h}$ are $\Q(\sqrt{2}), \Q(\sqrt{-2})$, and $\Q(\sqrt{-1})$, so
\[ \Q(E^{d}[8]) = 
\begin{cases}
K_{8,h}(\sqrt{d}) & \text{ if } d \neq \pm 1, \pm 2 \\
K_{8,h} & \text{ if } d = \pm 1, \pm 2
\end{cases}, \qquad \text{and} \qquad [\Q(E^{d}[8]) : \Q] = 
\begin{cases}
64 & \text{ if } d \neq \pm 1, \pm 2, \\
32 & \text{ if } d = \pm 1, \pm 2,
\end{cases} \]
as we wanted to prove.
\end{proof}

\begin{lemma}\label{lem-disc8index}
    Let $E/\Q$ be an elliptic curve with $j(E)=8000$, and so CM discriminant $\Delta_K f^2 = -8$. Let $G_{E,8}$ be the image of $\rho_{E,8}$, and let $d_{E,8}=[\mathcal{N}_{-2,0}(8):G_{E,8}]$. Then:
    \begin{enumerate}
        \item $d_{E,8}=1$ or $2$, and
        \item $d_{E,8}$ coincides with the index of the image $G_{E,2^\infty}$ of $\rho_{E,2^\infty}$ in $\mathcal{N}_{-2,0}(2^{\infty})$.
    \end{enumerate}
\end{lemma}
\begin{proof}
    Theorem \ref{thm-m8and16alvaro} says that the index of $G_{E, 2^{\infty}}$ in $\mathcal{N}_{-2,0}(2^{\infty})$ is either $1$ or $2$. This proves (1). For (2), Theorem \ref{thm-m8and16alvaro} describes all the $2$-adic images that are possible when $\Delta_K f^2=-8$, namely $\mathcal{N}_{-2,0}(2^{\infty})$, or $\langle G_{-2,0}^{2,j},c_\varepsilon \rangle$ for $3\leq j \leq 4$ and $\varepsilon=\pm 1$. In all cases, a \verb|Magma| computation shows that 
    $$d_{E,8}=[\mathcal{N}_{-2,0}(8):G_{E,8}]=[\mathcal{N}_{-2,0}(2^{\infty}):G_{E,2^\infty}],$$ i.e., the $2$-adic index equals the mod-$8$ index, as desired.
\end{proof}

\begin{theorem}\label{disc8ell2-ver} Let $d$ be a non-zero square-free integer, and let $E/\Q: y^2=x^3 - 4320d^2 x + 96768d^3$ be an elliptic curve with CM by the order $\Of=\Z[\sqrt{-2}]$ of discriminant $\Delta_K f^2=-8$ (i.e., $j(E)=8000$). Then, there are $5$ possible $2$-adic images, according to whether $d = \pm 1, \pm 2, t$, with $t \notin \pm (\Q^\times)^2, \pm 2(\Q^\times)^2$. Moreover, the image labels corresponding to each possibility are precisely as specified in Table \ref{cm2big}.
\end{theorem}

\begin{proof}
    Let $E/\Q$ be an elliptic curve with $j(E) = 8000$, equivalently an elliptic curve with CM by an order of discriminant $-8$. Then, there is a square-free integer $d$ such that $E$ is isomorphic to $E^d/\Q : y^2=x^3 - 4320d^2 x + 96768d^3$. Let $G_{E^d} = G_{E^d,2^\infty}$ be the $2$-adic image of $\rho_{E, 2^{\infty}}$ in $\mathcal{N}_{-2,0}(2^{\infty})$.

    Suppose first that $[\mathcal{N}_{-2,0}(2^{\infty}) : G_{E^d} ] = 2$. Then, Lemma \ref{lem-disc8index}(2) shows that the index of the mod 8 image is also $2$ inside $\mathcal{N}_{-2,0}(8)$, and Proposition \ref{prop-8divfieldm8} shows that $d \in \{ \pm 1, \pm 2 \}$. We have verified using \verb|Magma| that the $2$-adic images of $E^d$ are generated by $H \subseteq \mathcal{C}_{-2, 0}(2^{\infty})$ and image of complex conjugation $\gamma$, where the tuples $(d, H, \gamma)$ are of the form $(2, G^{2,3}_{-2,0}(2^{\infty}), c_1)$, $(1, G^{2,4}_{-2,0}(2^{\infty}), c_1)$, $(-1, G^{2,3}_{-2,0}(2^{\infty}), c_{-1})$, $(-2, G^{2,4}_{-2,0}(2^{\infty}), c_{-1})$. We also verify computationally that the labels corresponding to these images in Table \ref{cm2big} are correct.

    If instead $[N_{-2, 0}(2^{\infty}) : G_{E^d}] = 1$, then by Proposition \ref{prop-8divfieldm8} we have that $d \notin \{ \pm 1, \pm 2 \}$, and $G_E = \mathcal{N}_{-2, 0}(2^{\infty})$ with label 2.3.ns7-1.1.1. 
\end{proof}

\subsubsection{$\Delta_Kf^2=-16$, $\ell = 2$}\label{sec-16ell2} Here we prove the classification $2$-adic images of elliptic curves with CM by the order $\Z[2i]$ of discriminant $-16$ (i.e., $j_{K,f}=287496$), which is summarized in Figure \ref{fig-16ell2} and also detailed in Table \ref{tab-ell2}. Our first preliminary result describes the $2$-, $4$-, and $8$- division fields.

\begin{figure}[ht]
 \scalebox{.88}{
$$
\xymatrix@R=1mm@C=1mm{ 
& &  \fbox{1.1.1}
\ar@/_1pc/@{-}[ddll]_2
\ar@{-}[ldd]_2 
\ar@/^1pc/@{-}[ddrr]^2
\ar@{-}[rdd]^2 
\\
\\
\fbox{$\begin{array}{c} 8.2.1\\ 2\end{array}$}  &  
\fbox{$\begin{array}{c} 8.2.2\\ 1\end{array}$}  & 
  & 
  \fbox{$\begin{array}{c} 8.2.3\\ -1\end{array}$}  &  
  \fbox{$\begin{array}{c} 8.2.4\\ -2\end{array}$}   }
$$
}
\caption{The labels for the $2$-adic images of curves $y^2=x^3-11d^2x+14d^3$ with $j=287496$. For each value of $d$ indicated in the second row of each box, the $2$-adic label starts with \texttt{2.4.ns7-} and ends with the digits indicated in the first row of the box.}\label{fig-16ell2}

\end{figure}
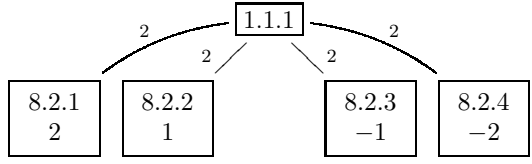
\begin{proposition}\label{prop-8divfieldm16}
Let $d\in \Q^*$ be square-free, and let $E^d\,:\,y^2=x^3 - 11 d^2x + 14d^3$. Then
\begin{itemize}
\item  $\Q(E^d[2]) = \Q(\sqrt{2})$,
%\item  $\Q(E^d[4]) = \Q(\zeta_4,\sqrt{d},\sqrt[4]{2})$.
\item  $\Q(E^d[4])=\Q(\alpha,\sqrt{d})$ such that $f_4(\alpha)=0$, where $f_4(x)=x^8-4x^6+8x^4-4x^2+1$,
 and
	\item $\Q(E^d[8]) = \Q(\beta,\sqrt{d})$ such that $f_8(\beta)=0$, where $f_8(x)=x^{32} + 16x^{31} + 120x^{30} + 560x^{29} + 1848x^{28} + 4784x^{27} + 11000x^{26} + 25344x^{25} + 59844x^{24} + 133856x^{23} +  260768x^{22} + 419392x^{21} + 534920x^{20} + 513536x^{19} + 332032x^{18} + 93856x^{17} - 43548x^{16} - 22112x^{15} + 61056x^{14} + 77728x^{13} + 20768x^{12} - 18304x^{11} + 320x^{10} + 21440x^{9} + 8240x^{8} - 8256x^{7} - 1888x^{6} + 3584x^{5} + 800x^{4} - 1216x^{3} + 320x^{2} + 8
$.
\end{itemize}
In particular, we have $[\Q(E^d[2]):\Q]=2$ and $[\Q(E^d[4]):\Q]=8$, and  we have $[\Q(E^d[8]):\Q]=32$ if and only if $d \in\{\pm 1,\pm2\}$; otherwise $[\Q(E^d[4]):\Q]=16$ and $[\Q(E^d[8]) : \Q] =64$.
\end{proposition}

\begin{proof}
Let $f(x) := x^3 - 11x + 14$. This polynomial factors as
\[ (x - 2)(x + 1 - 2\sqrt{2})(x + 1 + 2 \sqrt{2}) \]
over $\Q(\sqrt{2})$, so $\Q(E[2]) = \Q(E^{d}[2]) = \Q(\sqrt{2})$. Next, define $\psi_N(x)$ to be the $N$th division polynomial of $E$. We find that 
\begin{align*}
    \frac{\psi_4(x)}{2\psi_2(x)} 
    &=: (x-1)(x-3)g_4(x),
\end{align*}
where $g_4(x)$ is an irreducible degree 4 polynomial. Let $K_4$ be the splitting field of $g_4(x)$, which we verify through computation is a degree $8$ extension of $\Q$, and let $\alpha_1, \dotsc, \alpha_4$ be the roots of $g_4(x)$. We find that $f(\alpha_i)$ is a square in $\OO_{K_4}$ for all $i$, and that $\Q(\sqrt{2}), \Q(\sqrt{-2}),$ and $\Q(\sqrt{-1})$ are subfields of $K_4$, thus Remark \ref{rem-howtocomputedivisionfield} shows that 
\[ \Q(E^{d}[4]) = 
\begin{cases}
K_{4}(\sqrt{d}) & \text{ if } d \neq \pm 1, \pm 2 \\
K_{4} & \text{ if } d = \pm 1, \pm 2
\end{cases}, \qquad \text{and} \qquad [\Q(E^{d}[4]) : \Q] = 
\begin{cases}
16 & \text{ if } d \neq \pm 1, \pm 2, \\
8 & \text{ if } d = \pm 1, \pm 2.
\end{cases} \]
Next, we find that 
\[ \frac{\psi_8(x)}{2 \psi_4(x)} =: g_8(x) h_8(x) j_8(x),\]
where $g_8(x), h_8(x),$ and $j_8(x)$ are irreducible polynomials of degree $4, 4$, and $16$, respectively. Let $K_8$ denote the splitting field of $j_8(x)$, a field of degree $32$. We find that $g_8(x), h_8(x)$ and $j_8(x)$ split in $K_8$, and that for every $\gamma \in K_8$ that is a root of $g_8(x), h_8(x),$ or $j_8(x)$, we also have that $f(\gamma)$ is a square in $\OO_{K_8}$. In addition, the only quadratic subfields of $K_8$ are $\Q(\sqrt{-1}), \Q(\sqrt{2}),$ and $ \Q(\sqrt{-2})$. Thus,
\[ \Q(E^{d}[8]) = 
\begin{cases}
K_{8}(\sqrt{d}) & \text{ if } d \neq \pm 1, \pm 2 \\
K_{8} & \text{ if } d = \pm 1, \pm 2
\end{cases}, \qquad \text{and} \qquad [\Q(E^{d}[8]) : \Q] = 
\begin{cases}
64 & \text{ if } d \neq \pm 1, \pm 2, \\
32 & \text{ if } d = \pm 1, \pm 2,
\end{cases} \]
as desired.
\end{proof}

\begin{lemma}\label{lem-disc16index}
    Let $E/\Q$ be an elliptic curve with CM by an order of discriminant  $\Delta_K f^2 = -16$ and, therefore, $j(E)=287496$. Let $G_{E,8}$ be the image of $\rho_{E,8}$, and let $d_{E,8}=[\mathcal{N}_{-4,0}(8):G_{E,8}]$. Then:
    \begin{enumerate}
        \item $d_{E,8}=1$ or $2$, and
        \item $d_{E,8}$ coincides with the index of the image $G_{E,2^\infty}$ of $\rho_{E,2^\infty}$ in $\mathcal{N}_{-4,0}(2^{\infty})$.
    \end{enumerate}
\end{lemma}
\begin{proof}
  Theorem \ref{thm-m8and16alvaro} says that the index of $G_{E, 2^{\infty}}$ in $\mathcal{N}_{-4,0}(2^{\infty})$ is either $1$ or $2$. This proves (1). For (2), Theorem \ref{thm-m8and16alvaro} describes all the $2$-adic images that are possible when $\Delta_K f^2=-16$, namely $\mathcal{N}_{-4,0}(2^{\infty})$, or $\langle G_{-4,0}^{2,j},c_\varepsilon \rangle$ for $1\leq j \leq 4$ and $\varepsilon=\pm 1$. In all cases, a \verb|Magma| computation shows that 
    $$d_{E,8}=[\mathcal{N}_{-4,0}(8):G_{E,8}]=[\mathcal{N}_{-2,0}(2^{\infty}):G_{E,2^\infty}],$$ i.e., the $2$-adic index equals the mod-$8$ index, as desired.
\end{proof}

\begin{theorem}\label{disc8ell2-ver}\label{thm-16ell2} Let $d\in \Z$ be non-zero square-free integer, and let $E^d:y^2=x^3 - 11 d^2x + 14d^3$ be an elliptic curve with CM by the order $\Of=\Z[2i]$ of discriminant $\Delta_K f^2=-16$ (i.e., $j(E)=287496$). Then, there are $5$ possible $2$-adic images, according to whether $d = \pm 1, \pm 2, t$, with $t \notin \pm (\Q^\times)^2, \pm 2(\Q^\times)^2$. Moreover, the image labels corresponding to each possibility are precisely as specified in Table \ref{cm2big}.
\end{theorem}

\begin{proof}
    Let $E/\Q$ be an elliptic curve with CM by an order of discriminant $-16$, i.e., with $j(E) = 287496$. Then, there is a non-zero square-free integer $d$ such that $E$ is isomorphic to $E^d/\Q : y^2=x^3 - 11 d^2x + 14d^3$. Let $G_{E^d} = G_{E^d,2^\infty}$ be the $2$-adic image of $\rho_{E, 2^{\infty}}$ in $\mathcal{N}_{-4,0}(2^{\infty})$.

    Suppose first that $[\mathcal{N}_{-4,0}(2^{\infty}) : G_{E^d} ] = 2$. Then, Lemma \ref{lem-disc16index} shows that the index of the mod $8$ image is also $2$ inside $\mathcal{N}_{-4,0}(8)$, and Proposition \ref{prop-8divfieldm16} shows that $d \in \{ \pm 1, \pm 2 \}$. We have verified using \verb|Magma| that the $2$-adic images for $E^d$ are generated by $H \subseteq \mathcal{C}_{-4, 0}(2^{\infty})$ and image of complex conjugation $\gamma$, where the tuples $(d, H, \gamma)$ are of the form $(2, G^{2,1}_{-4,0}(2^{\infty}), c_1)$, $(1, G^{2,2}_{-4,0}(2^{\infty}), c_1)$, $(-1, G^{2,2}_{-4,0}(2^{\infty}), c_{-1})$, $(-2, G^{2,1}_{-4,0}(2^{\infty}), c_{-1})$. We also verify computationally that the labels corresponding to these images in Table \ref{cm2big} are correct.

    If instead $[N_{-2, 0}(2^{\infty}) : G_{E^d}] = 1$, then $G_{E^d}$ is as large as possible, i.e., $G_{E^d} = \mathcal{N}_{-4,0}(2^{\infty})$ with label \texttt{2.4.ns7-1.1.1} and $d \notin \{ \pm 1, \pm 2 \}$ modulo squares, as claimed.
\end{proof}

\subsection{The case of $\Delta_K f^2= -3$}\label{sec-delta3}

 Next we deal with the case when $E/\Q(j_{K,f})$ is an elliptic curve with CM by $\Of=\Z[(-1+\sqrt{-3})/2]$, i.e., $\Delta_K f^2 = -3$ and $j(E)=0$.

\subsubsection{$\Delta_K f^2 = -3$ and $\ell=2$}\label{sec-delta3ell2} By Theorem \ref{thm-jzero-goodredn}, there are two possibilities for the $2$-adic image of an elliptic curve $E/\Q$ with $j(E)=0$. Moreover, the $2$-adic image of $E$ is defined mod $2$, and $G_{E,2^\infty} = \mathcal{N}_{-1,1}(2^\infty)$ which corresponds to the label \texttt{2.0.ns5-1.1.1}, or $[\mathcal{N}_{-1,1}(2^\infty):G_{E,2^\infty}]=3$ and $G_{E,2^\infty} =\left\langle \cC_{-1,1}(2^\infty)^3, c_{\varepsilon}' \right\rangle$ which corresponds to the label \texttt{2.0.ns5-2.3.1}. Finally, by Prop. \ref{prop-zywina2}(iii), the index is $3$ if $E$ is given by $y^2=x^3+t^3$ for some non-zero square-free integer $t$ and the index is $1$ if $E$ is given by $y^2=x^3+t$ for some integer $t$ that is not a cube. Since all of our results are phrased in terms of twists of a chosen curve $E_\mathcal{O}$, we note that in this case $E_\mathcal{O}:y^2=x^3+16$ so, accordingly, the index-$3$ image  \texttt{2.0.ns5-2.3.1} occurs for curves $y^2=x^3+16\cdot (4t^3)$ and the index-$1$ image \texttt{2.0.ns5-1.1.1} occurs for $y^2=x^3+16\cdot t$ with $t\notin 4(\Q^\times)^3$, as we write in Table \ref{cm2big} and Figure \ref{fig-3ell2}.

\begin{figure}[ht]
 \scalebox{.88}{
$$
\xymatrix@R=1mm@C=1mm{ 
 \fbox{1.1.1}
\ar@{-}[ddd]^3
\\
\\
\\
\fbox{$\begin{array}{c} 2.3.1\\ 4 t^3\end{array}$} 
}
$$
}
\caption{The labels for the $2$-adic images of curves $y^2=x^3+16d$ with $j=0$. For each value of $d$ indicated in the second row of each box, the $2$-adic label starts with \texttt{2.0.ns5-} and ends with the digits indicated in the first row of the box.}\label{fig-3ell2}
\end{figure}
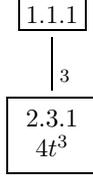

\subsubsection{$\Delta_Kf^2=-3$, $\ell = 3$}\label{sec-delta3ell3}

\begin{figure}[ht]
{\tiny
$$
\xymatrix@R=1mm@C=1mm{ 
&       & &       &  & \fbox{1.1.1}\ar@{-}[dd]_3\ar@{-}[dll]_2\ar@{-}[drr]^2 \ar@/_2pc/@{-}[ddllll]_3 \ar@/^2pc/@{-}[ddrrrr]^3   & &       & & & \\
&       & & \fbox{$\begin{array}{c}9.2.1\\ t^2\end{array}$}\ar@/_4pc/@{-}[ddlll]_3\ar@{-}[ddr]_3\ar@/^3pc/@{-}[ddrrrrr]^3  &  &        & & \fbox{$\begin{array}{c} 9.2.2\\ -3 t^2\end{array}$}\ar@/^4pc/@{-}[ddrrr]^3\ar@{-}[ddl]_3\ar@/_3pc/@{-}[ddlllll]_3 & & & \\
&  \fbox{$\begin{array}{c}3.3.1\\ t^3\end{array}$}\ar@{-}[ld]_2\ar@{-}[rd]^2 & &       &  &  \fbox{$\begin{array}{c}27.3.1\\ 3t^3\end{array}$}\ar@{-}[ld]_2\ar@{-}[rd]^2 & &       & &  \fbox{$\begin{array}{c}27.3.2\\ 9t^3\end{array}$}\ar@{-}[ld]_2\ar@{-}[rd]^2 & \\
 \fbox{$\begin{array}{c}9.6.1\\ 1\end{array}$} & &  \fbox{$\begin{array}{c}9.6.2\\ -27\end{array}$} & &  \fbox{$\begin{array}{c}27.6.1\\ 81\end{array}$} & &  \fbox{$\begin{array}{c}27.6.3\\ -3\end{array}$} & &  \fbox{$\begin{array}{c}27.6.2\\ 9\end{array}$} & &  \fbox{$\begin{array}{c}27.6.4 \\ -243\end{array}$} 
}
$$
}
\caption{The labels for the $3$-adic images of curves $y^2=x^3+16d$ with $j=0$. For each value of $d$ indicated in the second row of each box, the $3$-adic label starts with \texttt{3.1.ns5-} and ends with the digits indicated in the first row of the box.}\label{fig-3ell3}
\end{figure}
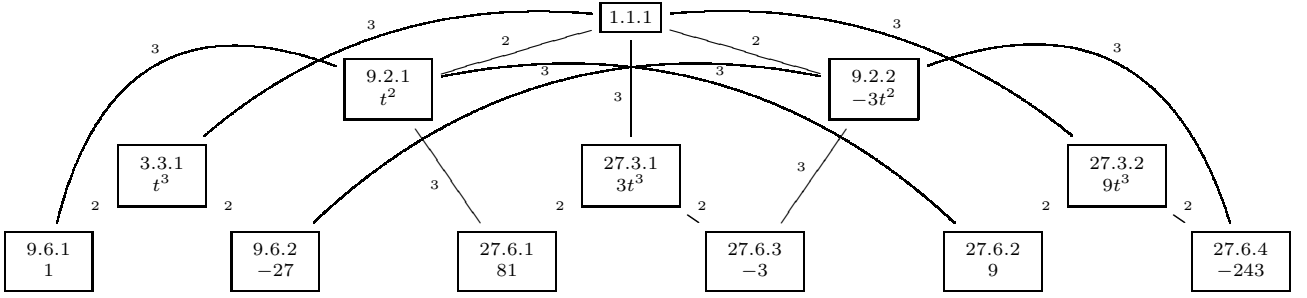

As in previous sections, our first step is to analyze the appropriate division fields, in this case the $3$rd and $9$th division fields.

\begin{proposition}
\label{lem-3and9divfielddegreej0}
Let $d\in \Q^*$ such that $16d$ is $6$th-power-free, and let $E^d\,:\,y^2=x^3+16 d$. Then
\begin{itemize}
\item  $\Q(E^d[3]) = \Q(\zeta_3,\sqrt[6]{d})$, and
\item  $\Q(E^d[9]) = \Q(\zeta_3,\sqrt[6]{d},\sqrt[3]{3},\zeta_9+\zeta_9^{-1})$.
\end{itemize}
In particular,
    \begin{enumerate}
        \item $[\Q(E^d[3]):\Q]=2$ if and only if $d\in \{1, -27 \}.$
        \item $[\Q(E^d[3]):\Q]=4$ if and only if $d =  t^3$ for some $t \in \Q^*$, $t\ne 1,-3$,
        \item $[\Q(E^d[3]) : \Q]= 6$ if and only if $d=t^2, -3t^2$ for some $t \in \Q^*\smallsetminus\{\pm 1,\pm 3\}$, or
        \item $[\Q(E^d[3]) : \Q] = 12$ otherwise,
    \end{enumerate}
    and
    \begin{enumerate}\setcounter{enumi}{4}
        \item $[\Q(E^d[9]):\Q]=18$ if and only if $d\in \{1, -3, 9, -27, 81, -243 \}.$
        \item $[\Q(E^d[9]):\Q]=36$ if and only if $d = t^3$, or $9t^3$ for some $t \in \Q^*$, $t\ne 1,-3$ and $d = 3 t^3$ for some $t \in \Q^*$, $t\ne -1,3$,
        \item $[\Q(E^d[9]) : \Q]= 54$ if and only if $d = t^2$ or $-3t^2$ for $t \in \Q^* \smallsetminus \{\pm 1, \pm 3, \pm 9\}$ or
        \item $[\Q(E^d[9]) : \Q] = 108$ otherwise.
    \end{enumerate}
\end{proposition}

\begin{proof}
Let $\psi_{N}(x)$ denote the $N$th division polynomial for $E^d$. We find that $\psi_3(x) = 3x(x^3 + 64d)$. A point $P\in E^d$ with $x(P) = 0$ has $y(P) = \pm 4 \sqrt{d}$, so $\sqrt{d} \in \Q(E^d[3]).$

For $\zeta_6$ a primitive $6$th root of unity, $x^3 + 64d = (x-\alpha_1)(x - \alpha_2)(x - \alpha_3)$, where $\alpha_i = 4 \zeta_6^{2i - 1} \sqrt[3]{d}$ for $i = 1, 2, 3$.
Note that $\alpha_2 = 4 \zeta_6^{3} \sqrt[3]{d} = -4\sqrt[3]{d}$, thus $\sqrt[3]{d} \in \Q(E^d[3])$.
Further, $\alpha_2/\alpha_1 = \zeta_6^2 = \zeta_3$, so $\sqrt{-3} \in \Q(E^d[3])$.

Finally, if $x = \alpha_i$, then $y = \pm 4 \sqrt{-3} \sqrt{d}$, thus 
\[ \Q(E^d[3]) = \Q(\sqrt{d},\sqrt{-3},\sqrt[3]{d}) = \Q(\sqrt{-3}, \sqrt[6]{d}),\]
and the values of the field degree $[\Q(E^d[3]) : \Q]$ follow by inspection. Next, we define
\[ \frac{\psi_9(x)}{3 \psi_3(x)} = F_9(x),\]
where $F_9$ is a degree $36$ polynomial. We perform a change of variables to better find a closed form expression for $\Q(E^{d}[9])$. Let $x^3 := 2^6 d z$, so $x = 4 \zeta_3^k \sqrt[3]{dz}$ for $k = 0, 1, 2$, and $x^3 + 16d = 2^4 d (4z + 1)$.

Then $F_9(z) = 2^{72} d^{12} f_9(z) g_9(z)$, where $f_9(z)$ and $g_9(z)$ are irreducible polynomials of degree $3$ and $9$, respectively. We find that $f_9(z)$ has as its splitting field $\Q(\zeta_9)^{+}$, the maximal real subfield of $\Q(\zeta_9)$, which is generated as $\Q(\zeta_9)^{+} = \Q(\zeta_9 + \zeta_9^{-1})$.

Next, let $K_9' \cong K[z]/\langle g_9(z) \rangle$, and let $K_9$ denote the splitting field of $g_9(z)$, with field degrees $9$ and $18$, respectively. By computation, we find that $\Q(\zeta_9)^{+}$ and $\Q(\sqrt[3]{3})$ are index 3 subfields of $K_9'$. Further, for $\gamma \in K_9$ satisfying $f_9(\gamma) = 0$ or $g_9(\gamma) = 0$, we find that $4\gamma + 1$ is a square in $K_9$, and $\gamma$ is a cube in $K_9$. For such a $\gamma$, suppose $\gamma = \beta^3$ and $4 \gamma + 1 = \delta^2$ for some $\beta, \delta \in K_9$. Then $x = 4 \zeta_3^k \beta \sqrt[3]{d}$, and $16d(4\gamma + 1) = 16d\delta^2$, which means $y = \pm 4 \delta \sqrt{d}$.

Finally, we find that $K_9 = \Q(\zeta_9)^{+}(\sqrt[3]{3}, \sqrt{-3})$, so
\begin{align*}
\Q(E^d[9])
&= K_9(\sqrt{d}, \sqrt[3]{d}) \\
&= \Q(\sqrt{-3}, \sqrt[3]{3}, \sqrt{d}, \sqrt[3]{d}, \zeta_9 + \zeta_9^{-1}),\\
&= \Q(\zeta_3, \sqrt[3]{3}, \sqrt[6]{d}, \zeta_9 + \zeta_9^{-1})
\end{align*}
and again, the values of $[\Q(E^d[9]) : \Q]$ follow by inspection. 
\end{proof}

\begin{lemma}\label{lem-j0index}
    Let $E/\Q$ be an elliptic curve with $j(E)=0$. Let $G_{E,9}$ be the image of $\rho_{E,9}$, and let $d_{E,9}=[\mathcal{N}_{-3/4,0}(9):G_{E,9}]$. Then:
    \begin{enumerate}
        \item $d_{E,9}=1$, $2$, $3$, or $6$, and
        \item $d_{E,9}$ coincides with the index of the image $G_{E,3^\infty}$ of $\rho_{E,3^\infty}$ in $\mathcal{N}_{-3/4,0}(3^{\infty})$.
    \end{enumerate}
\end{lemma}
\begin{proof}
    % By Theorem \ref{thm-j0ell3alvaro}, the index of $G_{E, 3^{\infty}}$ in $\mathcal{N}_{-3/4, 0}(3^{\infty})$ is $1$, $2$, $3$, or $6$. The results are verified via Magma computation by reducing the possible $3$-adic images mod $9$.

      Theorem \ref{thm-j0ell3alvaro} says that the index of $G_{E, 3^{\infty}}$ in $\mathcal{N}_{-3/4,0}(3^{\infty})$ is either $1$, $2$, $3$, or $6$. This proves (1). For (2), Theorem \ref{thm-j0ell3alvaro} describes all the $3$-adic images that are possible when $\Delta_K f^2=-3$, namely $\mathcal{N}_{-3/4,0}(3^{\infty})$, or $\langle G_{-3/4,0}^{k,j},c_\varepsilon \rangle$ for $k=2,3,6$ and $1\leq j \leq 3$ and $\varepsilon=\pm 1$. In all cases, a \verb|Magma| computation shows that 
    $$d_{E,9}=[\mathcal{N}_{-3/4,0}(9):G_{E,9}]=[\mathcal{N}_{-3/4,0}(3^\infty):G_{E,3^\infty}],$$ i.e., the $3$-adic index equals the mod-$9$ index, as desired.
\end{proof}

Let $E/\Q : y^2 = x^3 + 16d$ be an elliptic curve with $j(E) = 0$. Following the reasoning used in \cite[Cor. X.5.4.1]{silverman1}, a (sextic) twist of $E$ by $t \in \Q^{\ast}$ corresponds to a $1$-cocycle $\xi_\tau \in H^1(G_{\Q}, \Aut(E)) \cong H^1(G_{\Q}, \mu_6)$ for some $\tau \in \overline{\Q}$, a sixth root of $t$, where $\xi_\tau(\sigma) = \sigma(\tau)/\tau$. If instead we take a cubic twist of $E$ by $t$, corresponding to a sextic twist of $E$ by $t^2$, then instead $\xi_\tau$ is a $1$-cocycle in $H^1(G_{\Q}, \mu_3)$. With this background, we prove the following result.

\begin{lemma}\label{cubic_twist}
Let $E/\Q$ be an elliptic curve with $j(E) = 0$, so $E$ has CM by $\OO_K$, the maximal order of $K = \Q(\sqrt{-3})$.
Let $t \in \Q^{\ast}$ be 3rd-power-free and for $\zeta_3$ a primitive 3rd root of unity, let $\tau = \zeta_3^j \sqrt[3]{t}$, where $j \in \{0, 1, 2\}$. Finally, let $\xi_{\tau} \in H^1(G_{\Q}, \mu_3)$ be a $1$-cocycle defined by $\xi_{\tau}(\sigma) = \sigma(\tau)/\tau$. Then

\begin{enumerate}
    \item For $G_K := \Gal(\overline{K}/K)$, the restriction $\xi_\tau$ to $G_K$ is a cubic character dependent only on $t$.
    \item Let $c \in G_\Q$ be a complex conjugation. Then $\xi_{\tau}(c) = \zeta_3^{k}$ for some $k \in \{0, 1, 2\}$.
\end{enumerate}

\end{lemma}

\begin{proof}
    Let $\sigma \in G_K$. Then, $\sigma$ fixes the elements of $K$, including $\zeta_3$, so
\[ 
    \xi_{\tau} (\sigma) 
   = \frac{\sigma(\zeta_3^{j}\sqrt[3]{t})}{\zeta_3^{j}\sqrt[3]{t}} 
   = \frac{\sigma(\sqrt[3]{t})}{\sqrt[3]{t}},
\]
which is a cubic character dependent on $t$.

Let $c \in G_{\Q}$ be a complex conjugation, and let $\zeta_6$ be a 6th root of unity such that $\zeta_6^2 = \zeta_3$. We may write $\tau = \zeta_6^n \sqrt[3]{|t|}$ for some integer $0 \leq n \leq 5$. % where $n$ is even when $t > 0$ and odd when $t < 0$.
We then find
\[  \xi_{\tau}(c) = \frac{c(\zeta_6^n \sqrt[3]{|t|})}{\zeta_6^n \sqrt[3]{|t|}} = \frac{\zeta_6^{-n} \sqrt[3]{|t|}}{\zeta_6^n \sqrt[3]{|t|}} = \zeta_6^{-2n}. \]
Finally, note that $\zeta_6^{-2n} = \zeta_3^{-n}$, so $\xi_{\tau}(c) = \zeta_3^{k}$ where $k \equiv - n \bmod 3$.
% If instead $t < 0$, then let $\zeta_6$ be a primitive 6th root of unity such that $\zeta_6^2 = \zeta_3$. We write $\sqrt[3]{t} = \zeta_6^n \sqrt[3]{|t|}$ for some $n \in \{1, 3, 5\}$, and in this case, 
% \[  \xi_{\tau}(c) = \frac{c(\zeta_3^j \zeta_6^{n} \sqrt[3]{|t|})}{\zeta_3^j \zeta_6^{n} \sqrt[3]{|t|}} = \frac{\zeta_3^{-j} \zeta_6^{-n} \sqrt[3]{|t|}}{\zeta_3^j \zeta_6^{n} \sqrt[3]{|t|}} = \zeta_3^{-2j} \zeta_{6}^{-2n}. \]
% Since $\zeta_6^2 = \zeta_3$, we have $\zeta_3^{-2j} \zeta_{6}^{-2n} = \zeta_3^{j - n}$. In this case we also have $\xi_{\tau}(c) = \zeta_3^{k}$ for some $k \in \{0, 1, 2\}$, where $k \equiv j - n \bmod 3$.
\end{proof}

\begin{lemma}\label{lem-zetamat}
    Let $K = \Q(\sqrt{-3})$, and let $\OO_K$ be its maximal order. Define 
    \[ Z := \left( \begin{array}{cc}
        -1/2 & 1 \\ 
        -3/4 & -1/2
    \end{array} \right) \in \mathcal{C}_{-3/4, \, 0}(3^{\infty}), \]
    and let $G^{\, 2,1}_{-3/4, \, 0}$ and $G^{\, 6,1}_{-3/4, \, 0}$ be as in the statement of Theorem \ref{thm-j0ell3alvaro}. Finally, let $\zeta_3 := \frac{-1 + \sqrt{-3}}{2}$ be a 3rd root of unity contained in $\OO_K$, and let $[\zeta_3]\in \mathcal{C}_{-3/4, 0}(3^n)$ be the image of the class of $\zeta_3$ under the isomorphism $(\OO_K/3^n \OO_K)^{\times} \cong \mathcal{C}_{-3/4, 0}(3^n)$. Then
    \begin{enumerate}
        \item $[\zeta_3]\in \mathcal{C}_{-3/4, \, 0}(3^{\infty})$ is given by $Z$, and
        \item $G^{\, 2,1}_{-3/4, \, 0} = 
        \langle G^{\, 6,1}_{-3/4, 0}, Z \rangle$.
    \end{enumerate}

\end{lemma}

\begin{proof}
We first prove (1). From \cite[Remark 2.6]{lozano-galoiscm}, we have the isomorphism $(\OO_K/3^n \OO_K)^{\times} \cong \mathcal{C}_{-3/4, 0}(3^n)$ which takes $\alpha$ in the unit group to the multiplication-by-$\alpha$ matrix $M_{\alpha}$, with choice of basis $\{ \sqrt{-3}/2, 1 \}$. Under this choice of basis, $\zeta_3 \mapsto Z$, which holds for all $n \geq 1$, and thus is true $3$-adically.

Result (2) follows from the fact that $Z$ is of order 3, that $Z \in G^{\, 2,1}_{-3/4, \, 0}$ but $Z \notin G^{\, 6,1}_{-3/4, \, 0}$, and  $G^{\, 6,1}_{-3/4, \, 0}$ is an index 3 subgroup of $G^{\, 2,1}_{-3/4, \, 0}$.

\end{proof}

\begin{theorem}\label{j0ell3-ver}
    Let $d\in \Z$ be a non-zero $6$th-power-free integer, and let $E^d\,:\,y^2=x^3 + 16d$ be an elliptic curve with CM by the order $\Of=\Z[(-1+\sqrt{-3})/2]$ of discriminant $\Delta_K f^2=-3$ (i.e., $j(E)=0$). Then, there are $12$ possible $3$-adic images, according to whether $d = 1$ , $-3$, $9$, $-27$, $81$, $-243$, $t^2$, $-3t^2$, $t^3$, $3t^3$, $9t^3$, or $t$, with $t\notin (\Q^\times)^2, -3(\Q^\times)^2, (\Q^\times)^3, 3(\Q^\times)^3, 9(\Q^\times)^3$. Moreover, the image labels corresponding to each possibility are precisely as specified in Table \ref{cmodddividing}.
\end{theorem}

\begin{proof}
    Let $E/\Q$ be an elliptic curve with $j(E) = 0$. Then there is some $6$th-power-free integer $d$ such that $E$ is isomorphic to $E^d : y^2 = x^3 + 16d$. Let $G_{E^d} = G_{E^d,3^\infty}$ be the $3$-adic image of $\rho_{E^d, 3^{\infty}}$, as a subgroup of $\mathcal{N}_{-3/4, 0}(3^{\infty})$.

    Suppose first that $[\mathcal{N}_{-3/4, 0}(3^{\infty}) : G_{E^d}] = 6$. Then, Lemma \ref{lem-j0index} shows that the index of the mod-$9$ image is also $6$ inside $\mathcal{N}_{-3/4,0}(9)$, and Proposition \ref{lem-3and9divfielddegreej0} shows that $d \in \{1, -3, 9, -27, 81, -243 \}$. Via \verb|Magma| computation of mod-$27$ images, we have verified that the $3$-adic image of $E^d$ is generated by $H \subseteq \mathcal{C}_{-3/4, 0}(3^{\infty})$ and $\gamma$, where the triples $(d, H, \gamma)$ are given by the list: $(1, G^{\, 6, 1}_{-3/4, 0}, c_{1})$, $(-27, G^{\, 6, 1}_{-3/4, 0}, c_{-1})$, $(81, G^{\, 6, 2}_{-3/4, 0}, c_{1})$, $(9, G^{\, 6, 3}_{-3/4, 0}, c_{1})$, $(-3, G^{\, 6, 2}_{-3/4, 0}, c_{-1})$, and $(-243, G^{\, 6, 3}_{-3/4, 0}, c_{-1})$. An additional computation shows that the groups $\langle H, \gamma \rangle$ agree with the labels given in Table \ref{cmodddividing}.

    Now, suppose $[\mathcal{N}_{-3/4, 0}(3^{\infty}) : G_{E^d}] = 3$. Then Proposition \ref{lem-3and9divfielddegreej0} and Lemma \ref{lem-j0index} show that $[\Q(E^d[9]) : \Q] = 36$, and so $d = t^3$ or $9t^3$ with $t \neq 1, -3$, or $d = 3t^3$ with $t \neq -1, 3$, where in all cases $t \in \Q^{\ast}$ is a square-free integer. Hence, $E$ is a quadratic twist of $E^1$, $E^{-3}$, or $E^9$. We proceed by cases.

    \begin{itemize}
        \item If $E = E^{t^3}$, then $E$ is a quadratic twist of $E^1$ by $t$. By Proposition \ref{lem-3and9divfielddegreej0}, the only quadratic subfield of $\Q(E^1[9])$ is $\Q(\sqrt{-3})$, so if $t \neq 1, -3$ and $t$ is square-free, then $\Q(\sqrt{t}) \cap \Q(E^1[9]) = \Q$. It follows by Lemma \ref{lem-twistimage} that 
        \[ G_{E^d, 9} = \langle G_{E^1, 9}, -\operatorname{Id} \rangle
        = \langle G^{\, 6,1}_{-3/4, \, 0}, c_1, -\operatorname{Id} \rangle
        = \langle G^{\, 3,1}_{-3/4, \, 0}, c_1 \rangle, \]
        which corresponds to label \texttt{3.1.ns-3.3.1}.
    \item If $E = E^{3t^3}$, then $E$ is a quadratic twist of $E^{-3}$ by $-t$. We will study the image of this curve mod $27$.
    Note that the conductor of $E^{-3}$ is $243 = 3^5$. Thus, as a consequence of the criterion of N\'eron-Ogg-Shafarevich \cite{silverman1}, $\Q(E^{-3}[27])/\Q$ only ramifies at the prime $3$. Further, for $t \in \Q^{\ast}$ as defined above, $\Q(\sqrt{-t})$ will ramify at a prime other than $3$, so $\Q(\sqrt{-t}) \cap \Q(E^{-3}[27]) = \Q$. Hence, by Lemma \ref{lem-twistimage}, we have $\Q(E^d[27]) = \Q(\sqrt{-t},E^{-3}[27])$ and
    \[ G_{E^d, 27} = \langle G_{E^{-3}, 27}, -\operatorname{Id} \rangle = \langle G^{\, 6,2}_{-3/4, \, 0}, c_{-1}, -\operatorname{Id} \rangle
        = \langle G^{\, 3,2}_{-3/4, \, 0}, c_1 \rangle, \]
        which has label \texttt{3.1.ns-27.3.1}.
    \item If $E = E^{9t^3}$, then $E$ is a quadratic twist of $E^{9}$ by $t$. As before, we study the image of the curve mod $27$.
    Note that the conductor of $E^{9}$ is $243 = 3^5$. Thus, as a consequence of the criterion of N\'eron-Ogg-Shafarevich, $\Q(E^{9}[27])/\Q$ only ramifies at the prime $3$. Further, for $t \in \Q^{\ast}$ as defined above, $\Q(\sqrt{t})$ will ramify at a prime other than $3$, so $\Q(\sqrt{t}) \cap \Q(E^{9}[27]) = \Q$. Hence, by Lemma \ref{lem-twistimage}, we have $\Q(E^d[27]) = \Q(\sqrt{t},E^{9}[27])$ and 
        \[ G_{E^d, 27} = \langle G_{E^{9}, 27}, -\operatorname{Id} \rangle
        = \langle G^{\, 6,3}_{-3/4, \, 0}, c_{1}, -\operatorname{Id} \rangle
        = \langle G^{\, 3,3}_{-3/4, \, 0}, c_1 \rangle, \] 
        and this image corresponds to label \texttt{3.1.ns-27.3.2}.

    \end{itemize}

 Next, suppose $[\mathcal{N}_{-3/4, 0}(3^{\infty}) : G_{E^d}] = 2$. Then Proposition \ref{lem-3and9divfielddegreej0} and Lemma \ref{lem-j0index} show that $[\Q(E^d[9]) : \Q] = 54$, and so $d = t^2$ or $-3t^2$ with $t \neq \pm 1, \pm 3, \pm 9$, where $t \in \Q^{\ast}$ is 3rd-power-free in both cases. Hence, $E$ is a cubic twist of $E^1$ by $t$, or a cubic twist of $E^{-27}$ by $9t$. We proceed by cases.

 \begin{itemize}
     \item Since $E^{t^2}$ is a cubic twist of $E^1$ by $t$, it follows that 
     \[ \rho_{E^{t^2}, 9} = \xi_{\tau} \otimes \rho_{E^1, 9},\] 
     where $\tau \in \overline{\Q}$ is a cube root of $t$, and $\xi_\tau$ is a 1-cocycle defined as in Lemma \ref{cubic_twist}.
     
     Since $t \neq \pm 1, \pm 3,$ or $\pm 9$, and is 3rd-power-free, $\tau \in \Q(E^{t^2}[9])$, but $\Q(\tau) \cap \Q(E^1[9]) = \Q$, so by Lemma \ref{cubic_twist}, $\xi_\tau$ introduces a 3rd root of unity into the image of $E^{t^2}$, but no more than a 3rd root of unity. Thus, by Lemma \ref{lem-zetamat}, \[ G_{E^{t^2}, 9} = \langle G_{E^1, 9}, Z \rangle = \langle G^{\, 6,1}_{-3/4, \, 0}, c_1, Z \rangle = \langle G^{\, 2,1}_{-3/4, \, 0}, c_1 \rangle,\]
     which we verify has label \texttt{3.1.ns-9.2.1}.
     
     \item Since $E^{-3t^2}$ is a cubic twist of $E^{-27}$ by $9t$, it follows that 
     \[ \rho_{E^{-3t^2}, 9} = \xi_{\tau} \otimes \rho_{E^{-27}, 9},\] 
     where $\tau \in \overline{\Q}$ is a cube root of $9t$, and $\xi_\tau$ is a 1-cocycle defined as in Lemma \ref{cubic_twist}.
     
     Since $t \neq \pm 1, \pm 3,$ or $\pm 9$, and is 3rd-power-free, $\tau \in \Q(E^{-3t^2}[9])$, but $\Q(\tau) \cap \Q(E^{-27}[9]) = \Q$, so by Lemma \ref{cubic_twist}, $\xi_\tau$ introduces a 3rd root of unity into the image of $E^{-3t^2}$. Thus, by Lemma \ref{lem-zetamat}, \[ G_{E^{-3t^2}, 9} = \langle G_{E^{-27}, 9}, Z \rangle = \langle G^{\, 6,1}_{-3/4, \, 0}, c_{-1}, Z \rangle = \langle G^{\, 2,1}_{-3/4, \, 0}, c_{-1} \rangle,\]
     which we verify has label \texttt{3.1.ns-9.2.2}.
 \end{itemize}

 Finally, if $[\mathcal{N}_{-3/4,0}(3^{\infty}): G_{E^d}]=1$, it follows from our previous work that $E$ is isomorphic to $E^d$ with $d$ not in $\{1, -3, 9, -27, 81, -243 \}$, or $\{ t^3, 9t^3 : t \in \Q^{\ast}, t \neq 1, -3 \}$, or $\{ 3t^2 : t \in \Q^{\ast}, t \neq -1, 3\}$, or $\{t^2, -3t^2 : t \in \Q^{\ast}\}$, and the image is as large as possible, i.e., $G_{E^d}=\mathcal{N}_{-3/4,0}(3^{\infty})$ which corresponds to the label \texttt{3.1.ns-1.1.1}.
    \end{proof}

\subsubsection{$\Delta_Kf=-3$, $\ell > 3$}\label{sec-delta3ellmorethan3}

Let $E/\Q$ be an elliptic curve with CM by the maximal order of $K=\Q(\sqrt{-3})$, i.e., $j(E)=0$, and let $\ell>3$. By Theorem \ref{thm-jzero-goodredn} the $\ell$-adic image is defined mod-$\ell$, and Prop. \ref{prop-zywina3} determines the model of an elliptic curve over $\Q$ with $j(E)=0$ and a prescribed mod-$\ell$ image.  The possible images are thus \texttt{$\ell$.0.s-1.1.1}, or \texttt{$\ell$.0.ns-1.1.1}, or \texttt{$\ell$.0.s-$\ell$.3.1}, or \texttt{$\ell$.0.ns-$\ell$.3.1}, and the specific twists $y^2=x^3+16d$ with each image are as specified in Table \ref{cmoddnotdividing}.

\section{Orders of class number $2$}\label{sec-classnumber2}

Let $\Of$ be an imaginary quadratic order of class number $2$, and so the absolute value of $\Delta(\Of)=\Delta_K f^2$ belongs to $\Delta_2$, where
\begin{align*} 
\Delta_2 &= \{15, 20, 24, 32, 35, 36, 40, 48, 51, 52, 60, 64, 72, 75, 88, 91, 99, 100, 112, 115, 123, 147,\phantom{\}}\\
& \phantom{= \{} 148, 187, 232, 235, 267, 403, 427\}.
\end{align*}
Let $j_{K,f}=j(\Of)$ be a $j$-invariant attached to the lattice $\Of$. Since the class number of $\Of$ is $2$, we have $\Q(j_{K,f})/\Q$ is of degree $2$, and there are exactly two such $j$-invariants, namely $j_{K,f}$ and its non-trivial Galois conjugate. The field of definition of each $j_{K,f}$ is known and it has been recorded in the third column of Table \ref{tab-modelsclassnumber2}. Since $\Q(j_{K,f})$ is quadratic, it is equal to $F=\Q(\sqrt{d})$ for some non-zero square-free integer $d$. Let $\Z[a]$ be the ring of integers of $F$, where $a$ is a root of $x^2-d=0$ if $d\equiv 2,3\bmod 4$ or a root of $x^2-x-(d-1)/4=0$ if $d\equiv 1 \bmod 4$. In Table \ref{tab-modelsclassnumber2} we give an elliptic curve $E/\Q(j_{K,f})$ defined over $\Z[a]$ with CM by $\Of$. 

\begin{remark}\label{rem-conjugates}
    Let $F=\Q(j_{K,f})$ as above and let $\mathcal{O}_F = \Z[a]$. Then, $j_{K,f}=g(a)$ for some polynomial $g(x)\in \Z[x]$, and if $a'$ is the non-trivial Galois conjugate of $a$, then $j_{K,f}'=g(a')$ is another $j$-invariant with CM by $\Of$, so that the two $j$-invariants with such property are $j_{K,f}$ and $j_{K,f}'$. Similarly, if $E/\Q(j_{K,f})$ is given by a Weierstrass equation $w(x,y,a)=0$ for some polynomial $w\in \Z[x,y,z]$, then $E'/\Q(j_{K,f}')$ given by $w(x,y,a')=0$ is another elliptic curve over $\Q(j_{K,f})=\Q(j_{K,f}')$ with CM by $\Of$. In Table \ref{tab-modelsclassnumber2} we give one of the two conjugate elliptic curves and omit the other, but we emphasize here that both such curves have the same $\ell$-adic Galois representations, so if one model $E$ has a prescribed $\ell$-adic Galois representation, then $E'$ has the same property, and viceversa.

    For example, let $\Of$ be the order of discriminant $-15$ (here $\Delta_K=-15$ and $f=1$). Then, $F=\Q(j_{K,f})=\Q(\sqrt{5})$ and $\OF=\Z[a]$ where $a$ is a root of $x^2-x-1=0$. In this case $j_{K,f}=-85995a-52515$ and an elliptic curve with such $j$-invariant is 
    $$E: y^2+xy+ay=x^3-x^2-2ax+a.$$
    Note that since $a = (1\pm \sqrt{5})/2$ it turns out that its Galois conjugate is $1-a = (1\mp \sqrt{5})/2$. Hence, $j_{K,f}'=-85995(1-a)-52515 = 85995a - 138510$ and an elliptic curve with $j=j_{K,f}'$ is 
    $$E': y^2+xy+(1-a)y=x^3-x^2-2(1-a)x+(1-a).$$
    We note that $E'$ is also isomorphic over $\Q(\sqrt{5})$ to 
    $$E'': y^2+xy+(1+a)y=x^3-x^2+(a-2)x+(1-2a),$$
    which is the curve with LMFDB label \href{https://www.lmfdb.org/EllipticCurve/2.2.5.1/81.1/a/3}{\texttt{2.2.5.1-81.1-a3}}. In our tables we only list $E$ with the understanding that $E'\cong E''$ is the conjugate curve one easily obtains from $E$.
\end{remark}

Our method to determine what $\ell$-adic images are possible for $E/\Q(j_{K,f})$ and what models have each prescribed image is based on the following lemma. We will elaborate on the method in Section \ref{sec-method}.
 
\begin{lemma}\label{lem-whattwists}\label{lem-method}
    Let $F$ be a number field, let $E/F$ be an elliptic curve,  let $N> 2$, and let $G_{E,N}\subseteq \GL(2,\Z/N\Z)$ be the image of $\rho_{E,N}\colon G_F \to \GL(2,\Z/N\Z)$. Let $\alpha\in F$ such that $-\operatorname{Id}$ does not belong to $\operatorname{Im}(\rho_{E^\alpha,N})=G_{E^\alpha,N}$, where $E^\alpha$ is the corresponding quadratic twist of $E$. Then:
    \begin{enumerate}
        \item The element $\sqrt{\alpha}$ belongs to $F(E[N])$.
        \item The only prime ideals that can ramify in the extension $F(\sqrt{\alpha})/F$ are the primes dividing $N$ and the conductor $N_{E/F}$ of $E$.
        
        \item Let $M/F$ be the maximal abelian extension of $F$ of exponent $2$ that is unramified outside of the infinite places of $F$ and the primes dividing $N\cdot N_{E/F}$. Then, $F(\sqrt{\alpha})\subseteq M$.

        \item There are only finitely many possibilities for $\alpha$ up to squares in $(F^\times)^2$.

        \item Suppose that $F$ is a real quadratic field of class number $1$, and let $\wp_1,\ldots,\wp_n$ be all the prime ideals of $\OF$ dividing $N\cdot N_E$, where $N_E$ is the norm of the conductor of $E$. Let $\wp_i=(\pi_i)$ for $i=1,\ldots,n$ and some $\pi_i\in \OF$, and write $\OF^\times = \langle -1,u_F\rangle$, where $u_F$ is a fundamental unit. Then, $\alpha$ as above belongs to the set
        $$A_\alpha = \{\pm u_F^{k_0}\pi_1^{k_1}\cdots \pi_n^{k_n} : k_0,\ldots, k_n \in \{0,1\} \}.$$
    \end{enumerate}
\end{lemma}
\begin{proof}
    Let $E$, $\alpha$, and $E^\alpha$ be as in the statement, and assume that $-\id$ does not belong to the image $G_{E^\alpha,N}$. If $\sqrt{\alpha}$ does not belong to $F(E[N])$, then Lemma \ref{lem-twistimage} implies that $G_{E^\alpha,N}$ is conjugate to $\langle -\id, G_{E,N} \rangle$, and therefore $G_{E^\alpha,N}$ would contain $-\id$, a contradiction. Hence $\sqrt{\alpha}$ must belong to $F(E[N])$. This proves (1).

    By (1), we have inclusions $F\subseteq F(\sqrt{\alpha})\subseteq F(E[N])$. Hence, the criterion of N\'eron--Ogg--Shafarevich implies that $F(\sqrt{\alpha})/F$ can only ramify at primes dividing $N\cdot N_{E/F}$. This proves (2).

Part (3) follows directly from (2), because the extension $F(\sqrt{\alpha})/F$ is a quadratic extension, thus abelian of exponent $2$, and the only prime ideals that can ramify are those dividing $N\cdot N_{E/F}$.% Hence $F(\sqrt{\alpha})\subseteq M$ as claimed.

Finally, $M/F$ is necessarily finite (see \cite[Ch. VIII, Prop. 1.6]{silverman1}). Hence, the set of classes of $\alpha$ in $F^\times/(F^\times)^2$ such that $F(\sqrt{\alpha})\subseteq M$ must be finite. This shows (4).

For (5), consider the extension $F(\sqrt{\alpha})/F$ for $\alpha\in F$ as above. Without loss of generality, we may assume that $\alpha \in \OF$. Since the discriminant of the polynomial $x^2-\alpha$ is $4\alpha$, this extension may be ramified at primes above $2$ and any other prime $\wp$ of $\OF$ such that $\operatorname{ord}_\wp(\alpha)=1$. By (2), the only primes that can ramify in $F(\sqrt{\alpha})/F$ are the primes dividing $N\cdot N_E$, so the prime ideals in the factorization of $(\alpha)$ must be contained in $\{\wp_1,\ldots,\wp_n\}$ with notation as in the statement of the lemma. Since $F$ is assumed to be of class number $1$, then $\OF$ is a PID, and $\alpha$, up to squares, must be of the form $\pm u_F^{k_0}\pi_1^{k_1}\cdots \pi_n^{k_n}$ for some $k_0,\ldots,k_n\in\{0,1\}$, as desired.
\end{proof}

\subsection{The method for orders of class number $2$}\label{sec-method} Our method to determine what $\ell$-adic images are possible for $E/\Q(j_{K,f})$ and what models have each prescribed image is based on Lemma \ref{lem-whattwists} and works as follows:

\begin{enumerate}
    \item Choose an imaginary quadratic order $\Of\subseteq K$ of class number $2$. In particular, $\Delta(\Of)=\Delta_K f^2$ belongs to $\Delta_2$.
    \item Determine $j_{K,f}$, a $j$-invariant with CM by $\Of$, and its Galois conjugate $j_{K,f}'$. This is accomplished by computing the Hilbert class polynomial of discriminant $\Delta(\Of)$, as in \cite[Section 7.6.2]{cohen}. 
    \item Determine $F=\Q(j_{K,f})=\Q(\sqrt{d})$, for some non-zero square-free integer $d$, and $\OF = \Z[a]$, its ring of integers. The element $a$ is a root of $x^2-d=0$ if $d\equiv 2,3\bmod 4$ or a root of $x^2-x-(d-1)/4=0$ if $d\equiv 1 \bmod 4$. Note that since $\Of$ is of class number $2$, we have that $\Q(j_{K,f})$ must be a real quadratic field, so $d > 0$ (see \cite[Exercise 2.9]{silverman2}).
    \item Determine a model $E/\Q(j_{K,f})$ of an elliptic curve with CM by $\Of$ defined over $\OF=\Z[a]$. See Section \ref{sec-choiceofmodels} for a description of our choice of a model and the properties of the model. The curve $E$ will also be denoted by $E_{\Of}$ and the model is given in Table \ref{tab-modelsclassnumber2}.
    \item First, we determine what $\ell$-adic images may occur, as follows:
    \begin{enumerate} 
    \item Suppose first that $\ell=2$. Since $\Delta(\Of)\in \Delta_2$, then $j_{K,f}\neq 0,1728$. The possible $2$-adic images are the maximal image, and those described by Theorem \ref{thm-m8and16alvaro}:
    \begin{enumerate}
        \item If $\Delta_K f^2 \equiv 0 \bmod 16$, then there are four possible index-$2$ images. These correspond to labels ending with \texttt{2.$\nu$.c-8.2.$j$} and $j=1,2,3,4$, where as always $\nu=\nu_2(\Delta_K f^2)$ and \texttt{c}, an element of the set $\{\texttt{s,ns3,ns5,ns7}\}$, is the square class of $u=\Delta_K\cdot f^2\cdot 2^{-\nu}$ in $\Z_2^\times/(\Z_2^\times)^2$.
        \item If $\Delta_K\equiv 0 \bmod 8$, or $\Delta_K\equiv 4 \bmod 8$ and $f\equiv 0 \bmod 4$, or $\Delta_K\equiv 1 \bmod 4$ and $f\equiv 0 \bmod 8$, then there are four possible index-$2$ images. These correspond to labels ending with \texttt{2.$\nu$.c-16.2.$k$} and $k=1,2,3,4$. Note that in some cases the images in (a) and (b) can both occur and all eight index-$2$ possibilities may happen (for example: see $\Delta_K f^2=-4\cdot 4^2$ or $-8\cdot 2^2$, where all possibilities do occur).
        \item Otherwise, if $\Delta(\Of)$ does not satisfy the conditions in (a) or (b), there is a unique $2$-adic image, namely the index-$1$ maximal normalizer image determined by the order, with label \texttt{2.$\nu$.c-1.1.1}. 
    \end{enumerate}

    \item If $\ell>2$ and it does not divide $\Delta_K f^2$ (and since we know that $j_{K,f}\neq 0,1728$ in this case), then Theorem \ref{thm-goodredn} shows that there is only one possible $\ell$-adic image, namely the index-$1$ maximal normalizer image determined by the order, with label \texttt{$\ell$.0.c-1.1.1}.

    \item If $\ell>2$ is a divisor of $\Delta_K f^2$, then Theorem \ref{thm-oddprimedividingdisc}, shows that there are three possible $\ell$-adic images, namely \texttt{$\ell$.$\nu$.c-1.1.1}, \texttt{$\ell$.$\nu$.c-$\ell$.2.1}, and \texttt{$\ell$.$\nu$.c-$\ell$.2.2}. We note that not all three need to occur, see Example \ref{ex-method2} for such a case.
    \end{enumerate} 
    \item Once we have determined in the previous step what images are possible, we need to find models with such an image, if any exist. Let $E=E_{\Of}$ be the curve given in Table \ref{tab-modelsclassnumber2}. Suppose $\ell=2$ or $\ell$ divides $\Delta_K f^2$, and suppose $G$ is an $\ell$-adic image that is strictly contained in $\mathcal{N}_{\delta,\phi}(\ell^\infty)$ such that there is an elliptic curve $E'/\Q(j_{K,f})$ with CM by $\Of$ and $j(E')=j_{K,f}$ such that the image of $\rho_{E',\ell^\infty}$ is conjugate to $G$. Without loss of generality, let us assume that $G=G_{E',\ell^\infty}$. Since $E$ and $E'$ are defined over $\Q(j_{K,f})$, and since $j(E)=j(E')=j_{K,f}\neq 0,1728$, $E'$ must be a quadratic twist of $E$. Namely, there is $\alpha \in \OF$ such that $E'=E^\alpha$. Since $j_{K,f} \neq 0, 1728$ and $G = G_{E^{\alpha}, \ell}$ is not maximal, Theorems \ref{thm-oddprimedividingdisc} and \ref{thm-m8and16alvaro} show that the index $[\mathcal{N}_{\delta,\phi}(\ell^\infty):G]$ is $2$. Further, inspection of the non-maximal $\ell$-adic images in Theorems \ref{thm-oddprimedividingdisc}(a) and \ref{thm-m8and16alvaro} shows that $-\operatorname{Id} \notin G$. By Lemma \ref{lem-method}, we have that $\alpha$ can be chosen to be in the finite set $A_\alpha$. Note that the value of $N$ in Lemma \ref{lem-method} can be chosen to be $N=8$ if $\ell=2$ and $N=\ell$ if $\ell>2$, and in that case $\sqrt{\alpha}\in F(E[N])$.
    \item Compute the set $A_\alpha$, and compute all the possible twists $E^\alpha$, for $\alpha\in A_\alpha$.
    \item Compute the $\ell$-adic images $G_{E^\alpha,\ell^\infty}$ for each $\alpha \in A_\alpha$ using the algorithms of \cite{elladic}. By part (6), any non-maximal $\ell$-adic image that occurs, must occur for one of the twists $E^\alpha$.
    \item Moreover, Lemma \ref{lem-method} also implies that if $\alpha$ is not equivalent to an element of $A_\alpha$ up to squares, then the image $G_{E^\alpha,\ell^\infty}$ does contain $-\operatorname{Id}$, and therefore it is the index-$1$ maximal image. Hence, part (8) does determine all the possible $\ell$-adic images that occur for elliptic curves with CM by $\Of$.
\end{enumerate}

We illustrate the method below in Examples \ref{ex-method1} and \ref{ex-method2}.

\begin{remark}\label{rem-classnumber3}
    This method does not work in general for orders of class number greater than $2$. For example, let $\Of$ be the maximal order of imaginary quadratic field $K = \Q(\sqrt{-283})$. The field $K$ is class number $3$ in this case, and the associated $j$-invariant $j_{K,f}$ is an algebraic integer of degree $3$ where $\Q(j_{K,f}) \cong \Q[x]/(x^3 + 4x - 1)$ is class number $2$. Thus, the decomposition in Lemma \ref{lem-whattwists}(5) cannot be found in general for an elliptic curve $E/\Q(j_{K,f})$ with CM by $\Of$.
\end{remark}

\begin{example}\label{ex-method1}
    Let $K=\Q(\sqrt{-3})$, let $f=7$, and let $\Of$ be the order with discriminant $\Delta(\Of)=-147=-3\cdot 7^2$. By Table \ref{tab-modelsclassnumber2}, the field $\Q(j_{K,f})$ is $F=\Q(\sqrt{21})$. Let $a$ be a root of $x^2-x-5=0$. In the same table, we have given a model of an elliptic curve $E/\Q(\sqrt{21})$   
    $$y^2 + y = x^3 + (-a - 1)x^2 + (4131a - 11618)x + 221331a - 618025$$
    with CM by the order $\Of$, and conductor of norm $49$. The $j$-invariant of $E$ is given by
    $$j_{K,f} = 7604567359488000a - 21226536456192000.$$
    We note that, in practice, one first finds the $j$-invariant $j_{K,f}$, then an elliptic curve $E'/\Q(\sqrt{21})$ with $j(E')=j_{K,f}$, and finally a minimal quadratic twist $E$ of $E'$ (minimal in the sense of Section \ref{sec-choiceofmodels}). 

By Theorem \ref{thm-goodredn}, any elliptic curve over $F$ with CM by $\Of$ will have $\ell$-adic image $\mathcal{N}_{\delta,\phi}(\ell^\infty)$ for any prime $\ell$ not dividing $2\Delta_K f = 2\cdot 3\cdot 7$. The label of the image, thus, will be either \texttt{$\ell$.0.s-1.1.1} or \texttt{$\ell$.0.ns-1.1.1} according to whether $\Delta_K=-3$ is a square mod $\ell$ or not, respectively.

Now let $\ell=2$. Since $j_{K,f}\neq 0,1728$, and $\Delta_K f^2 \not\equiv 0 \bmod 4$ and $f\not\equiv 0\bmod 8$, it follows by Theorem \ref{thm-m8and16alvaro} that the $2$-adic image of any such elliptic curve is also $\mathcal{N}_{-49,7}(2^\infty)$, which corresponds to the label \texttt{2.0.ns5-1.1.1} because $-3\cdot 7^2\equiv -3\equiv 5 \bmod 8$. It remains to see if the images in Theorem \ref{thm-oddprimedividingdisc} (a), can occur for $\ell=3$ and $\ell=7$.
    
  Let $\ell=7$. The algorithm from \cite{elladic} determines that the $7$-adic image of the elliptic curve $E/\Q(\sqrt{21})$ given above is \texttt{7.2.s-7.2.2}. Since the only prime dividing $N\cdot N_E=7\cdot 7^2$ in this case is $7$, Lemma \ref{lem-whattwists} shows that any other twist of $E$ with a $7$-adic image missing $-\operatorname{Id}$ must be of the form $E^\alpha$ with $\alpha$ in the set
  $$\{\pm (3-a)^{k_0} (-4+a)^{k_1} : k_0,k_1 \in \{0,1\} \},$$
  where $3-a$ is a fundamental unit for $\OF$ and $\wp_7 = (-4+a)$ is the unique prime ideal above $7$. Trying out all the $8$ possibilities we find that the twist by $a-3$ has $7$-adic image with label \texttt{7.2.s-7.2.1}, and any other twist $E^\alpha$ with $\alpha \neq 1, a-3$ up to squares has image \texttt{7.2.s-1.1.1}. (Note that the quotient $(a-3)/(-7) = ((4-a)/7)^2$, so $a-3$ and $-7$ are equivalent up to squares. The table records $-7$ as the appropriate twist for convenience.)

  Now we move on to $\ell=3$. The algorithm from \cite{elladic} determines that the $3$-adic image of the elliptic curve $E/\Q(\sqrt{21})$ given above is \texttt{3.1.ns-1.1.1}. In order to possibly find twists $E^\alpha$ with images \texttt{3.1.ns-3.2.1} and \texttt{3.1.ns-3.2.2}, Lemma \ref{lem-whattwists} says that $\alpha$ is in the set 
  $$\{\pm (3-a)^{k_0} (-4+a)^{k_1}(2-a)^{k_2} : k_0,k_1,k_2 \in \{0,1\} \},$$
  where $3-a$, $-4+a$ are as above, and $(2-a)=\wp_3$ is the unique prime ideal of $\OF$ above $3$. Our algorithm finds that that the twist by $2-a$ gives the image \texttt{3.1.ns-3.2.2} and the twist by $4a-11$ gives the image \texttt{3.1.ns-3.2.1}. (Note that the quotient $(4a-11)/(a+1) = (3-a)^2$, so they are equivalent up to squares. The table records $a+1$ and $2-a$ as the appropriate twists.) 
\end{example}

\begin{example}\label{ex-method2}
   Let $K=\Q(\sqrt{-10})$, let $f=1$, let $\Of$ be the order with discriminant $\Delta(\Of)=-40$, and consider the $5$-adic images attached to elliptic curves with CM by $\Of$ defined over $\Q(j_{K,f})=\Q(\sqrt{5})$.  Let $E/\Q(\sqrt{5}): y^2 = x^3 +(6a-28)x +16a - 56$ where $a$ is a root of $x^2-x-1=0$, which is a curve with CM by $\Of$. The norm of the conductor of $E$ is $N_E=4096=2^{12}$. The $5$-adic image of $E$ is \texttt{5.1.ns-1.1.1}.

   Now suppose that $E^\alpha$ is a twist that has one of the other two  possible images, namely  \texttt{5.1.ns-5.2.1} or \texttt{5.1.ns-5.2.2}. Then, by Lemma \ref{lem-whattwists}, the element $\alpha$ belongs to 
   $$\{\pm a^{k_0} 2^{k_1}(1-2a)^{k_2} : k_0,k_1,k_2 \in \{0,1\} \},$$
   where $a$ is a fundamental unit for $\OF$, $(2)$ is the unique prime ideal above $2$, and $(1-2a)=\wp_5$ is the unique prime ideal above $5$.   However, when we calculate the label of the $5$-adic image of $E^\alpha$ for all such choices, we find that all curves have image \texttt{5.1.ns-1.1.1}. Thus, that is the only possible $5$-adic image for this CM order.
\end{example}

%%%%%%%%%%%%%%%%%%%%%%%
\section{Tables}\label{sec-tables}
%%%%%%%%%%%%%%%%%%%%%%%%%
In this section, we record the data for each order $\mathcal{O}=\mathcal{O}_{K,f}$ of class number $1$ or $2$.
\begin{enumerate}
\item In Table  \ref{tab-modelsclassnumber2} we give a Weierstrass model for an elliptic curve $E=E_{\mathcal{O}}$ defined over $\Q(j_{K,f})$, such that $E$ has CM by $\mathcal{O}$, which is minimal in a certain sense (see Section \ref{sec-choiceofmodels}).
\item For each prime $\ell$, we give a description of all the possible $\ell$-adic images of Galois attached to an elliptic curve $E$ with CM by an order $\mathcal{O}$ as above and, for each possible image $G\subseteq \GL(2,\Z_\ell)$, we give a precise description of the twists of $E_{\mathcal{O}}$ that have $G$ as their $\ell$-adic Galois image:
\begin{enumerate}
\item[$\bullet$] Table \ref{cm2big}: $\ell =2$,
\item[$\bullet$] Table \ref{tab-ellnotdividing}: $\ell \ne 2$ and $\ell \not|\Delta(\mathcal O)$,
\item[$\bullet$] Table \ref{tab-cmodddividing}: $\ell \ne 2$ and $\ell |\Delta(\mathcal O)$.
\end{enumerate}
\end{enumerate}
{%\color{red}
\renewcommand{\arraystretch}{1.5}
\begin{longtable}{|c|c|c|c|}
\caption{Models for CM Elliptic Curves}
\label{tab-modelsclassnumber2}\label{tab-models}\\

%$\Delta(\mathcal O)$ 
\hline
$\Delta_K$ & $f$ & $\Q(j_{K,f})$ & $E_{\mathcal O}/\Q(j_{K,f})$ \\
 \hline 
\endfirsthead

\multicolumn{4}{c}%
{{\bfseries \tablename\ \thetable{} -- continued from previous page}} \\
\hline%$\Delta(\mathcal O)$ 
$\Delta_K$ & $f$ & $\Q(j_{K,f})$ & $E_{\mathcal O}/\Q(j_{K,f})$ \\
 \hline 
\endhead
\hline \multicolumn{4}{|r|}{{Continued on next page}} \\ \hline
\endfoot

\endlastfoot

{$-3$} & $1$ & $\QQ$& \href{https://www.lmfdb.org/EllipticCurve/Q/27/a/4}{$y^2=x^3+16$}   \\ %\eclabel{27.a4}
\cline{2-4}
                    & $2$ & $\QQ$  & \href{https://www.lmfdb.org/EllipticCurve/Q/36/a/2}{ $y^2=x^3-15x+22$ } \\ % \eclabel{36.a2}  
\cline{2-4}
             &  $3$ & $\QQ$  & \href{https://www.lmfdb.org/EllipticCurve/Q/27/a/2}{$y^2=x^3-480x+4048$} \\ %\eclabel{27.a2}  
\cline{2-4}
 & $4$ &   $\Q(\sqrt{3})$ & \href{https://www.lmfdb.org/EllipticCurve/2.2.12.1/16.1/a/3}{$y^2+ (a + 1)xy + (a + 1)y = x^3 + (-a - 1)x^2 + (4a - 13)x + 11a - 21$}\\
\cline{2-4}
& $5$ &   $\Q(\sqrt{5})$ & \href{https://www.lmfdb.org/EllipticCurve/2.2.5.1/2025.1/c/1}{$y^2 + y = x^3 + (6a - 48)x + 109a - 76$} \\
\cline{2-4}
& $7$ &   $\Q(\sqrt{21})$ &\href{https://www.lmfdb.org/EllipticCurve/2.2.21.1/49.1/a/1}{$y^2 + y = x^3 + (-a - 1)x^2 + (4131a - 11618)x + 221331a - 618025$} \\
\hline
 {$-4$} & $1$ &  $\QQ$ & \href{https://www.lmfdb.org/EllipticCurve/Q/64/a/4}{$y^2=x^3+x$}\\ %  \eclabel{64.a4}   
 \cline{2-4}
   & $2$ &  $\QQ$  & \href{https://www.lmfdb.org/EllipticCurve/Q/32/a/2}{$y^2=x^3-11x+14$}  \\ %\eclabel{32.a2} 
 \cline{2-4}  
& $3$ &   $\Q(\sqrt{3})$ & \href{https://www.lmfdb.org/EllipticCurve/2.2.12.1/9.1/a/1}{$y^2 + (a + 1)xy + ay = x^3 + (a - 1)x^2 + (25a - 45)x + 72a - 127$}\\
\cline{2-4}
& $4$ &   $\Q(\sqrt{2})$ & \href{https://www.lmfdb.org/EllipticCurve/2.2.8.1/32.1/a/1}{$y^2 + axy = x^3 + x^2 + (15a - 22)x + 46a - 69$}\\
\cline{2-4}
& $5$ &   $\Q(\sqrt{5})$ &{$y^2 = x^3 + (117a - 556)x + 3920a - 3640$} \\
\hline
{$-7$} & $1$ & $\QQ$ & \href{https://www.lmfdb.org/EllipticCurve/Q/49/a/2}{$y^2=x^3-1715x+33614$} \\ %  \eclabel{49.a2}  
 \cline{2-4}
&  $2$ & $\QQ$ & \href{https://www.lmfdb.org/EllipticCurve/Q/49/a/1}{$y^2=x^3-29155x+1915998$}\\  %\eclabel{49.a1}   
\cline{2-4}
& $4$ &   $\Q(\sqrt{7})$ & \href{https://www.lmfdb.org/EllipticCurve/2.2.28.1/1.1/a/3}{$y^2 + xy + y = x^3 - x^2 + (-270a - 715)x + 3223a + 8527$}\\
\hline
{$-8$}  & $1$ & $\QQ$ & \href{https://www.lmfdb.org/EllipticCurve/Q/256/a/2}{$y^2=x^3-4320x+96768$ }\\ %\eclabel{256.a2}  
 \cline{2-4}
& $2$ &   $\Q(\sqrt{2})$ & \href{https://www.lmfdb.org/EllipticCurve/2.2.8.1/64.1/a/1}{$y^2 + axy = x^3 + (-a - 1)x^2 + (2a+2)x - 3a - 5$}\\
\cline{2-4}
& $3$ &   $\Q(\sqrt{6})$ & \href{https://www.lmfdb.org/EllipticCurve/2.2.24.1/1.1/a/3}{$y^2 + axy + (a + 1)y = x^3 + (a + 1)x^2 + (67a - 161)x + 458a - 1122$}\\
\hline
{$-11$} & $1$ & $\QQ$ & \href{https://www.lmfdb.org/EllipticCurve/Q/121/b/2}{$y^2=x^3-9504x+365904$}\\ %  \eclabel{121.b2}  
  \cline{2-4}
 & $3$ &   $\Q(\sqrt{33})$ & \href{https://www.lmfdb.org/EllipticCurve/2.2.33.1/1.1/a/1}{$y^2 + y = x^3 - ax^2 + (-435a - 1030)x -7890a -18717$}\\
\hline
{$-15$} & $1$ &   $\Q(\sqrt{5})$ & \href{https://www.lmfdb.org/EllipticCurve/2.2.5.1/81.1/a/1}{$y^2 + xy + a y = x^3 - x^2 - 2 a x + a$}\\
\cline{2-4}
& $2$ &   $\Q(\sqrt{5})$ & \href{https://www.lmfdb.org/EllipticCurve/2.2.5.1/81.1/a/5}{$y^2 + axy + (a + 1)y = x^3 + (-a - 1)x^2 + (-13a - 14)x - 20a - 6$} \\
\hline
 $-19$ & $1$ & $\QQ$ & \href{https://www.lmfdb.org/EllipticCurve/Q/361/a/2}{$y^2=x^3-608x+5776$}\\ %  \eclabel{361.a2} 
   \hline %6
$-20$ & $1$ & $\QQ(\sqrt{5})$ & \href{https://www.lmfdb.org/EllipticCurve/2.2.5.1/4096.1/k/1}{$y^2 = x^3 - ax^2 + (-a - 9)x -6a - 15$}\\ 
\hline
$-24$ & $1$ &   $\Q(\sqrt{2})$ & \href{https://www.lmfdb.org/EllipticCurve/2.2.8.1/81.1/b/1}{$y^2 + axy + y = x^3 + x^2 + (a-3)x - a + 1$}\\
\hline
$-35$ & $1$ &   $\Q(\sqrt{5})$ & \href{https://www.lmfdb.org/EllipticCurve/2.2.5.1/2401.1/b/1}{$y^2 + y = x^3 - x^2 + (-14a + 19)x + 21a - 36$}\\
\hline
$-40$ & $1$ &   $\Q(\sqrt{5})$ &{$y^2 = x^3 + (6a - 28)x + 16a - 56$}\\
\hline
 $-43$ & $1$ & $\QQ$ & \href{https://www.lmfdb.org/EllipticCurve/Q/1849/b/2}{$y^2=x^3-13760x+621264$}\\ %\eclabel{1849.b2}  
 \hline
$-51$ & $1$ &   $\Q(\sqrt{17})$ & \href{https://www.lmfdb.org/EllipticCurve/2.2.17.1/81.1/b/1}{$y^2 + y = x^3 + (-6a - 12)x + 14a + 19$}\\
\hline
$-52$ & $1$ &   $\Q(\sqrt{13})$ &{$ y^2 = x^3 + (10a - 35)x + 40a - 76$}\\
\hline
$-67$ & $1$ & $\QQ$ & \href{https://www.lmfdb.org/EllipticCurve/Q/4489/b/2}{$y^2=x^3-117920x+15585808$}\\ % \eclabel{4489.b2}   
 \hline  %6
$-88$ & $1$ &   $\Q(\sqrt{2})$ &{$y^2 + axy + y = x^3 + x^2 + (-193a - 453)x + 2233a + 4008$}\\
\hline
$-91$ & $1$ &   $\Q(\sqrt{13})$ &{$y^2 + y = x^3 + (84a - 182)x + 539a - 1213$}\\
\hline
$-115$ & $1$ &   $\Q(\sqrt{5})$ &{$ y^2 + y = x^3 + (46a - 368)x + 2645a - 6216 $}\\
\hline
$-123$ & $1$ &   $\Q(\sqrt{41})$ & \href{https://www.lmfdb.org/EllipticCurve/2.2.41.1/81.1/c/1}{$y^2 + y = x^3 + (-60a - 210)x + 560a + 1384$}\\
\hline
$-148$ & $1$ &   $\Q(\sqrt{37})$ &{$y^2 = x^3 + (290a - 1615)x + 8120a - 23268$}\\
\hline
$-163$ & $1$ & $\QQ$ & \href{https://www.lmfdb.org/EllipticCurve/Q/26569/a/2}{$y^2=x^3-34790720x+78984748304$} \\ % \eclabel{26569.a2} 
\hline
$-187$ & $1$ &   $\Q(\sqrt{17})$ &{$y^2 + y = x^3 + (1430a - 3520)x -40898a + 104090$}\\
\hline
$-232$ & $1$ &   $\Q(\sqrt{29})$ &{$y^2 = x^3 + (7280a - 36310)x + 960960a - 2492952$}\\
\hline
$-235$ & $1$ &   $\Q(\sqrt{5})$ &{$ y^2 + y = x^3 + (4136a - 17578)x + 324723a -962572$} \\
\hline
$-267$ & $1$ &   $\Q(\sqrt{89})$ & \href{https://www.lmfdb.org/EllipticCurve/2.2.89.1/81.1/a/1}{$y^2 + y = x^3 + (-1590a - 8580)x + 92750a + 359875$}\\
\hline
$-403$ & $1$ &   $\Q(\sqrt{13})$ &{$y^2 + y = x^3 + (186930a - 427490)x + 58571989a - 135261471$}\\
\hline
$-427$ & $1$ &   $\Q(\sqrt{61})$ &{$y^2 + y = x^3 + (30030a - 137060)x + 5787145a - 25355528$}\\
\hline
\end{longtable}
}
%%%%%%%%%%%%

\newpage

{%\color{red}
\renewcommand{\arraystretch}{1.5}
\begin{longtable}{|c|c|c|c|c|l|l|}
\caption{CM $2$-adic Galois images}\label{cm2big}\label{tab-ell2}\\
 \hline 
 %$\Delta(\mathcal O)$ 
$\Delta_K$ & $f$ & $\Q(j_{K,f})$ & \multicolumn{1}{|c|}{CM-label} & $G_{E,2^\infty}$ & twist  &  \multicolumn{1}{|c|}{conditions}\\  \hline 
\endfirsthead
\multicolumn{7}{c}%
{{\bfseries \tablename\ \thetable{} -- continued from previous page}} \\
 \hline 
 %$\Delta(\mathcal O)$ 
$\Delta_K$ & $f$ & $\Q(j_{K,f})$ & \multicolumn{1}{|c|}{CM-label} & $G_{E,2^\infty}$ & twist  &  \multicolumn{1}{|c|}{conditions}\\  \hline 
\endhead

\hline \multicolumn{7}{|r|}{{Continued on next page}} \\ \hline
\endfoot

%\hline \hline
\endlastfoot

{$-3$} & {$1$}  & {$\QQ$} & \texttt{2.0.ns5-1.1.1} & $\mathcal N_{-1,1}(2^{\infty})$ &  $t$ & $t\in \Q^*,\,\,t\not\in 4(\Q^\ast)^3$ \\ 
 \cline{4-7}   
   &   &   & \texttt{2.0.ns5-2.3.1} & $\langle \cC_{-1,1}(2^{\infty})^3, c_{1}' \rangle$ &   $4t^3$ &  $t\in \Q^*$  \\  
   \cline{2-7}
         & $2$ &$\QQ$ & \texttt{2.2.ns5-1.1.1} & $\mathcal N_{-3, 0}(2^{\infty})$ & $t$ & $t\in \Q^*$ \\ \cline{2-7}
          & $3$ &$\QQ$ & \texttt{2.0.ns5-1.1.1} & $\mathcal N_{-9, 3}(2^{\infty})$ & $t$ & $t\in \Q^*$ \\ 
          \cline{2-7}
   &{$4$} & {$\QQ(\sqrt{3})$} & \texttt{2.4.ns5-1.1.1} & $\mathcal N_{-12,0}(2^{\infty})$ &   $t$  &  $t\in \Q(\sqrt{3})^*,\,\, t\ne \pm 1,\pm 2$ \\  \cline{4-7} 
 & & & \texttt{2.4.ns5-8.2.1} & $\langle G^{\, 2,1}_{-12,0},c_{1}\rangle$ &   $-2$  &  \ecnflabel{2.2.12.1}{256.1-c4}  \\ \cline{4-7}   
 & & & \texttt{2.4.ns5-8.2.2} & $\langle G^{\, 2,2}_{-12,0},c_{1}\rangle$ &   $-1$ &  \ecnflabel{2.2.12.1}{16.1-a4}  \\ \cline{4-7}   
 & & & \texttt{2.4.ns5-8.2.3} & $\langle G^{\, 2,2}_{-12,0},c_{-1}\rangle$ &  $1$  & \ecnflabel{2.2.12.1}{16.1-a3} \\ \cline{4-7}   
 & & & \texttt{2.4.ns5-8.2.4} & $\langle G^{\, 2,1}_{-12,0},c_{-1}\rangle$ &   $2$  &  \ecnflabel{2.2.12.1}{256.1-c3} \\ \cline{2-7} 
 & $5$ & $\QQ(\sqrt{5})$ & \texttt{2.0.ns5-1.1.1} & $\mathcal N_{-25, 5}(2^{\infty})$ & $t$ & $t\in \Q(\sqrt{5})^*$ \\ \cline{2-7}
& $7$ & $\QQ(\sqrt{21})$ &  \texttt{2.0.ns5-1.1.1} & $\mathcal N_{-49, 7}(2^{\infty})$ & $t$ & $t\in \Q(\sqrt{21})^*$ \\ \hline

{$-4$} & {$1$}  & {$\QQ$} &  \texttt{2.2.ns7-1.1.1} & $\mathcal N_{-1,0}(2^{\infty})$ &   $t$ &  $t\in \Q^*,\,\,t\not \in \pm\Q^2,\pm 2 \Q^2$\\  \cline{4-7}  
 & & & \texttt{2.2.ns7-4.2.1} & $\langle G^{\, 2,1}_{-1,0},c'_{1}\rangle$ &   $t^2$  &$t\in \Q^*, \,\,t\ne \pm 1,\pm 2$
\\ \cline{4-7} 
 & & & \texttt{2.2.ns7-4.2.2} & $\langle G^{\, 2,1}_{-1,0},c_{1}\rangle$  &  $-t^2$ & $t\in \Q^*, \,\,t\ne \pm 1,\pm 2$ \\ \cline{4-7} 
 & &  & \texttt{2.2.ns7-8.2.1} & $\langle G^{\, 2,2}_{-1,0},c'_{1}\rangle$ &  $2t^2$ &  $t\in \Q^*, \,\,t\ne \pm 1,\pm 2$\\ \cline{4-7}  
 & & & \texttt{2.2.ns7-8.2.2} & $\langle G^{\, 2,2}_{-1,0},c_{1}\rangle$ &  $-2t^2$ &  $t\in \Q^*, \,\,t\ne \pm 1,\pm 2$ \\  \cline{4-7}   
 & & & \texttt{2.2.ns7-4.4.1} & $\langle G^{\, 4,2}_{-1,0},c'_{-1}\rangle$ &   $4$ &  \eclabel{32.a4} \\  \cline{4-7}   
 & & & \texttt{2.2.ns7-4.4.2} & $\langle G^{\, 4,1}_{-1,0},c'_{-1}\rangle$ &   $1$ &  \eclabel{64.a4} \\  \cline{4-7}   
 & & & \texttt{2.2.ns7-4.4.3} & $\langle G^{\, 4,1}_{-1,0},c_{-1}\rangle$ &   $-4$ &  \eclabel{64.a3} \\  \cline{4-7}   
 & & & \texttt{2.2.ns7-4.4.4} & $\langle G^{\, 4,2}_{-1,0},c_{-1}\rangle$ &   $-1$ &  \eclabel{32.a3} \\  \cline{4-7}   
 & & & \texttt{2.2.ns7-16.4.1} & $\langle G^{\, 4,3}_{-1,0},c'_{1}\rangle$ &  $2$ &  \eclabel{256.c2}  \\  \cline{4-7}   
 & & & \texttt{2.2.ns7-16.4.2} & $\langle G^{\, 4,4}_{-1,0},c'_{1}\rangle$ &  $8$ &  \eclabel{256.b2} \\  \cline{4-7}   
 & & & \texttt{2.2.ns7-16.4.3} & $\langle G^{\, 4,3}_{-1,0},c_{-1}\rangle$ &   $-8$ &  \eclabel{256.c1} \\  \cline{4-7}   
 & & & \texttt{2.2.ns7-16.4.4} & $\langle G^{\, 4,4}_{-1,0},c_{-1}\rangle$ &   $-2$ &  \eclabel{256.b1} \\    \cline{2-7}
    &   {$2$} &  {$\QQ$} & \texttt{2.4.ns7-1.1.1} & $\mathcal N_{-4,0}(2^{\infty})$ &   $t$  &  $t\in \Q^*,\,\, t\ne \pm 1,\pm 2$ \\  \cline{4-7} 
 & & & \texttt{2.4.ns7-8.2.1} & $\langle G^{\, 2,1}_{-4,0},c_{1}\rangle$ &   $2$  &  \eclabel{64.a2} \\ \cline{4-7}   
 & & & \texttt{2.4.ns7-8.2.2} & $\langle G^{\, 2,2}_{-4,0},c_{1}\rangle$ &   $1$ &  \eclabel{32.a2}  \\ \cline{4-7}   
 & & & \texttt{2.4.ns7-8.2.3} & $\langle G^{\, 2,2}_{-4,0},c_{-1}\rangle$ &  $-1$  &  \eclabel{32.a1} \\ \cline{4-7}   
 & & & \texttt{2.4.ns7-8.2.4} & $\langle G^{\, 2,1}_{-4,0},c_{-1}\rangle$ &   $-2$  &  \eclabel{64.a1} \\ \cline{2-7} 
  & 3 & $\QQ(\sqrt{3})$& \texttt{2.2.ns7-1.1.1} & $\mathcal N_{-9, 0}(2^{\infty})$ & $t$ & $t\in \Q(\sqrt{3})^*$ \\ \cline{2-7}
    & {$4$}  &{$\QQ(\sqrt{2 })$} & \texttt{2.6.ns7-1.1.1} & $\mathcal N_{-16,0}(2^{\infty})$ &   $t$  &  $t\in \Q(\sqrt{2})^*,\,\, t\ne \pm 1,\pm a, \pm a + 1, a \pm 2$\\  \cline{4-7} 
 & & & \texttt{2.6.ns7-8.2.1} & $\langle G^{\, 2,1}_{-16,0},c_{1}\rangle$ &   $-1$  &  \ecnflabel{2.2.8.1}{32.1-a2} \\ \cline{4-7}   
 & & & \texttt{2.6.ns7-8.2.2} & $\langle G^{\, 2,2}_{-16,0},c_{1}\rangle$ &   $a$ &  \ecnflabel{2.2.8.1}{1024.1-k2}  \\ \cline{4-7}   
 & & & \texttt{2.6.ns7-8.2.3} & $\langle G^{\, 2,2}_{-16,0},c_{-1}\rangle$ &  $-a$  & \ecnflabel{2.2.8.1}{1024.1-k1} \\ \cline{4-7}   
 & & & \texttt{2.6.ns7-8.2.4} & $\langle G^{\, 2,1}_{-16,0},c_{-1}\rangle$ &   $1$  &  \ecnflabel{2.2.8.1}{32.1-a1} \\ \cline{4-7} 
 & & & \texttt{2.6.ns7-16.2.1} & $\langle G^{\, 2,3}_{-16,0},c_{1}\rangle$ &   $a-2$  &  \ecnflabel{2.2.8.1}{1024.1-f1} \\ \cline{4-7}   
 & & & \texttt{2.6.ns7-16.2.2} & $\langle G^{\, 2,4}_{-16,0},c_{1}\rangle$ &   $a+1$ &  \ecnflabel{2.2.8.1}{256.1-a1}  \\ \cline{4-7}   
 & & & \texttt{2.6.ns7-16.2.3} & $\langle G^{\, 2,3}_{-16,0},c_{-1}\rangle$ &  $a+2$ & \ecnflabel{2.2.8.1}{1024.1-f2}  \\ \cline{4-7}   
 & & & \texttt{2.6.ns7-16.2.4} & $\langle G^{\, 2,4}_{-16,0},c_{-1}\rangle$ &   $-a+1$  &  \ecnflabel{2.2.8.1}{256.1-a2} \\ \cline{2-7} 
 & $5$ & $\QQ(\sqrt{5})$& \texttt{2.2.ns7-1.1.1} & $\mathcal N_{-25, 0}(2^{\infty})$ & $t$ & $t\in \Q(\sqrt{5})^*$ \\ \hline

{$-7$} & $1$ & $\QQ$ & \texttt{2.0.s-1.1.1} & $\mathcal N_{-2, 1}(2^{\infty})$ & $t$ & $t\in \Q^*$ \\ \cline{2-7}
          & $2$ & $\QQ$& \texttt{2.2.s-1.1.1} & $\mathcal N_{-7, 0}(2^{\infty})$ & $t$ & $t\in \Q^*$ \\ \cline{2-7}
   &{$4$} & {$\QQ(\sqrt{7 })$} & \texttt{2.4.s-1.1.1} & $\mathcal N_{-28,0}(2^{\infty})$ &   $t$  &  $t\in \Q(\sqrt{7})^*,\,\, t\ne \pm 1,\pm 2$ \\  \cline{4-7} 
 & & & \texttt{2.4.s-8.2.1} & $\langle G^{\, 2,1}_{-28,0},c_{1}\rangle$ &   $-2$  &   \ecnflabel{2.2.28.1}{256.1-j4}\\ \cline{4-7}   
 & & & \texttt{2.4.s-8.2.2} & $\langle G^{\, 2,2}_{-28,0},c_{1}\rangle$ &   $-1$ &    \ecnflabel{2.2.28.1}{1.1-a4} \\ \cline{4-7}   
 & & & \texttt{2.4.s-8.2.3} & $\langle G^{\, 2,2}_{-28,0},c_{-1}\rangle$ &  $1$  &  \ecnflabel{2.2.28.1}{1.1-a3} \\ \cline{4-7}   
 & & & \texttt{2.4.s-8.2.4} & $\langle G^{\, 2,1}_{-28,0},c_{-1}\rangle$ &   $2$  &  \ecnflabel{2.2.28.1}{256.1-j3} \\ \hline

  {$-8$} &    {$1$} & {$\QQ$} &  \texttt{2.3.ns7-1.1.1} & $\mathcal N_{-2,0}(2^{\infty})$ &  $t$ &  $t\in \Q^*,\,\, t\ne \pm 1,\pm 2$ \\  \cline{4-7} 
 & & & \texttt{2.3.ns7-16.2.1} & $\langle G^{\, 2,3}_{-2,0},c_{1}\rangle$ &   $2$ &  \eclabel{256.d1} \\  \cline{4-7}   
 & & & \texttt{2.3.ns7-16.2.2} & $\langle G^{\, 2,4}_{-2,0},c_{1}\rangle$ &   $1$ &  \eclabel{256.a2} \\  \cline{4-7}   
 & & & \texttt{2.3.ns7-16.2.3} & $\langle G^{\, 2,3}_{-2,0},c_{-1}\rangle$ &  $-1$ &  \eclabel{256.d2}\\  \cline{4-7}   
 & & & \texttt{2.3.ns7-16.2.4} & $\langle G^{\, 2,4}_{-2,0},c_{-1}\rangle$ & $-2$ &  \eclabel{256.a1} \\    \cline{2-7}
  &{$2$}  &{$\QQ(\sqrt{2 })$} & \texttt{2.5.ns7-1.1.1} & $\mathcal N_{-8,0}(2^{\infty})$ &   $t$  &  $t\in \Q(\sqrt{2})^*,\,\, t\ne \pm 1,\pm a, \pm a + 1, a \pm 2$\\  \cline{4-7} 
 & & & \texttt{2.5.ns7-8.2.1} & $\langle G^{\, 2,1}_{-8,0},c_{1}\rangle$ &   $-1$  &   \ecnflabel{2.2.8.1}{64.1-a2} \\ \cline{4-7}   
 & & & \texttt{2.5.ns7-8.2.2} & $\langle G^{\, 2,2}_{-8,0},c_{1}\rangle$ &   $a+1$ &  \ecnflabel{2.2.8.1}{256.1-c2} \\ \cline{4-7}   
 & & & \texttt{2.5.ns7-8.2.3} & $\langle G^{\, 2,2}_{-8,0},c_{-1}\rangle$ &  $-a+1$  & \ecnflabel{2.2.8.1}{256.1-c1} \\ \cline{4-7}   
 & & & \texttt{2.5.ns7-8.2.4} & $\langle G^{\, 2,1}_{-8,0},c_{-1}\rangle$ &   $1$  &  \ecnflabel{2.2.8.1}{64.1-a1} \\ \cline{4-7} 
 & & & \texttt{2.5.ns7-16.2.1} & $\langle G^{\, 2,3}_{-8,0},c_{1}\rangle$ &   $a$  &  \ecnflabel{2.2.8.1}{1024.1-m1} \\ \cline{4-7}   
 & & & \texttt{2.5.ns7-16.2.2} & $\langle G^{\, 2,4}_{-8,0},c_{1}\rangle$ &   $a-2$ &  \ecnflabel{2.2.8.1}{1024.1-d2} \\ \cline{4-7}   
 & & & \texttt{2.5.ns7-16.2.3} & $\langle G^{\, 2,3}_{-8,0},c_{-1}\rangle$ &  $-a$  &  \ecnflabel{2.2.8.1}{1024.1-m2} \\ \cline{4-7}   
 & & & \texttt{2.5.ns7-16.2.4} & $\langle G^{\, 2,4}_{-8,0},c_{-1}\rangle$ &   $a+2$  &  \ecnflabel{2.2.8.1}{1024.1-d1} \\ \cline{2-7}
  & {$3$}   & {$\QQ(\sqrt{6})$} & \texttt{2.3.ns7-1.1.1} & $\mathcal N_{-18,0}(2^{\infty})$ &   $t$  &  $t\in \Q(\sqrt{6})^*,\,\, t\ne \pm a \pm 2$ \\  \cline{4-7} 
 & & & \texttt{2.3.ns7-16.2.1} & $\langle G^{\, 2,3}_{-18,0},c_{1}\rangle$ &   $a + 2$  &  N/A \\ \cline{4-7}  
 & & & \texttt{2.3.ns7-16.2.2} & $\langle G^{\, 2,4}_{-18,0},c_{1}\rangle$ &   $a - 2$ &   N/A \\ \cline{4-7}  
 & & & \texttt{2.3.ns7-16.2.3} & $\langle G^{\, 2,3}_{-18,0},c_{-1}\rangle$ &  $-a + 2$  & N/A \\ \cline{4-7}  
 & & & \texttt{2.3.ns7-16.2.4} & $\langle G^{\, 2,4}_{-18,0},c_{-1}\rangle$ &   $-a - 2$  &   N/A\\ \hline 

           {$-11$} & $1$ &$\QQ$ & \texttt{2.0.ns5-1.1.1} & $\mathcal N_{-3, 1}(2^{\infty})$ & $t$ & $t\in \Q^*$ \\ \cline{2-7}
          & $3$ & $\QQ(\sqrt{33})$ & \texttt{2.0.ns5-1.1.1} & $\mathcal N_{-27, 3}(2^{\infty})$ & $t$ & $t\in \Q(\sqrt{33})^*$ \\ \hline
{$-15$}  & $1$& $\QQ(\sqrt{5})$ &  \texttt{2.0.s-1.1.1} & $\mathcal N_{-4, 1}(2^{\infty})$ & $t$ & $t\in \Q(\sqrt{5})^*$ \\ \cline{2-7}
 & $2$ & $\QQ(\sqrt{5})$ & \texttt{2.2.s-1.1.1} & $\mathcal N_{-15, 0}(2^{\infty})$ & $t$ & $t\in \Q(\sqrt{5})^*$ \\ \hline

         $-19$ & $1$ & $\QQ$& \texttt{2.0.ns5-1.1.1} & $\mathcal N_{-5, 1}(2^{\infty})$ & $t$ & $t\in \Q^*$ \\ \hline
    
$-20$ & $1$ & $\QQ(\sqrt{5})$ & \texttt{2.2.ns3-1.1.1} & $\mathcal N_{-5, 0}(2^{\infty})$ & $t$ & $t\in \Q(\sqrt{5})^*$ \\ \hline
$-24$ & $1$ & $\QQ(\sqrt{2})$& \texttt{2.3.ns5-1.1.1} & $\mathcal N_{-6, 0}(2^{\infty})$ & $t$ & $t\in \Q(\sqrt{2})^*$ \\ \hline
$-35$ & $1$ & $\QQ(\sqrt{5})$ & \texttt{2.0.ns5-1.1.1} & $\mathcal N_{-9, 1}(2^{\infty})$ & $t$ & $t\in \Q(\sqrt{5})^*$ \\ \hline
{$-40$}   & {$1$} & {$\QQ(\sqrt{5})$} & \texttt{2.3.ns3-1.1.1} & $\mathcal N_{-10,0}(2^{\infty})$ &   $t$  &  $t\in \Q(\sqrt{5})^*,\,\, t\ne \pm a, \pm 2a$ \\  \cline{4-7} 
 & & & \texttt{2.3.ns3-16.2.1} & $\langle G^{\, 2,3}_{-10,0},c_{1}\rangle$ &   $-a$  &  N/A \\ \cline{4-7}   
 & & & \texttt{2.3.ns3-16.2.2} & $\langle G^{\, 2,4}_{-10,0},c_{1}\rangle$ &   $-2a$ &   N/A \\ \cline{4-7}   
 & & & \texttt{2.3.ns3-16.2.3} & $\langle G^{\, 2,3}_{-10,0},c_{-1}\rangle$ &  $2a$  & N/A \\ \cline{4-7}   
 & & & \texttt{2.3.ns3-16.2.4} & $\langle G^{\, 2,4}_{-10,0},c_{-1}\rangle$ &   $a$  &   N/A\\ \hline 
       $-43$ & $1$ & $\QQ$& \texttt{2.0.ns5-1.1.1} & $\mathcal N_{-11, 1}(2^{\infty})$ & $t$ & $t\in \Q^*$ \\ \hline
$-51$ & $1$ & $\QQ(\sqrt{17})$ & \texttt{2.0.ns5-1.1.1} & $\mathcal N_{-13, 1}(2^{\infty})$ & $t$ & $t\in \Q(\sqrt{17})^*$ \\ \hline
$-52$ & $1$& $\QQ(\sqrt{13})$ &  \texttt{2.2.ns3-1.1.1} & $\mathcal N_{-13, 0}(2^{\infty})$ & $t$ & $t\in \Q(\sqrt{13})^*$ \\ \hline

        $-67$ & $1$ & $\QQ$& \texttt{2.0.ns5-1.1.1} & $\mathcal N_{-17, 1}(2^{\infty})$ & $t$ & $t\in \Q^*$ \\ \hline

$-88$ & $1$ & $\QQ(\sqrt{2})$& \texttt{2.3.ns5-1.1.1} & $\mathcal N_{-22, 0}(2^{\infty})$ & $t$ & $t\in \Q(\sqrt{2})^*$ \\ \hline

$-91$ & $1$ &$\QQ(\sqrt{13})$ &  \texttt{2.0.ns5-1.1.1} & $\mathcal N_{-23, 1}(2^{\infty})$ & $t$ & $t\in \Q(\sqrt{13})^*$ \\ \hline

$-115$ & $1$ & $\QQ(\sqrt{5})$ & \texttt{2.0.ns5-1.1.1} & $\mathcal N_{-29, 1}(2^{\infty})$ & $t$ & $t\in \Q(\sqrt{5})^*$ \\ \hline

$-123$ & $1$& $\QQ(\sqrt{41})$ & \texttt{2.0.ns5-1.1.1} & $\mathcal N_{-31, 1}(2^{\infty})$ & $t$ & $t\in \Q(\sqrt{41})^*$ \\ \hline

$-148$ & $1$ & $\QQ(\sqrt{37})$ & \texttt{2.2.ns3-1.1.1} & $\mathcal N_{-37, 0}(2^{\infty})$ & $t$ & $t\in \Q(\sqrt{37})^*$ \\ \hline

         $-163$ & $1$ & $\QQ$& \texttt{2.0.ns5-1.1.1} & $\mathcal N_{-41, 1}(2^{\infty})$ & $t$ & $t\in \Q^*$ \\ \hline

$-187$ & $1$& $\QQ(\sqrt{17})$ &  \texttt{2.0.ns5-1.1.1} & $\mathcal N_{-47, 1}(2^{\infty})$ & $t$ & $t\in \Q(\sqrt{17})^*$ \\ \hline

{$-232$}   & {$1$}& {$\QQ(\sqrt{29})$} & \texttt{2.3.ns3-1.1.1} & $\mathcal N_{-58,0}(2^{\infty})$ &   $t$  &  $t\in \Q(\sqrt{29})^*,\,\, t\ne \pm 1,\pm 2$ \\  \cline{4-7} 
 & & & \texttt{2.3.ns3-16.2.1} & $\langle G^{\, 2,3}_{-58,0},c_{1}\rangle$ &   $-1$  &  N/A\\ \cline{4-7}   
 & & & \texttt{2.3.ns3-16.2.2} & $\langle G^{\, 2,4}_{-58,0},c_{1}\rangle$ &   $-2$ &   N/A \\ \cline{4-7}   
 & & & \texttt{2.3.ns3-16.2.3} & $\langle G^{\, 2,3}_{-58,0},c_{-1}\rangle$ &  $2$  & N/A \\ \cline{4-7}   
 & & & \texttt{2.3.ns3-16.2.4} & $\langle G^{\, 2,4}_{-58,0},c_{-1}\rangle$ &   $1$  &  N/A \\ \hline

$-235$ & $1$ & $\QQ(\sqrt{5})$ & \texttt{2.0.ns5-1.1.1} & $\mathcal N_{-59, 1}(2^{\infty})$ & $t$ & $t\in \Q(\sqrt{5})^*$ \\ \hline

$-267$ & $1$& $\QQ(\sqrt{89})$ &  \texttt{2.0.ns5-1.1.1} & $\mathcal N_{-67, 1}(2^{\infty})$ & $t$ & $t\in \Q(\sqrt{89})^*$ \\ \hline

$-403$ & $1$ & $\QQ(\sqrt{13})$ & \texttt{2.0.ns5-1.1.1} & $\mathcal N_{-101, 1}(2^{\infty})$ & $t$ & $t\in \Q(\sqrt{13})^*$ \\ \hline

$-427$ & $1$& $\QQ(\sqrt{61})$ &  \texttt{2.0.ns5-1.1.1} & $\mathcal N_{-107, 1}(2^{\infty})$ & $t$ & $t\in \Q(\sqrt{61})^*$ \\ \hline

\end{longtable}
}

%%%%%%%%%%%%%%%%%%%%%%%%%%%%%%%%%%%

{%\color{blue}

\renewcommand{\arraystretch}{1.5}
\begin{longtable}{|c|c|l|c|c|l|}
\caption{CM $\ell$-adic Galois images: $\ell \ne 2$ and $\ell \not| \, \Delta(\mathcal{O})$}\label{cmoddnotdividing}\label{tab-ellnotdividing}
\\ \hline
$\Delta(\mathcal O)$  & condition on $\ell$ & \multicolumn{1}{|c|}{CM-label}  & $G_{E,\ell^\infty}$ 
& twist  &  \multicolumn{1}{|c|}{conditions} \\  
\hline 
{$-3$}   & $\ell \equiv 1\pmod 9$  & \texttt{$\ell$.0.s-1.1.1} &  $\mathcal{N}_{-3/4,0}(\ell^{\infty})$ & $t$ & $t\in\Q^*$  \\ \cline{2-6} 
& $\ell\equiv 8\pmod 9$ & \texttt{$\ell$.0.ns-1.1.1} & $\mathcal{N}_{-3/4,0}(\ell^{\infty})$
& $t$ & $t\in\Q^*$ \\  \cline{2-6} 
& {$\ell \equiv 4,7\pmod 9$}  & \texttt{$\ell$.0.s-1.1.1} & $\mathcal{N}_{-3/4,0}(\ell^{\infty})$
& $t$ & $t\in\Q^*$, $t\not\in \ell^r\cdot (\Q^\ast)^3$ \\  \cline{3-6} 
&                                                              & \texttt{$\ell$.0.s-$\ell$.3.1} & $\langle G^{\, 3,1}_{-3/4,0}, c_{\varepsilon} \rangle$
& $\ell^r t^3$ & $t\in\Q^*$, $r\equiv (\ell-1)/3\pmod 3$, $r=1,2$\\  \cline{2-6}
& {$\ell \equiv 2,5\pmod 9$}  & \texttt{$\ell$.0.ns-1.1.1} & $\mathcal{N}_{-3/4,0}(\ell^{\infty})$
& $t$ & $t\in\Q^*$, $t\not\in \ell^r\cdot (\Q^\ast)^3$ \\  \cline{3-6} 
&                                                              & \texttt{$\ell$.0.ns-$\ell$.3.1} & $\langle \cC_{-3/4,\,0}(\ell^{\infty})^3, c_\varepsilon \rangle $
& $\ell^r t^3$ & $t\in\Q^*$, $-r\equiv (\ell+1)/3 \pmod 3$, $r=1,2$\\  \hline
 {$\ne -3$}  & $ \left(\frac{\Delta(\mathcal O)}{\ell}\right)=+1$  &  \texttt{$\ell$.0.ns-1.1.1} & $\mathcal{N}_{\delta, 0}(\ell^{\infty})$
 & $t$ & $t\in\Q^*$  \\ \cline{2-6}        
& $ \left(\frac{\Delta(\mathcal O)}{\ell}\right)=-1$  &  \texttt{$\ell$.0.s-1.1.1} & $\mathcal{N}_{\delta, 0}(\ell^{\infty})$
& $t$ & $t\in\Q^*$  \\ \hline                          
\end{longtable}

}
%\newpage

%%%%%%%%%%%%%%%%%%%%%%%%%%%%%%%%%%%%

%\newpage 

{%\color{red}
\renewcommand{\arraystretch}{1.5}
\begin{longtable}{|c|c|c|c|c|l|l|}
\caption{CM $\ell$-adic Galois images: $\ell\ne 2$ and $\ell|\Delta(\mathcal O)$}\label{cmodddividing}\label{tab-cmodddividing}\\
 \hline 
 %$\Delta(\mathcal O)$ 
$\Delta_K$ & $f$ & $\Q(j_{K,f})$ & \multicolumn{1}{|c|}{CM-label} & $G_{E,\ell^\infty}$ & twist  &  \multicolumn{1}{|c|}{conditions}\\  \hline 
\endfirsthead
\multicolumn{7}{c}%
{{\bfseries \tablename\ \thetable{} -- continued from previous page}} \\
 \hline 
 %$\Delta(\mathcal O)$ 
$\Delta_K$ & $f$ & $\Q(j_{K,f})$ & \multicolumn{1}{|c|}{CM-label} & $G_{E,\ell^\infty}$ & twist  &  \multicolumn{1}{|c|}{conditions}\\  \hline 
\endhead

\hline \multicolumn{7}{|r|}{{Continued on next page}} \\ \hline
\endfoot

%\hline \hline
\endlastfoot

 \hline 

{$-3$} & {$1$} & {$\QQ$} 
& \texttt{3.1.ns-1.1.1} & $\mathcal N_{-3/4,\,0}(3^{\infty})$ &   $t$ &  $t\in \Q^*,\,\,t\not\in \Q^3,3\Q^3,9\Q^3,\Q^2,-3\Q^2 $ \\ \cline{4-7} 
& & &  \texttt{3.1.ns-3.3.1} & $\langle G^{\, 3,1}_{-3/4,0},c_{1}\rangle$  &  $t^3$ & $t\in \Q^*,\,\,t\ne 1,-3$ \\  \cline{4-7}   
& &  & \texttt{3.1.ns-9.2.1} & $\langle G^{\, 2,1}_{-3/4,0},c_{1}\rangle$  & $t^2$  & $t\in \Q^*,\,\,t\ne \pm 1,\pm 3,\pm 9$ \\  \cline{4-7}  
& &&\texttt{3.1.ns-9.2.2} & $\langle G^{\, 2,1}_{-3/4,0},c_{-1}\rangle$  & $-3t^2$  & $t\in \Q^*,\,\,t\ne \pm 1,\pm 3,\pm 9$ \\ \cline{4-7}  
& &&\texttt{3.1.ns-27.3.1} & $\langle G^{\, 3,2}_{-3/4,0},c_{1}\rangle$  & $3t^3$  &  $t\in \Q^*,\,\,t\ne -1,3$\\  \cline{4-7}  
& &&\texttt{3.1.ns-27.3.2} & $\langle G^{\, 3,3}_{-3/4,0},c_{1}\rangle$  & $9t^3$ &   $t\in \Q^*,\,\,t\ne 1,-3$ \\ \cline{4-7}   
& &&\texttt{3.1.ns-9.6.1} & $\langle G^{\, 6,1}_{-3/4,0},c_{1}\rangle$ & $1$  & \eclabel{27.a4}\\  \cline{4-7}   
& &&\texttt{3.1.ns-9.6.2} & $\langle G^{\, 6,1}_{-3/4,0},c_{-1}\rangle$ &  $-27$  & \eclabel{27.a3} \\  \cline{4-7}  
& &&\texttt{3.1.ns-27.6.1} & $\langle G^{\, 6,2}_{-3/4,0},c_{1}\rangle$ &  $81$  & \eclabel{243.a2}  \\  \cline{4-7}   
& &&\texttt{3.1.ns-27.6.2} & $\langle G^{\, 6,3}_{-3/4,0},c_{1}\rangle$ &  $9$  & \eclabel{243.b2} \\  \cline{4-7}   
& &&\texttt{3.1.ns-27.6.3} & $\langle G^{\, 6,2}_{-3/4,0},c_{-1}\rangle$ &  $-3$ &\eclabel{243.a1} \\ \cline{4-7}  
& &&\texttt{3.1.ns-27.6.4} & $\langle G^{\, 6,3}_{-3/4,0},c_{-1}\rangle$ & $-243$ & \eclabel{243.b1}\\ \cline{2-7} 
%{$-pf^2\ne -3$} 
%& \texttt{p.1.ns-1.1.1} &  &   $t$ &  $t\in \Q^*,\,\,t\ne 1,-p $ \\ \cline{4-7} 
%& \texttt{p.1.ns-p.2.1} &  &   $1$ &   \\ \cline{4-7} 
%& \texttt{p.1.ns-p.2.2} &  &   $-p$ &  \\ \hline
& {$2$}   & {$\QQ$}  
 & \texttt{3.1.ns-1.1.1} & $\mathcal{N}_{-3,\,0}(3^{\infty})$ &   $t$ &  $t\in \Q^*,\,\,t\ne 1,-3 $ \\ \cline{4-7} 
 & &  & \texttt{3.1.ns-3.2.1} & $\langle G^{2,1}_{-3,\,0},c_{1}\rangle$ &   $1$ &  \eclabel{36.a2}  \\ \cline{4-7}  
 & &  & \texttt{3.1.ns-3.2.2} & $\langle G^{2,1}_{-3,\,0},c_{-1}\rangle$ &   $-3$ & \eclabel{36.a1} \\ \cline{2-7}
& {$3$} & {$\QQ$}  
& \texttt{3.3.ns-1.1.1} & $\mathcal{N}_{-27/4,\,0}(3^{\infty})$  &   $t$ &  $t\in \Q^*,\,\,t\ne 1,-3 $ \\ \cline{4-7} 
& &  & \texttt{3.3.ns-3.2.1} & $\langle G^{2,1}_{-27/4,\,0},c_{1}\rangle$ &   $1$ &  \eclabel{27.a2} \\ \cline{4-7} 
& &  & \texttt{3.3.ns-3.2.2} &  $\langle G^{2,1}_{-27/4,\,0},c_{-1}\rangle$ &   $-3$ & \eclabel{27.a1} \\ \cline{2-7}
%{$-48$} & {$-$} &  {$\Q(\sqrt{3})$} 
 & {$4$} &  {$\Q(\sqrt{3})$} 
 & \texttt{3.1.ns-1.1.1} & $\mathcal{N}_{-12,\,0}(3^{\infty})$ &   $t$ &  $t\in \Q(\sqrt{3})^*,\,\,t\ne \pm 2a $ \\ \cline{4-7} 
& &  & \texttt{3.1.ns-3.2.1} & $\langle G^{\,2,1}_{-12,\,0},c_{1}\rangle$ &   $2a$ &  \ecnflabel{2.2.12.1}{36.1-a4} \\ \cline{4-7} 
& &  & \texttt{3.1.ns-3.2.2} & $\langle G^{\,2,1}_{-12,\,0},c_{-1}\rangle$ &   $-2a$ &  \ecnflabel{2.2.12.1}{36.1-a3} \\ 
\cline{2-7}
% {$-75$} &  {$-$} &  {$\Q(\sqrt{5})$} 
 &  {$5$} &  {$\Q(\sqrt{5})$} 
& \texttt{3.1.ns-1.1.1} & $\mathcal{N}_{-75/4,\,0}(3^{\infty})$ &   $t$ &  $t\in \Q(\sqrt{5})^*,\,\,t\ne a - 3, 3a + 6  $ \\ \cline{4-7} 
&& &   \texttt{3.1.ns-3.2.1} & $\langle G^{\,2,1}_{-75/4,\,0},c_{1}\rangle$ &   $a - 3$ &  \ecnflabel{2.2.5.1}{2025.1-d1} \\ 
\cline{4-7} 
&&   & \texttt{3.1.ns-3.2.2} & $\langle G^{\, 2,1}_{-75/4,\,0},c_{-1}\rangle$ &   $3a + 6$ & \ecnflabel{2.2.5.1}{2025.1-d2} \\ \cline{4-7} 
& &  & \texttt{5.2.ns-1.1.1} & $\mathcal{N}_{-75/4,\,0}(5^{\infty})$ &   $t$ &  $t\in \Q(\sqrt{5})^*$ \\ \cline{2-7}
%{$-147$} & {$-$} &  {$\Q(\sqrt{21})$} 
 & {$7$} &  {$\Q(\sqrt{21})$} 
 & \texttt{3.1.ns-1.1.1} & $\mathcal{N}_{-147/4,\,0}(3^{\infty})$ &   $t$ &  $t\in \Q(\sqrt{21})^*, t \ne a + 1, -a + 2$ \\ \cline{4-7} 
& &  & \texttt{3.1.ns-3.2.1} & $\langle G^{\, 2, 1}_{-147/4, 0}, c_1 \rangle$ &   $a + 1$ &  \ecnflabel{2.2.21.1}{441.1-d1} \\ \cline{4-7} 
& &  & \texttt{3.1.ns-3.2.2} & $\langle G^{\, 2, 1}_{-147/4, 0}, c_{-1} \rangle$ &   $-a + 2$ &  \ecnflabel{2.2.21.1}{441.1-d2}  \\ \cline{4-7} 
& &  & \texttt{7.2.s-1.1.1} & $\mathcal{N}_{-147/4,\,0}(7^{\infty})$ &   $t$ &  $t\in \Q(\sqrt{21})^*,\,\,t\ne 1,-7 $ \\ \cline{4-7} 
&  & & \texttt{7.2.s-7.2.1} & $\langle G^{\,2,1}_{-147/4,\,0},c_{1}\rangle$ &   $-7$ & \ecnflabel{2.2.21.1}{49.1-a2}  \\ \cline{4-7} 
& &  & \texttt{7.2.s-7.2.2} & $\langle G^{\, 2,1}_{-147/4,\,0},c_{-1}\rangle$ &   $1$ & \ecnflabel{2.2.21.1}{49.1-a1} \\ \hline
%{$-36$} & {$-$} &  {$\Q(\sqrt{3})$} 
{$-4$} & {$3$} &  {$\Q(\sqrt{3})$} 
 & \texttt{3.2.ns-1.1.1} & $\mathcal{N}_{-9,\,0}(3^{\infty})$ &   $t$ &  $t\in \Q(\sqrt{3})^*,\,\,t\ne \pm 1 $ \\ \cline{4-7} 
& &  & \texttt{3.2.ns-3.2.1} & $\langle G^{2,1}_{-9,\,0},c_{1}\rangle$ &   $-1$ &  \ecnflabel{2.2.12.1}{9.1-a2}  \\ \cline{4-7} 
& &  & \texttt{3.2.ns-3.2.2} & $\langle G^{2,1}_{-9,\,0},c_{-1}\rangle$ &   $1$ & \ecnflabel{2.2.12.1}{9.1-a1}  \\ \cline{2-7}
%$-100$ & $-$ &  $\Q(\sqrt{5})$ &\texttt{5.2.s-1.1.1} & $\mathcal{N}_{-25,\,0}(5^{\infty})$ &   $t$ &  $t\in \Q(\sqrt{5})^*$ \\ \hline
 & $5$ &  $\Q(\sqrt{5})$ &\texttt{5.2.s-1.1.1} & $\mathcal{N}_{-25,\,0}(5^{\infty})$ &   $t$ &  $t\in \Q(\sqrt{5})^*$ \\ \hline
%{$-112$} &  {$-$} &  {$\Q(\sqrt{7})$} 
{$-7$}  & {$1$}  & {$\QQ$}  
& \texttt{7.1.ns-1.1.1} & $\mathcal{N}_{-7/4,\,0}(7^{\infty})$  &   $t$ &  $t\in \Q^*,\,\,t\ne 1,-7 $ \\ \cline{4-7} 
 & &  & \texttt{7.1.ns-7.2.1} & $\langle G^{2,1}_{-7/4,\,0},c_{1}\rangle$ &   $1$ & \eclabel{49.a2}  \\ \cline{4-7} 
 & &  & \texttt{7.1.ns-7.2.2} & $\langle G^{2,1}_{-7/4,\,0},c_{-1}\rangle$ &   $-7$ & \eclabel{49.a4} \\ \cline{2-7}
& {$2$}  & {$\QQ$}  
& \texttt{7.1.ns-1.1.1} & $\mathcal{N}_{-7,\,0}(7^{\infty})$ &   $t$ &  $t\in \Q^*,\,\,t\ne 1,-7 $ \\ \cline{4-7} 
 & &  & \texttt{7.1.ns-7.2.1} & $\langle G^{2,1}_{-7,\,0},c_{1}\rangle$ &   $1$ & \eclabel{49.a1}  \\ \cline{4-7} 
 & &  & \texttt{7.1.ns-7.2.2} & $\langle G^{2,1}_{-7,\,0},c_{-1}\rangle$ &   $-7$ & \eclabel{49.a3} \\ \cline{2-7}
 &  {$4$} &  {$\Q(\sqrt{7})$} 
 & \texttt{7.1.ns-1.1.1} & $\mathcal{N}_{-28,\,0}(7^{\infty})$ &   $t$ &  $t\in \Q(\sqrt{7})^*,\,\,t\ne \pm 2a $ \\ \cline{4-7} 
&  & & \texttt{7.1.ns-7.2.1} & $\langle G^{\,2,1}_{-28,\,0},c_{1}\rangle$ &   $2a$ &  \ecnflabel{2.2.28.1}{49.1-a3} \\ \cline{4-7} 
&  & & \texttt{7.1.ns-7.2.2} & $\langle G^{\,2,1}_{-28,\,0},c_{-1}\rangle$ &   $-2a$ & \ecnflabel{2.2.28.1}{49.1-a4} \\ \hline
%{$-72$} &  {$6$} &  {$\Q(\sqrt{-})$} 
{$-8$} &  {$3$} &  {$\Q(\sqrt{6})$} 
 & \texttt{3.2.s-1.1.1} & $\mathcal{N}_{-18,\,0}(3^{\infty})$ &   $t$ &  $t\in \Q(\sqrt{6})^*,\,\,t\ne 1, -3 $ \\ \cline{4-7} 
&  & & \texttt{3.2.s-3.2.1} & $\langle G^{\,2,1}_{-18,\,0},c_{1}\rangle$ &   $-3$ &  \ecnflabel{2.2.24.1}{1.1-a4} \\ \cline{4-7} 
& &  & \texttt{3.2.s-3.2.2} & $\langle G^{\,2,1}_{-18,\,0},c_{-1}\rangle$ &   $1$ &  \ecnflabel{2.2.24.1}{1.1-a3} \\ \hline
%{$-99$} & {$-$} &  {$\Q(\sqrt{-})$} 
{$-11$} & {$1$}  & {$\QQ$}   
& \texttt{11.1.ns-1.1.1} & $\mathcal{N}_{-11/4,\,0}(11^{\infty})$  &   $t$ &  $t\in \Q^*,\,\,t\ne 1,-11 $ \\ \cline{4-7} 
 & &  & \texttt{11.1.ns-11.2.1} & $\langle G^{2,1}_{-11/4,\,0},c_{1}\rangle$ &   $1$ &  \eclabel{121.b2} \\ \cline{4-7} 
 & &  & \texttt{11.1.ns-11.2.2} & $\langle G^{2,1}_{-11/4,\,0},c_{-1}\rangle$ &   $-11$ & \eclabel{121.b1} \\ \cline{2-7}
 & {$3$} &  {$\Q(\sqrt{33})$} 
 & \texttt{3.2.s-1.1.1} & $\mathcal{N}_{-99/4,\,0}(3^{\infty})$ &   $t$ &  $t\in \Q(\sqrt{33})^*, t \neq 1, -3$ \\ \cline{4-7} 
& &  & \texttt{3.2.s-3.2.1} & $\langle G^{\,2,1}_{-99/4,\,0},c_{1}\rangle$ &   $-3$ &   \ecnflabel{2.2.33.1}{1.1-a2} \\ \cline{4-7} 
& &  & \texttt{3.2.s-3.2.2} & $\langle G^{\, 2,1}_{-99/4,\,0},c_{-1}\rangle$ &   $1$ &  \ecnflabel{2.2.33.1}{1.1-a1} \\ \cline{4-7} 
& &  & \texttt{11.1.ns-1.1.1} & $\mathcal{N}_{-99/4,\,0}(11^{\infty})$ &   $t$ &  $t\in \Q(\sqrt{33})^*,\,\,t\ne 4a+9, -4a+13 $ \\ \cline{4-7} 
& &  & \texttt{11.1.ns-11.2.1} & $\langle G^{\,2,1}_{-99/4,\,0},c_{1}\rangle$ &   $-4a+13$ & \ecnflabel{2.2.33.1}{121.1-b1}  \\ \cline{4-7} 
& &  & \texttt{11.1.ns-11.2.2} & $\langle G^{\, 2,1}_{-99/4,\,0},c_{-1}\rangle$ &   $4a+9$ &  \ecnflabel{2.2.33.1}{121.1-b2} \\ \hline
{$-15$} & {$1$} &  {$\Q(\sqrt{5})$} 
 & \texttt{3.1.s-1.1.1} & $\mathcal{N}_{-15/4,\,0}(3^{\infty})$ &   $t$ &  $t\in \Q(\sqrt{5})^*,\,\,t\ne 1,-3 $ \\ \cline{4-7} 
 & & & \texttt{3.1.s-3.2.1} & $\langle G^{\,2,1}_{-15/4,\,0},c_{1}\rangle$ &   $1$ &  \ecnflabel{2.2.5.1}{81.1-a1} \\ \cline{4-7} 
 & & & \texttt{3.1.s-3.2.2} & $\langle G^{\, 2,1}_{-15/4,\,0},c_{-1}\rangle$ &   $-3$ &  \ecnflabel{2.2.5.1}{81.1-a2} \\ \cline{4-7}  
 & & &  \texttt{5.1.ns-1.1.1} & $\mathcal{N}_{-15/4,\,0}(5^{\infty})$ &   $t$ &  $t\in \Q(\sqrt{5})^*$ \\ \cline{2-7}
%{$-60$} &  {$-$} &  {$\Q(\sqrt{5})$} 
 &  {$2$} &  {$\Q(\sqrt{5})$} 
&  \texttt{3.1.s-1.1.1} & $\mathcal{N}_{-15,\,0}(3^{\infty})$ &   $t$ &  $t\in \Q(\sqrt{5})^*,\,\,t\ne 1,-3 $ \\ \cline{4-7} 
 & &&  \texttt{3.1.s-3.2.1} & $\langle G^{\,2,1}_{-15,\,0},c_{1}\rangle$ &   $1$ &  \ecnflabel{2.2.5.1}{81.1-a5} \\ \cline{4-7} 
 & &&  \texttt{3.1.s-3.2.2} & $\langle G^{\, 2,1}_{-15,\,0},c_{-1}\rangle$ &   $-3$ &  \ecnflabel{2.2.5.1}{81.1-a6} \\ \cline{4-7} 
 & &&  \texttt{5.1.ns-1.1.1} & $\mathcal{N}_{-15,\,0}(5^{\infty})$ &   $t$ &  $t\in \Q(\sqrt{5})^*$ \\ \hline
{$-19$} & {$1$}  & {$\QQ$}   
& \texttt{19.1.ns-1.1.1} & $\mathcal{N}_{-19/4,\,0}(19^{\infty})$  &   $t$ &  $t\in \Q^*,\,\,t\ne 1,-19 $ \\ \cline{4-7} 
 & &  & \texttt{19.1.ns-19.2.1} & $\langle G^{2,1}_{-19/4,\,0},c_{1}\rangle$ &   $1$ & \eclabel{361.a2}  \\ \cline{4-7} 
 & &  & \texttt{19.1.ns-19.2.2} & $\langle G^{2,1}_{-19/4,\,0},c_{-1}\rangle$ &   $-19$ & \eclabel{361.a1} \\ \hline
$-20$ & $1$ & $\Q(\sqrt{5})$ & \texttt{5.1.s-1.1.1} & $\mathcal{N}_{-5,\,0}(5^{\infty})$ &   $t$ &  $t\in \Q(\sqrt{5})^*$ \\ \hline
{$-24$}  & {$1$} &  {$\Q(\sqrt{2})$} 
  & \texttt{3.1.s-1.1.1} & $\mathcal{N}_{-6,\,0}(3^{\infty})$ &   $t$ &  $t\in \Q(\sqrt{2})^*,\,\,t\ne 1,-3 $ \\ \cline{4-7} 
& & & \texttt{3.1.s-3.2.1} & $\langle G^{2,1}_{-6,\,0},c_{1}\rangle$ &   $1$ &  \ecnflabel{2.2.8.1}{81.1-b1} \\ \cline{4-7} 
& & &   \texttt{3.1.s-3.2.2} & $\langle G^{2,1}_{-6,\,0},c_{-1}\rangle$ &   $-3$ &  \ecnflabel{2.2.8.1}{81.1-b2} \\ \hline
{$-35$}  & {$1$} &  {$\Q(\sqrt{5})$} 
&  \texttt{5.1.ns-1.1.1} & $\mathcal{N}_{-35/4,\,0}(5^{\infty})$ &   $t$ &  $t\in \Q(\sqrt{5})^*$ \\ \cline{4-7} 
 & &&  \texttt{7.1.s-1.1.1} & $\mathcal{N}_{-35/4,\,0}(7^{\infty})$ &   $t$ &  $t\in \Q(\sqrt{5})^*,\,\,t\ne 1,-7 $ \\ \cline{4-7} 
 & &&  \texttt{7.1.s-7.2.1} & $\langle G^{\,2,1}_{-35/4,\,0},c_{1}\rangle$  &   $-7$ &  \ecnflabel{2.2.5.1}{2401.1-b2} \\ \cline{4-7} 
 & &&  \texttt{7.1.s-7.2.2} & $\langle G^{\, 2,1}_{-35/4,\,0},c_{-1}\rangle$ &   $1$ & \ecnflabel{2.2.5.1}{2401.1-b1}  \\ \hline
$-40$ & $1$ &  $\Q(\sqrt{5})$ & \texttt{5.1.ns-1.1.1} & $\mathcal{N}_{-10,\,0}(5^{\infty})$ &   $t$ &  $t\in \Q(\sqrt{5})^*$ \\ \hline
{$-43$} & {$1$}  & {$\QQ$}  
& \texttt{43.1.ns-1.1.1} & $\mathcal{N}_{-43/4,\,0}(43^{\infty})$ &   $t$ &  $t\in \Q^*,\,\,t\ne 1,-43 $ \\ \cline{4-7} 
 & &  & \texttt{43.1.ns-43.2.1} & $\langle G^{2,1}_{-43/4,\,0},c_{1}\rangle$ &   $1$ &  \eclabel{1849.b2} \\ \cline{4-7} 
 & &  & \texttt{43.1.ns-43.2.2} & $\langle G^{2,1}_{-43/4,\,0},c_{-1}\rangle$ &   $-43$ & \eclabel{1849.b1} \\ \hline
{$-51$} & {$1$} &  {$\Q(\sqrt{17})$} 
 & \texttt{3.1.s-1.1.1} & $\mathcal{N}_{-51/4,\,0}(3^{\infty})$ &   $t$ &  $t\in \Q(\sqrt{17})^*,\,\,t\ne 1,-3 $ \\ \cline{4-7} 
& &  & \texttt{3.1.s-3.2.1} & $\langle G^{\,2,1}_{-51/4,\,0},c_{1}\rangle$ &   $1$ &   \ecnflabel{2.2.17.1}{81.1-b1}\\ \cline{4-7} 
& &  & \texttt{3.1.s-3.2.2} & $\langle G^{\, 2,1}_{-51/4,\,0},c_{-1}\rangle$ &   $-3$ &   \ecnflabel{2.2.17.1}{81.1-b2} \\ \cline{4-7} 
& &  & \texttt{17.1.ns-1.1.1} & $\mathcal{N}_{-51/4,\,0}(17^{\infty})$ &   $t$ &  $t\in \Q(\sqrt{17})^*$ \\ \hline
$-52$ & $1$ & $\Q(\sqrt{13})$ & \texttt{13.1.s-1.1.1} & $\mathcal{N}_{-13,\,0}(13^{\infty})$ &   $t$ &  $t\in \Q(\sqrt{13})^*$ \\ \hline
{$-67$}  & {$1$}  & {$\QQ$}   
& \texttt{67.1.ns-1.1.1} & $\mathcal{N}_{-67/4,\,0}(67^{\infty})$ &   $t$ &  $t\in \Q^*,\,\,t\ne 1,-67 $ \\ \cline{4-7} 
 & &  & \texttt{67.1.ns-67.2.1} & $\langle G^{2,1}_{-67/4,\,0},c_{1}\rangle$ &   $1$ & \eclabel{4489.b2}  \\ \cline{4-7} 
 & &  & \texttt{67.1.ns-67.2.2} & $\langle G^{2,1}_{-67/4,\,0},c_{-1}\rangle$ &   $-67$ & \eclabel{4489.b1} \\ \hline
{$-88$} & {$1$} &  {$\Q(\sqrt{2})$} 
&  \texttt{11.1.s-1.1.1} & $\mathcal{N}_{-22,\,0}(11^{\infty})$ &   $t$ &  $t\in \Q(\sqrt{2})^*,\,\,t\ne 1,-11 $ \\ \cline{4-7} 
& & & \texttt{11.1.s-11.2.1} & $\langle G^{2,1}_{-22,\,0},c_{1}\rangle$ &   $1$ &  N/A \\ \cline{4-7} 
& & &  \texttt{11.1.s-11.2.2} & $\langle G^{2,1}_{-22,\,0},c_{-1}\rangle$ &   $-11$ &  N/A \\ \hline
{$-91$}&  {$1$} &  {$\Q(\sqrt{13})$} 
 & \texttt{7.1.s-1.1.1} & $\mathcal{N}_{-91/4,\,0}(7^{\infty})$ &   $t$ &  $t\in \Q(\sqrt{13})^*,\,\,t\ne 1,-7 $ \\ \cline{4-7} 
&  & & \texttt{7.1.s-7.2.1} & $\langle G^{\,2,1}_{-91/4,\,0},c_{1}\rangle$ &   $-7$ & N/A  \\ \cline{4-7} 
&  & & \texttt{7.1.s-7.2.2} & $\langle G^{\, 2,1}_{-91/4,\,0},c_{-1}\rangle$ &   $1$ &  N/A \\ \cline{4-7} 
&  & & \texttt{13.1.ns-1.1.1} & $\mathcal{N}_{-91/4,\,0}(13^{\infty})$ &   $t$ &  $t\in \Q(\sqrt{13})^*$ \\ \hline
{$-115$}  & {$1$} &  {$\Q(\sqrt{5})$} 
 & \texttt{5.1.ns-1.1.1} & $\mathcal{N}_{-115/4,\,0}(5^{\infty})$ &   $t$ &  $t\in \Q(\sqrt{5})^*$ \\ \cline{4-7} 
&  & & \texttt{23.1.s-1.1.1} & $\mathcal{N}_{-115/4,\,0}(23^{\infty})$ &   $t$ &  $t\in \Q(\sqrt{5})^*,\,\,t\ne 1,-23 $ \\ \cline{4-7} 
& & &  \texttt{23.1.s-23.2.1} & $\langle G^{\,2,1}_{-115/4,\,0},c_{1}\rangle$ &   $-23$ &  N/A \\ \cline{4-7} 
& &  & \texttt{23.1.s-23.2.2} & $\langle G^{\, 2,1}_{-115/4,\,0},c_{-1}\rangle$ &   $1$ &  N/A \\ \hline
{$-123$} & {$1$} &  {$\Q(\sqrt{41})$} 
 & \texttt{3.1.s-1.1.1} & $\mathcal{N}_{-123/4,\,0}(3^{\infty})$ &   $t$ &  $t\in \Q(\sqrt{41})^*,\,\,t\ne 1,-11 $ \\ \cline{4-7} 
& &  & \texttt{3.1.s-3.2.1} & $\langle G^{\,2,1}_{-123/4,\,0},c_{1}\rangle$ &   $1$ & \ecnflabel{2.2.41.1}{81.1-c1}   \\ \cline{4-7} 
& &  & \texttt{3.1.s-3.2.2} & $\langle G^{\, 2,1}_{-123/4,\,0},c_{-1}\rangle$ &   $-3$ & \ecnflabel{2.2.41.1}{81.1-c2}  \\ \cline{4-7} 
& &  & \texttt{41.1.ns-1.1.1} & $\mathcal{N}_{-123/4,\,0}(41^{\infty})$ &   $t$ &  $t\in \Q(\sqrt{41})^*$ \\ \hline
$-148$ & $1$ &  $\Q(\sqrt{37})$ &\texttt{37.1.s-1.1.1} & $\mathcal{N}_{-37,\,0}(37^{\infty})$ &   $t$ &  $t\in \Q(\sqrt{37})^*$ \\ \hline
{$-163$} & {$1$}  & {$\QQ$}   
& \texttt{163.1.ns-1.1.1} & $\mathcal{N}_{-163/4,\,0}(163^{\infty})$ &   $t$ &  $t\in \Q^*,\,\,t\ne 1,-163 $ \\ \cline{4-7} 
 & &  & \texttt{163.1.ns-163.2.1} & $\langle G^{2,1}_{-163/4,\,0},c_{1}\rangle$ &   $1$ &  \eclabel{26569.a2} \\ \cline{4-7} 
 & &  & \texttt{163.1.ns-163.2.2} & $\langle G^{2,1}_{-163/4,\,0},c_{-1}\rangle$ &   $-163$ & \eclabel{26569.a1}  \\ \hline
{$-187$} & {$1$} &  {$\Q(\sqrt{17})$} 
 & \texttt{11.1.s-1.1.1} & $\mathcal{N}_{-187/4,\,0}(11^{\infty})$ &   $t$ &  $t\in \Q(\sqrt{17})^*,\,\,t\ne 1,-11 $ \\ \cline{4-7} 
&  & & \texttt{11.1.s-11.2.1} & $\langle G^{\,2,1}_{-187/4,\,0},c_{1}\rangle$ &   $1$ &  N/A \\ \cline{4-7} 
&  & & \texttt{11.1.s-11.2.2} & $\langle G^{\, 2,1}_{-187/4,\,0},c_{-1}\rangle$ &   $-11$ &   N/A \\ \cline{4-7} 
&  & & \texttt{17.1.ns-1.1.1} & $\mathcal{N}_{-187/4,\,0}(17^{\infty})$ &   $t$ &  $t\in \Q(\sqrt{17})^*$ \\ \hline
$-232$ &$1$ & $\Q(\sqrt{29})$ & \texttt{29.1.ns-1.1.1} & $\mathcal{N}_{-58,\,0}(29^{\infty})$ &   $t$ &  $t\in \Q(\sqrt{29})^*$ \\ \hline
{$-235$}&   {$1$} &  {$\Q(\sqrt{5})$} 
 & \texttt{5.1.ns-1.1.1} & $\mathcal{N}_{-235/4,\,0}(5^{\infty})$ &   $t$ &  $t\in \Q(\sqrt{5})^*$ \\ \cline{4-7} 
&  & & \texttt{47.1.s-1.1.1} & $\mathcal{N}_{-235/4,\,0}(47^{\infty})$ &   $t$ &  $t\in \Q(\sqrt{5})^*,\,\,t\ne 1,-47 $ \\ \cline{4-7} 
&  & & \texttt{47.1.s-47.2.1} & $\langle G^{\,2,1}_{-235/4,\,0},c_{1}\rangle$ &   $-47$ &  N/A \\ \cline{4-7} 
&  & & \texttt{47.1.s-47.2.2} & $\langle G^{\, 2,1}_{-235/4,\,0},c_{-1}\rangle$ &   $1$ &  N/A \\ \hline
{$-267$} & {$1$} &  {$\Q(\sqrt{89})$} 
 & \texttt{3.1.s-1.1.1} & $\mathcal{N}_{-267/4,\,0}(3^{\infty})$ &   $t$ &  $t\in \Q(\sqrt{89})^*,\,\,t\ne 1,-3 $ \\ \cline{4-7} 
& &  & \texttt{3.1.s-3.2.1} & $\langle G^{\,2,1}_{-267/4,\,0},c_{1}\rangle$ &   $1$ & \ecnflabel{2.2.89.1}{81.1-a1}  \\ \cline{4-7} 
& &  & \texttt{3.1.s-3.2.2} & $\langle G^{\, 2,1}_{-267/4,\,0},c_{-1}\rangle$ &   $-3$ &  \ecnflabel{2.2.89.1}{81.1-a2} \\ \cline{4-7} 
& &  & \texttt{89.1.ns-1.1.1} & $\mathcal{N}_{-267/4,\,0}(89^{\infty})$ &   $t$ &  $t\in \Q(\sqrt{89})^*$ \\ \hline
{$-403$} &  {$1$} &  {$\Q(\sqrt{13})$} 
 & \texttt{13.1.ns-1.1.1} & $\mathcal{N}_{-403/4,\,0}(13^{\infty})$ &   $t$ &  $t\in \Q(\sqrt{13})^*$ \\ \cline{4-7} 
&  & & \texttt{31.1.s-1.1.1} & $\mathcal{N}_{-403/4,\,0}(31^{\infty})$ &   $t$ &  $t\in \Q(\sqrt{13})^*,\,\,t\ne 1,-31 $ \\ \cline{4-7} 
& &  & \texttt{31.1.s-31.2.1} & $\langle G^{\,2,1}_{-403/4,\,0},c_{1}\rangle$ &   $-31$ &  N/A \\ \cline{4-7} 
& &  & \texttt{31.1.s-31.2.2} & $\langle G^{\, 2,1}_{-403/4,\,0},c_{-1}\rangle$ &   $1$ &  N/A \\ \hline
{$-427$} & {$1$} &  {$\Q(\sqrt{61})$} 
 & \texttt{7.1.s-1.1.1} & $\mathcal{N}_{-427/4,\,0}(7^{\infty})$ &   $t$ &  $t\in \Q(\sqrt{61})^*,\,\,t\ne 1,-7 $ \\ \cline{4-7} 
& &  & \texttt{7.1.s-7.2.1} & $\langle G^{\,2,1}_{-427/4,\,0},c_{1}\rangle$ &   $-7$ &  N/A \\ \cline{4-7} 
& &  & \texttt{7.1.s-7.2.2} & $\langle G^{\, 2,1}_{-427/4,\,0},c_{-1}\rangle$ &   $1$ &  N/A \\ \cline{4-7} 
&  & & \texttt{61.1.ns-1.1.1} & $\mathcal{N}_{-427/4,\,0}(61^{\infty})$ &   $t$ &  $t\in \Q(\sqrt{61})^*$ \\ \hline

\end{longtable}
}

\

\bibliography{bibliography}
\bibliographystyle{plain}
\end{document}